\documentclass[11pt]{article}
\usepackage{amssymb,amsmath,setspace,graphicx,multicol,wrapfig,amsfonts,amsthm,titling,url,array,parskip,enumitem,bbm,color,tikz,tikz-cd,caption,mathdots}
\usetikzlibrary{patterns}
\usepackage{mathtools}
\usepackage{placeins,comment}
\usepackage[colorlinks=true, pdfstartview=FitV, linkcolor=blue, citecolor=blue, urlcolor=blue]{hyperref}
\usepackage[capitalize,nameinlink,noabbrev,nosort]{cleveref}
\usepackage{aurical,pbsi}
\usepackage[centertableaux]{ytableau}
\usepackage[T1]{fontenc}
\usepackage[hmargin=2.5cm,vmargin=2.5cm]{geometry}
\usepackage{todonotes}
\usepackage{cancel}

\ytableausetup{boxsize=1em} 

\makeatletter
\newtheorem*{rep@theorem}{\rep@title}
\newcommand{\newreptheorem}[2]{%
\newenvironment{rep#1}[1]{%
 \def\rep@title{#2 \ref{##1}}%
 \begin{rep@theorem}}%
 {\end{rep@theorem}}}
\makeatother

\usepackage{queue}

\usepackage[backend=biber, 
maxbibnames=99, maxalphanames=5, giveninits=true, url=false, isbn=false, doi=false]{biblatex}
\graphicspath{ {./figs/} }

\usepackage[all]{nowidow} 
\usepackage[protrusion=true,expansion=true]{microtype} 

\usepackage[usenames,dvipsnames]{xcolor}

\usepackage[mathscr]{euscript}

\newcommand{\MLQ}[1][]{\text{MLQ}\ifx&#1&\else(#1)\fi}
\newcommand{\NMLQ}[1][]{\text{NMLQ}\ifx&#1&\else(#1)\fi}
\newcommand{\SMLQ}[1][]{\text{MLQ}^+\ifx&#1&\else(#1)\fi}
\newcommand{\SNMLQ}[1][]{\text{NMLQ}^+\ifx&#1&\else(#1)\fi}

\numberwithin{equation}{section}

\newcommand{\sbf}{\mathbf{u}}

\usepackage{thmtools}
\theoremstyle{plain}
\newtheorem{theorem}{Theorem}[section]
\newreptheorem{theorem}{Theorem}
\newtheorem{prop}[theorem]{Proposition}
\newtheorem{lemma}[theorem]{Lemma}
\newtheorem{cor}[theorem]{Corollary}

\theoremstyle{definition}
\newtheorem{definition}[theorem]{Definition}
\newtheorem{example}[theorem]{Example}
\newtheorem{question}[theorem]{Question}
\newtheorem{remark}[theorem]{Remark}
\crefname{definition}{Definition}{Definitions}
\crefname{theorem}{Theorem}{Theorems}
\crefname{lemma}{Lemma}{Lemmas}
\crefname{example}{Example}{Examples}

\newcommand\id{\text{id}}

\newcommand{\type}{\textup{\texttt{type}}}
\newcommand{\strtype}{\textup{\texttt{strtype}}}
\newcommand\act{\textup{\texttt{act}}}

\newcommand\sort{\texttt{sort}}
\newcommand\rev{\texttt{rev}}
\newcommand{\maj}{\textup{\texttt{maj}}}

\newcommand{\inc}{\texttt{inc}}

\newcommand{\dg}{\texttt{dg}}
\newcommand{\charge}{\textup{\texttt{charge}}}
\newcommand\QS{\text{QS}}
\newcommand\Sym{\text{Sym}}
\newcommand\QSym{\text{QSym}}
\newcommand\SSYT{\text{SSYT}}
\newcommand\SSAF{\text{SSAF}}
\newcommand\SSCT{\text{SSCT}}

\newcommand{\SSQT}{SSQT}

\newcommand{\leg}{leg}
\newcommand{\South}{South}

\newcommand{\rw}{rw}
\newcommand{\cw}{cw}
\newcommand\cA{\mathcal{A}}

\newcommand\ZZ{\mathbb{Z}}
\newcommand\QQ{\mathbb{Q}}
\newcommand\NN{\mathbb{N}}

\newcommand\Wc{\mathcal{W}}

\newcommand{\eL}{e^\leftarrow}
\newcommand{\eLs}{e^{\leftarrow\star}}
\newcommand{\eD}{e^\downarrow}
\newcommand{\eDs}{e^{\downarrow\star}}
\newcommand{\fRs}{f^{\uparrow\star}}
\newcommand{\fR}{f^\rightarrow}
\newcommand{\fU}{f^\uparrow}

\usepackage[dvipsnames]{xcolor}
\newcommand{\highlight}{Goldenrod!93!black}

\newlength\cellsize 
\savebox2{%
\begin{picture}(12,12)
\put(0,0){\line(1,0){12}}
\put(0,0){\line(0,1){12}}
\put(12,0){\line(0,1){12}}
\put(0,12){\line(1,0){12}}
\end{picture}}

\newcommand\cellify[1]{\def\thearg{#1}\def\nothing{}%
\ifx\thearg\nothing
\vrule width0pt height\cellsize depth0pt\else
\hbox to 0pt{\usebox2\hss}\fi%
\vbox to 12\unitlength{
\vss
\hbox to 12\unitlength{\hss$#1$\hss}
\vss}}

\newcommand\tableau[1]{\vtop{\let\\=\cr
\setlength\baselineskip{-16000pt}
\setlength\lineskiplimit{16000pt}
\setlength\lineskip{0pt}
\halign{&\cellify{##}\cr#1\crcr}}}

\savebox3{%
\begin{picture}(12,12)
\put(0,0){\line(1,0){12}}
\put(0,0){\line(0,1){12}}
\put(12,0){\line(0,1){12}}
\put(0,12){\line(1,0){12}}
\end{picture}}

\newcommand\expath[1]{%
\hbox to 0pt{\usebox3\hss}%
\vbox to 12\unitlength{
\vss
\hbox to 12\unitlength{\hss$#1$\hss}
\vss}}

\newcommand\cell[3]{
\def\i{#1} \def\j{#2} \def\entry{#3}
\draw (\j-1,-\i)--(\j,-\i)--(\j,-\i+1)--(\j-1,-\i+1)--(\j-1,-\i);
\node at (\j-.5,-\i+.5) {\entry};
}

\newcommand{\tcr}[1]{\textcolor{red}{#1}}

\hypersetup{ 	
pdfsubject = {},
pdftitle = {Multiline queue stuff},
pdfauthor = {Harper Niergarth}
}

\title{Positive expansions of permuted basement and quasisymmetric Macdonald polynomials at $t=0$}
\author{Olya Mandelshtam, Harper Niergarth, and Kartik Singh}

\begin{document} 

\maketitle

\abstract{ It is well known that the $q$-Whittaker polynomials, which are $t=0$ specializations of the Macdonald polynomials $P_\lambda(X;q,t)$, expand positively as the sum of Schur polynomials. Macdonald polynomials have a quasisymmetric refinement: the quasisymmetric Macdonald polynomials $G_\gamma(X;q,t)$, and a nonsymmetric refinement: the ASEP polynomials $f_\alpha(X;q,t)$. We study the $t=0$ specializations of both these families of polynomials and show analogous properties: the quasisymmetric Macdonald polynomials expand positively as a sum of quasisymmetric Schur functions, $\QS_\gamma(X)$, and the ASEP polynomials expand positively as a sum of Demazure atoms, $\cA_\alpha(X)$. As a corollary of the latter, we prove more generally that any permuted basement Macdonald polynomial has a positive expansion in the Demazure atoms at $t=0$. We give a description of the structure coefficients of $G_\gamma(X;q,0)$ and $f_\alpha(X;q,0)$ in both cases in terms of the charge statistic on a restricted set of semistandard tableaux.}

\setcounter{tocdepth}{1}
\tableofcontents

\section{Introduction}

Macdonald \cite{Mac88} defined the \textbf{symmetric Macdonald polynomials} $P_\lambda(X;q,t)$  in the ring of symmetric functions $\Sym$ in the commuting variables $X=x_1,x_2,\ldots$ with coefficients in $\QQ(q,t)$, as the unique family of functions indexed by partitions $\lambda$, that satisfy certain triangularity conditions with respect to the monomial basis, and orthogonality conditions with respect to a certain $(q,t)$-deformation of the Hall inner product. At $q=t=0$, these conditions reduce to those which characterize the Schur polynomials $s_\lambda(X)$. Other specializations of the parameters recover various well-studied families of symmetric functions such as the Hall--Littlewood polynomials ($q=0$), the Jack polynomials $(q=t^\alpha, t\to 1)$, and the $q$-Whittaker polynomials ($t=0$), which are the main objects of study in this paper.

In general, the coefficient of a given monomial in $P_\lambda(X;q,t)$ is a rational function in the parameters $q,t$. However, for certain specializations of $q$ and $t$, the coefficients simplify greatly, in particular for the case $t=0$. In their celebrated paper \cite{LS78}, Lascoux and Sch\"utzenberger gave the explicit Schur expansion of $P_\lambda(X;q,0)$ in terms of a $\charge$ statistic on the set of semistandard Young tableaux of certain shape and content:
\begin{equation}\label{eq:P charge}
P_\lambda(X;q,0)=\sum_{\mu\leq \lambda} s_{\mu} \sum_{T\in\SSYT(\mu',\lambda')}q^{\charge(T)}.
\end{equation}

There is a similar rich story for the entire polynomial ring $\QQ(q,t)[X]$, which admits an analogous $(q,t)$-inner product, with the orthogonal basis given by the family of \textbf{nonsymmetric Macdonald polynomials} $E_\alpha(X;q,t)$ indexed by compositions $\alpha$. These were constructed by Macdonald and Cherednik \cite{Mac96,Che95} and further studied by Opdam \cite{Opd95} to help understand the symmetric Macdonald polynomials: $P_\lambda(X;q,t)$ is the unique monic symmetric function that can be obtained as a linear sum of the $E_\alpha$'s over all permutations $\alpha$ of $\lambda$. At $q=t=0$, the nonsymmetric analogues of the Schur polynomials are the \textbf{key polynomials} $\kappa_\alpha$(X), which are obtained as the \textbf{Demazure characters} in type $A$. A further decomposition of the key polynomials into their smallest non-intersecting pieces yields the \textbf{Demazure atoms} $\cA_\alpha(X)$, which were first studied by Lascoux and Sch\"utzenberger \cite{LS90}.

Mason \cite{Mas09} showed that the Demazure atoms correspond to the $q=t=\infty$ specialization of $E_\alpha(X;q,t)$ and gave the description of the combinatorial objects generating them. Notably, at $t=0$, there is an analogous positivity result for the structure coefficients of the key- and atom- decompositions in this setting. Assaf proved the positive expansion of $E_\alpha$ in the key polynomial basis \cite{Ass17}, and later with Gonz\'alez gave the following explicit formula: 
\[E_\alpha(X;q,0)=\sum_{\beta\leq\alpha} \kappa_{\beta}(X)\sum_{T} q^{\maj(T)},
\]
where the inner sum is over a certain set of tableaux determined by $\alpha$ \cite{AG19}. The above implies positivity of $E_\alpha(X;q,0)$ in the Demazure atom expansion, as well.

Between the rings of symmetric functions $\Sym$ and ordinary polynomials $\mathbb{F}[X]$ (over a fixed coefficient ring $\mathbb{F}$, which will be $\QQ$, $\QQ(q,t)$, or $\QQ[q]$ in our case) lies the ring of quasisymmetric functions $\QSym$. The bases for $\QSym$ are indexed by \textbf{strong compositions}, which are compositions with only positive parts. We can relate bases for the three rings $\Sym \subset \QSym \subset \mathbb{F}[X]$ using refinements. Let $\{A_\lambda\}$, $\{B_\gamma\}$, and $\{C_\alpha\}$ be bases for $\Sym$, $\QSym$, and $\mathbb{F}[X]$, respectively, where $\lambda$, $\gamma$, and $\alpha$ vary over partitions, strong compositions and (weak) compositions. We say $\{B_\gamma\}$ is a \textbf{quasisymmetric refinement} of $\{A_\lambda\}$ if the following is true for all $\lambda$:
\begin{equation*}
    A_\lambda = \sum_{\gamma: \sort(\gamma)= \lambda}B_\gamma.
\end{equation*}
where $\sort(\gamma)$ is the weakly decreasing rearrangement of $\gamma$. 
Similarly, we say that $\{C_\alpha\}$ is a \textbf{nonsymmetric refinement} of $\{A_\lambda\}$ (or equivalently of $\{B_\gamma\}$) if the following equations hold true for all $\lambda$ (or $\gamma$), respectively:
\[
    A_\lambda = \sum_{\alpha:\sort(\alpha)=\lambda}C_\alpha\qquad \text{and}\qquad 
    B_\gamma = \sum_{\alpha:\alpha^+=\gamma}C_\alpha,
\]
where $\alpha^+$ is the strong composition obtained as the \textbf{compression} of $\alpha$ by removing all zero-parts of $\alpha$. We call the triple of bases $(A_\alpha,B_\gamma,C_\alpha)$ a \emph{triple of refinements}. 

For example, a triple of refinements for $\Sym$, $\QSym$, and $\mathbb{Q}[X]$ is given by the standard monomial basis $\{m_\lambda(X)\}$ indexed by partitions $\lambda$, the quasisymmetric monomial basis $\{M_\gamma(X)\}$ indexed by strong compositions $\gamma$, and the set of monomials $\{x^\alpha\}$ indexed by compositions $\alpha$. Indeed, by definition, the basis $\{x^\alpha\}$ is the nonsymmetric refinement of $\{M_\gamma(X)\}$, which in turn is the quasisymmetric refinement of $\{m_\lambda(X)\}$. For another example, another triple of refinements for the same triple of rings is given by the Schur functions $\{s_\lambda(X)\}$, the basis of \textbf{quasisymmetric Schur polynomials} $\{\QS_\gamma(X)\}$ introduced by Haglund, Luoto, Mason and van Willigenburg in \cite{HLMvW11}, and the  basis of Demazure atoms $\{\cA_\alpha(X)\}$.

The Macdonald polynomial $P_\lambda(X;q,t)$ decomposes as a sum of nonsymmetric Macdonald polynomials indexed by the set $\{\alpha:\sort(\alpha)=\lambda\}$; however, the nonsymmetric Macdonald polynomials $\{E_\alpha\}$ are \emph{not} the nonsymmetric refinement of $P_\lambda$ according to our definition. Ferreira \cite{Fer11} and Alexandersson \cite{Ale19} introduced the combinatorial study of the \textbf{permuted basement nonsymmetric Macdonald polynomials} $E_\alpha^\sigma(X;q,t)$, a generalization of $E_\alpha(X;q,t)$, additionally indexed by a permutation $\sigma$ of the parts of $\alpha$.  In terms of these polynomials, the correct refinement of $P_\lambda(X;q,t)$ is given by the set of polynomials $\{E_{\inc(\lambda)}^\sigma(X;q,t)\}$ ranging over permutations $\sigma$, where $\inc(\lambda)$ is the weakly increasing rearrangement of the parts of $\lambda$ \cite{CMW22}. This family of permuted basement polynomials was named the \textbf{ASEP polynomials} due to their connection to the stationary distribution of the asymmetric simple exclusion process (ASEP) \cite{CdGW15}. For a composition $\alpha$, define $f_\alpha(X;q,t)\coloneq E_{\inc(\alpha)}^\sigma(X;q,t)$, where $\sigma$ is a permutation that rearranges $\alpha$ into $\inc(\alpha)$. Then, one obtains a nonsymmetric refinement of $P_\lambda(X;q,t)$ in terms of the $\alpha$'s:  
\begin{equation}\label{eq:P in F}
P_\lambda(X_n;q,t)=\sum_{\sort(\alpha)=\lambda} f_\alpha(X_n;q,t).
\end{equation}
Note that $P_\lambda(X;q,t)$ can be obtained other sums of permuted basement polynomials, but this is the only such sum in which all components appear with the coefficient 1.

The ASEP polynomials satisfy $f_\alpha(X;0,0)=\cA_\alpha(X)$. The connection with the ASEP and the parallel with the $q=t=0$ setting led the first author, Corteel, Haglund, Mason, and Williams to define the \textbf{quasisymmetric Macdonald polynomial}, $G_\gamma(X;q,t)$ as a sum over the set of ASEP polynomials that compress to $\gamma$:
\begin{definition}[{\cite{CHMMW22}}]
For a strong composition $\gamma$, define    \[
G_{\gamma}(X;q,t)=\sum_{\alpha:\alpha^+=\gamma}f_\alpha(X;q,t).
\]
\end{definition}
By construction, the $G_\gamma$'s are quasisymmetric and refine $P_\lambda$. Thus we obtain $(P_\lambda, G_\gamma, f_\alpha)$ as a triple of refinements, which specializes at $q=t=0$ to the triple $(s_\lambda,\QS_\gamma,\cA_\alpha)$. We complete the following picture, where the rightward arrows represent quasisymmetric and nonsymmetric refinements, respectively, and the downward arrows represent the $q=t=0$ specialization: 

\begin{center}
\begin{tikzcd}[sep=normal,ampersand replacement=\&]
\Sym\&\&\&\QSym\&\&\& \mathbb{Q}(q,t)[X]\\
P_{\lambda}(X;q,t)  \arrow[rrr,"",,] 
\arrow[d,"" left] \&\&\&G_\gamma(X;q,t) \arrow[rrr,"",,] 
\arrow[d,"" left,] \& \& \& f_\alpha(X;q,t) \arrow[d,"",,] \\
s_{\lambda}(X) \arrow[rrr,"",,]  
\&\&\&\QS_\gamma(X) \arrow[rrr,"",,] 
\&\&\& \mathcal{A}_{\alpha}(X)
\end{tikzcd}
\end{center}

As mentioned earlier, the $t=0$ specializations of both the symmetric and nonsymmetric Macdonald polynomials (henceforth referred to as $q$-Whittaker polynomials) decompose positively as sums of their $q=t=0$ specializations, with the coefficients given by the corresponding \textbf{Kostka--Foulkes polynomials}. In both cases, formulas for the Kostka--Foulkes polynomials can be obtained as generating functions over certain restricted sets,  with a $\charge$ statistic for $P_\lambda(X;q,0)$  due to \cite{LS78} and with a $\maj$ statistic for $E_\alpha(X;q,0)$ due to \cite{AG19}. 

In this paper, we first extend the result of \cite{AG19} to prove that the ASEP polynomials $f_\alpha(X;q,0)$ expand positively in the Demazure atoms. To be precise, we show the following: 

\begin{reptheorem}{theorem:weakmain}
For a composition $\alpha$, the ASEP polynomial is given by
  \begin{equation}\label{eq:nsym}
f_{\alpha}(X;q,0) = \sum_{\beta:\sort(\beta)\leq \sort(\alpha)} K_{\alpha,\beta}(q) \cA_{\beta},
\end{equation}
where
\begin{equation}\label{eq:K}
K_{\alpha,\beta}(q) = \sum_{\substack{Q\in\SSYT(\sort(\beta)',\sort(\alpha)')\\\type(\rho^{-1}(M_{\beta},Q))=\alpha}} q^{\charge(Q)}.
\end{equation}
where $M_{\beta}\in\NMLQ(\beta)$ is the unique multiline queue with 
type and content $\beta$, and $\rho$ is the collapsing map (see \cref{sec:crystalops}).
\end{reptheorem}

Using this, we also establish the quasisymmetric analogue: we show that the coefficients in the decomposition of a quasisymmetric $q$-Whittaker polynomial into quasisymmetric Schur functions are polynomials in $q$ with nonnegative integer coefficients.

\begin{reptheorem}{theorem:main}
For a strong composition $\gamma$, the quasisymmetric Macdonald polynomial is given by
  \begin{equation}\label{eq:quasi}
G_{\gamma}(X;q,0) = \sum_{\tau:\sort(\tau)\leq \sort(\gamma)} K_{\gamma,\tau}(q) \QS_{\tau},
\end{equation}
where
\begin{equation}\label{eq:K}
K_{\gamma,\tau}(q) = \sum_{\substack{Q\in\SSYT(\sort(\tau)',\sort(\gamma)')\\\strtype(\rho^{-1}(M_{\tau},Q))=\gamma}} q^{\charge(Q)}.
\end{equation}
where $\strtype$ is the strong type of a multiline queue (see \cref{sec:MLQ}).
\end{reptheorem}
We refer to the coefficients $K_{\alpha,\beta}(q)$ and $K_{\gamma,\tau}(q)$ in the nonsymmetric and quasisymmetric results as the \textbf{nonsymmetric} and \textbf{quasisymmetric Kostka--Foulkes polynomials}, respectively, and give descriptions of them as generating functions over restricted sets of semistandard tableaux with a $\charge$ statistic.

In summary, we obtain the following parallel expansions:
\begin{align}
    P_\lambda(X;q,0)&=\sum_{\mu\leq\lambda} K_{\lambda\mu}(q)s_\mu(X)\,,&K_{\lambda\mu}(q)=\sum_{T\in \mathcal{A}_{\lambda\mu}}q^{\charge(T)}\label{eq:P lambda}\\
    G_\gamma(X;q,0)&=\sum_{\tau\leq\gamma} K_{\gamma\tau}(q)\QS_\tau(X)\,,&K_{\gamma\tau}(q)=\sum_{T\in \mathcal{B}_{\gamma\tau}}q^{\charge(T)}\label{eq:G gamma}\\
    f_\alpha(X;q,0)&=\sum_{\beta\leq \alpha} K_{\alpha\beta}(q)\cA_\beta(X)\,,&K_{\alpha\beta}(q)=\sum_{T\in \mathcal{C_{\alpha\beta}}}q^{\charge(T)}\label{eq:f alpha}
\end{align}
where $\mathcal{A}_{\lambda\mu}$, $\mathcal{B}_{\gamma\tau}$, and $\mathcal{C}_{\alpha\beta}$ are sets of semistandard tableaux, and sums are over partitions, strong compositions, and weak compositions, respectively, with $\leq$ representing the associated dominance orders.

Finally, as a corollary of \eqref{eq:f alpha}, we prove that any permuted basement Macdonald polynomial $E_\alpha^\sigma(X;q,0)$ expands positively into Demazure atoms.

\subsection{Proof sketch}
We prove our results by studying the \emph{graded crystal} structure on \emph{multiline queues}. One can consider multiline queues as binary matrices endowed with additional data, that associates to each multiline queue a composition, called its \emph{(weak) type}; the corresponding strong composition is called its \textbf{strong type}. 

Multiline queues arise from probability theory, where they were introduced by Ferrari and Martin to compute the stationary distribution of the ASEP at $t=0$ \cite{FM07}. In \cite{CMW22}, multiline queues were enhanced with additional statistics, and the set of multiline queues of size $\lambda$ was shown to give a combinatorial formula for the $q$-Whittaker polynomial $P_\lambda(X;q,0)$.\footnote{In the $t=0$ case, an equivalent formula  has long been known: when the multiline queues are interpreted as Kirillov--Reshetikhin crystals, the formula for $P_\lambda(X;q,0)$ can be expressed in terms of the classical \emph{charge} or \emph{energy} statistics \cite{NY97}.} When refined by (weak) type or strong type, these weighted multiline queues also yield formulas for $f_\alpha(X;q,0)$ and $G_\gamma(X;q,0)$, respectively. 

Our approach is motivated by the work of the first author with Valencia-Porras \cite{MV24}, where the Lascoux--Sch\"utzenberger charge formula was proved by defining a \emph{collapsing operator} on multiline queues. From a crystal-theoretic perspective, we can view multiline queues as type $A^{(1)}$ Kirillov--Reshetikhin crystals equipped with the energy grading, recorded by a power of $q$. In this framework, Schur functions are generated by multiline queues in the component with energy 0, referred to as \emph{nonwrapping multiline queues}. These are isolated from the full crystal by specializing $q=0$, and we therefore refer to them as the $q=0$ component. In \cite{MV24}, a second family of crystal operators was introduced, commuting with the original ones, and the collapsing operator was expressed as a composition of these. The collapsing operator gives a bijection between multiline queues of arbitrary energy and pairs consisting of a nonwrapping multiline queue and a recording tableau. Moreover, for a fixed recording tableau, the preimage of a Schur function forms a subgraph of a connected component in the original crystal at fixed energy. 

A similar crystal-theoretic argument was made by Assaf and Gonz\'alez \cite{AG19} to establish positivity of the expansion of $E_\alpha(X;q,0)$ in key polynomials, discussed in \cref{sec:comparison}. In the present work, where we prove Demazure atom and quasisymmetric Schur positivity results, we refine this argument to keep track of the type corresponding to each multiline queue. We consider the subgraph of the multiline queue crystal consisting of all multiline queues of a fixed type $\alpha$. Then, we restrict the collapsing operator to this set, sending each multiline queue in that component to a pair consisting of a representative of a Demazure atom and a recording tableau.

The crux of our proof lies in showing that, for a fixed recording tableau, the preimage of the subcrystal corresponding to a Demazure atom $\cA_\alpha(X)$ is entirely contained within a single nonsymmetric subcrystal in the original crystal at fixed energy. A crucial difference between this and the Schur and key polynomial cases is that in our setting, the $q=0$ subcrystal associated with a single Demazure atom, i.e.~for a fixed type, may be disconnected (see \cref{fig:big figure}). This makes the proof of \cref{theorem:weakmain} significantly more delicate. Nevertheless, we show that the recursive structure of multiline queues allows us to move between disconnected components of the subcrystal corresponding to a fixed type in a controlled way. This ultimately allows us to prove our results, and furthermore suggests a general strategy for proving positivity in Demazure atom and quasisymmetric Schur expansions in other settings.

This article is organized as follows. \cref{sec:background} defines the relevant symmetric, quasisymmetric, and nonsymmetric functions and introduces their combinatorics in terms of multiline queues. In \cref{sec:crystals} we discuss the action of crystal operators on multiline queues and their type. We use results from that section to prove our main results \cref{theorem:weakmain} and \cref{theorem:main} in \cref{sec:induction}, and we prove Demazure positivity of any permuted basement Macdonald polynomials as a corollary in \cref{sec:permuted}. We discuss connections to tableau formulas in \cref{sec:comparison}.

\subsection*{Acknowledgements}
We would like to thank Jer\'onimo Valencia Porras, Zeus Dantas e Moura, Oliver Pechenik, Sylvie Corteel, Sarah Mason, and Per Alexandersson for helpful discussions. All three authors were partially supported by NSERC grant RGPIN-2021-02568.

\section{Preliminaries}\label{sec:background}

\subsection{Symmetric and quasisymmetric functions}

Let $\Sym$ be the ring of symmetric functions in the variables $X = x_1,x_2,\ldots$. Bases of $\Sym$ are indexed by \textbf{partitions} $\lambda = (\lambda_1,\ldots,\lambda_\ell)$ which are finite weakly decreasing sequences of positive integers. Each $\lambda_i$ is called a \textbf{part} of $\lambda$, the number of parts $\ell = \ell(\lambda)$ is called the \textbf{length} of $\lambda$, the sum of parts is called the \textbf{size} of $\lambda$. For two partitions $\lambda,\mu$, we say $\lambda$ \textbf{dominates} $\mu$ and write $\lambda \geq \mu$ if $\lambda_1 + \cdots + \lambda_i \geq \mu_1 + \cdots + \mu_i$ for all $i$.

A \textbf{(weak) composition} $\alpha = (\alpha_1,\ldots,\alpha_\ell)$ is a finite sequence of nonnegative integers, and the \textbf{rearrangement} $\sort(\alpha)$ of $\alpha$ is the partition obtained by sorting the positive parts of $\alpha$ in weakly decreasing order. A \textbf{strong composition} $\gamma$ is a finite sequence of positive integers. For a weak composition $\alpha$, its \textbf{compression} $\alpha^+$ is the strong composition obtained by deleting the 0's. We will use the word ``composition'' to mean weak composition by default. See \cref{ex:partitions} for an example. 

For a composition $\alpha=(\alpha_1,\ldots,\alpha_\ell)$, define the monomial $x^\alpha=x_1^{\alpha_1}\cdots x_\ell^{\alpha_\ell}$. A function in the variables $X$ is \textbf{symmetric} if the coefficient of $x^\alpha$ matches the coefficient of $x^\beta$ for any two compositions $\alpha$, $\beta$ such that $\sort(\alpha)=\sort(\beta)$.  Then the basis of \textbf{monomial symmetric functions}, denoted $m_\lambda(X)$, is given by 
\begin{align*}
    m_\lambda(X) = \sum_{\alpha: \sort(\alpha) = \lambda} x^\alpha.
\end{align*}

Each partition $\lambda$ corresponds to a \textbf{Young diagram}, which consists of $\lambda_i$ left justified boxes on the $i$th row. We use the convention that the rows are labelled from the bottom to top and columns from left to right, consistent with the French notation. The \textbf{conjugate partition} $\lambda'$ of a partition $\lambda$ is obtained by reflecting the diagram of $\lambda$ along its main diagonal. 

\begin{example}\label{ex:partitions}
For example, the composition $\alpha=(0,4,1,0,3,0)$ compresses to $\alpha^+=(4,1,3)$ and sorts to $\lambda =\sort(\alpha)= (4,3,1)$ with $\ell(\lambda)=3$. The conjugate partition is $\lambda'=(3,2,2,1)$, and the diagrams  of $\lambda$, and $\lambda'$ are below:
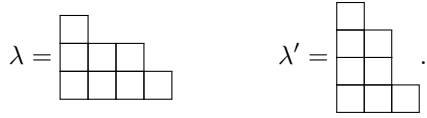
\begin{align*}
       \lambda = \resizebox{1.5cm}{!}{\ydiagram{1,3,4}} \qquad\qquad
        \lambda' = \ydiagram{1,2,2,3}.
\end{align*}
\end{example}

Let $\QSym$ be the ring of quasisymmetric functions in the variables $X = x_1,x_2,\ldots$. A function in the variables $X$ is \textbf{quasisymmetric} if the coefficient of $x^\alpha$ matches the coefficient of $x^\beta$ for any two compositions $\alpha$, $\beta$ such that $\alpha^+=\beta^+$. 
Clearly, all symmetric functions are quasisymmetric functions: $\Sym\subset \QSym$. Bases of this space are indexed by \textbf{strong compositions}. 
The basis of \textbf{monomial quasisymmetric functions} are defined as
    \begin{align*}
        M_\gamma(X) = \sum_{i_1 < \cdots < i_k} x_{i_1}^{\gamma_1}x_{i_2}^{\gamma_2} \cdots x_{i_k}^{\gamma_k}=\sum_{\alpha: \alpha^+=\gamma} x^\alpha.
    \end{align*}
By definition, these are a quasisymmetric refinement of the monomial symmetric functions, that is:
    \begin{align}\label{eq:monomial refinment}
        m_\lambda(X) = \sum_{\gamma: \sort(\gamma) = \lambda} M_\gamma (X).
    \end{align}
Finally, taking the set of monomials $\{x^\alpha\}$ ranging over all compositions $\alpha$ with finite support, we get a basis for the polynomial ring $\QQ[X]$, which functions as the nonsymmetric refinement for the $\{m_\lambda(X)\}$ and $\{M_\gamma(X)\}$ bases by construction.

\subsection{Multiline queues}\label{sec:MLQ}

\begin{definition}
    Let $\lambda$ be a partition, $L = \lambda_1$ and $n > \ell(\lambda)$ be an integer. A \textbf{multiline queue} of shape $(\lambda,n)$ is an arrangement of balls on an $L \times n$ array with rows numbered $1$ through $L$ from bottom to top, such that row $j$ contains $\lambda'_j$ balls. Columns are numbered $1$ through $n$ from left to right periodically modulo $n$, so that $j$ and $j + n$ correspond to the same column number. The site $(r, j)$ of $M$ refers to the cell in column $j$ of row $r$ of $M$. A multiline queue can be represented as a tuple $M = (B_1,\ldots, B_L)$ of $L$ subsets of $[n]:=\{1,\ldots, n\}$ where for each $j$, $|B_j|=\lambda'_j$, with the elements of $B_j$ corresponding to the sites containing balls in row $j$. Denote the set of multiline queues of shape $(\lambda,n)$ by $\MLQ_\lambda$. Throughout this article, we regard $n$ as fixed, and thus will omit it from the notation.
\end{definition}

The \textbf{Ferrari–-Martin (FM) algorithm} is a labeling procedure that deterministically assigns a label to each ball in a multiline queue $M$, producing a labeled multiline queue $L_M$. It is commonly described as a queuing process, in which, for each row $i$,
balls are paired between row $i$ and row $i - 1$, one at a time, according to a certain priority order. This form of the procedure produces the labeling $L_M$ in addition to a set of pairings between balls with the same label in adjacent rows. Note that while the set of pairings depends on the pairing order and may not be unique (see \cref{rem:not unique}), the labeling $L_M$ is uniquely determined.

\begin{definition}[FM algorithm]\label{def:FM} 
Let $M = (B_1,\ldots, B_L)\in\MLQ_\lambda$. For each row $r = L,\ldots, 2$:
    \begin{itemize}
        \item Every unlabeled ball is labeled ``$r$''.
        \item Once all balls in row $r$ are labeled, each of them is sequentially paired to the first unlabeled ball weakly to its right in row $r - 1$, wrapping around from column n to column $1$ if necessary. Once a ball in row $r-1$ is paired with a ball in row $r$, it acquires the label from the ball that paired to it. The order in which balls are paired from row $r$ to row $r - 1$ is from the largest label ``$L$'' to the smallest label ``$r$'', and (by convention) from left to right among balls with the same label. 
    \end{itemize}
    To complete the process, all unpaired balls in row $1$ are labeled ``$1$''. Define the array $L_M$ by setting $L_M(r,i)$ to be the label assigned to the ball at site $(r,i)$, or 0 if there is no ball. 
\end{definition}

See \cref{ex:FM algorithm} for a detailed run-through of the algorithm, and \cref{fig:MLQ array example} for an example of the array $L_M$. For a multiline queue $M=(B_1,B_2,...,B_L)$, the \textbf{content} of the multiline queue $x^M$ is the monomial:
\[ 
    x^M:= \prod_{i=1}^L\prod_{b\in B_i}x_b.
\]

\begin{definition}
    A \textbf{strand} of a multiline queue $M$ is a set of balls $\{b_1,...,b_h\}$, such that $b_i$ is in row $i$, and FM algorithm pairs $b_{i+1}$ with $b_i$ for all $i$. The \textbf{length} of the strand is $h$. The column containing the ball in row 1 of the strand is its \textbf{anchor}; if a strand of length $k$ is anchored at $i$, then equivalently $\type(M)_i=k$. Note that all the balls in a fixed strand have the same label, which is the length of the strand.
\end{definition}

\begin{definition}
    The \textbf{type} of a multiline queue is the composition $\type(M)$ obtained by reading the labels on the first row of $L_M$ and recording $0$ for an empty site. The \textbf{strong type} of a multiline queue is the strong composition $\strtype(M):= \type(M)^+$.
\end{definition}

\begin{definition}[Row word and Column word]
    For a multiline queue $M$, the \textbf{row word} of $M$, $rw(M)$, is obtained by recording the column numbers of balls in $M$, scanning the rows \textit{top to bottom}, and traversing each row \textit{right to left}. The \textbf{column word} of $M$, $cw(M)$, is obtained by recording the row numbers of balls in $M$, scanning the columns \textit{left to right}, and traversing each column \textit{top to bottom}.
\end{definition}
\begin{remark}\label{rem:not unique}
    Different pairing orders among balls of the same label may change the pairings themselves, but not which balls are paired to in the row below: thus the labels assigned remain the same (see, e.g.~\cite[Lemma 2.2]{AGS20}, \cite[Remark 3.14]{MV24}). Hence, for concreteness we adopt a left–to–right pairing convention and we treat the resulting pairings as a decoration on the multiline queue diagram (the pairings are determined uniquely up to this convention). In particular, $\type(M)$ is independent of the chosen pairing order.
\end{remark}

\begin{example}\label{ex:FM algorithm}
    We show the sequence of steps of the FM pairing procedure for the multiline queue $M=(\{1,3,4,5\},\{2,3,4\},\{3,5\})\in\MLQ_{(3,3,2,1)}$.
    \begin{center}
\resizebox{!}{1.75cm}
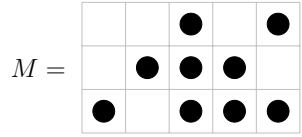
{
\begin{tikzpicture}[scale=0.7]
\def \w{1};
\def \h{1};
\def \r{0.25};
    
\begin{scope}[xshift=0cm]
\node at (-1,1.5) {\large $M=$};
\draw[gray!50,thin,step=\w] (0,0) grid (5*\w,3*\h);
\foreach \xx\yy\i\c in {2/2/3/black,4/2/3/black,1/1/3/black,2/1/3/black,3/1/2/black,0/0/1/black,2/0/3/black,3/0/3/black,4/0/2/black}
    {
    \draw[fill=\c] (\w*.5+\w*\xx,\h*.5+\h*\yy) circle (\r cm);
    }
\end{scope}

\end{tikzpicture}
}
\end{center}

For $r=3$, we start by labelling all particles in row 3 with ``3''. Next, for $\ell=3$, pair particles with label $\ell$ in row 3, from left to right, to an available particle in row 2 weakly to the right:

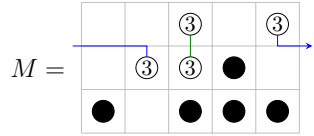
\begin{center}
\resizebox{!}{1.75cm}{
\begin{tikzpicture}[scale=0.7]
\def \w{1};
\def \h{1};
\def \r{0.25};
    
\begin{scope}[xshift=0cm]
\node at (-1,1.5) {\large $M=$};
\draw[gray!50,thin,step=\w] (0,0) grid (5*\w,3*\h);
\foreach \xx\yy\i\c in {2/2/3/white,4/2/3/white,1/1/3/white,2/1/3/white,3/1/2/black,0/0/1/black,2/0/3/black,3/0/3/black,4/0/2/black}
    {
    \draw[fill=\c] (\w*.5+\w*\xx,\h*.5+\h*\yy) circle (\r cm);
    \node at (\w*.5+\w*\xx,\h*.5+\h*\yy) {\i};
    }

\draw[black!50!green] (\w*2.5,\h*2.5-\r)--(\w*2.5,\h*1.5+\r);
\draw[blue,-stealth] (\w*4.5,\h*2.5-\r)--(\w*4.5,\h*2.0)--(\w*5.3,\h*2.0);
\draw[blue] (-.2,\h*2.0)--(\w*1.5,\h*2.0)--(\w*1.5,\h*1.5+\r);
\end{scope}

\end{tikzpicture}
}
\end{center}
    
    Next, for $r=2$, label all unlabelled particles in row 2 with ``2''. Pair particles with label $\ell=3$ in row 2 to unpaired particles in row 1:

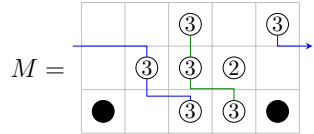
\begin{center}
\resizebox{!}{1.75cm}{
\begin{tikzpicture}[scale=0.7]
\def \w{1};
\def \h{1};
\def \r{0.25};
    
\begin{scope}[xshift=0cm]
\node at (-1,1.5) {\large $M=$};
\draw[gray!50,thin,step=\w] (0,0) grid (5*\w,3*\h);
\foreach \xx\yy\i\c in {2/2/3/white,4/2/3/white,1/1/3/white,2/1/3/white,3/1/2/white,0/0/1/black,2/0/3/white,3/0/3/white,4/0/2/black}
    {
    \draw[fill=\c] (\w*.5+\w*\xx,\h*.5+\h*\yy) circle (\r cm);
    \node at (\w*.5+\w*\xx,\h*.5+\h*\yy) {\i};
    }

\draw[black!50!green] (\w*2.5,\h*2.5-\r)--(\w*2.5,\h*1.5+\r);
\draw[blue,-stealth] (\w*4.5,\h*2.5-\r)--(\w*4.5,\h*2.0)--(\w*5.3,\h*2.0);
\draw[blue] (-.2,\h*2.0)--(\w*1.5,\h*2.0)--(\w*1.5,\h*1.5+\r);
    
\draw[black!50!green] (\w*2.5,\h*1.5-\r)--(\w*2.5,\h*1.02)--(\w*3.5,\h*1.02)--(\w*3.5,\h*.5+\r);
\draw[blue] (\w*1.5,\h*1.5-\r)--(\w*1.5,\h*.85)--(\w*2.5,\h*.85)--(\w*2.5,\h*.5+\r);

\end{scope}

\end{tikzpicture}
}
\end{center}
    To complete the pairings for $r=2$, pair particles with label $\ell=2$ in row 2 to unpaired particles weakly to the right in row 1. Finally, for $r=1$, label all unlabeled particles in row 1 with ``1'', completing the pairing procedure, and yielding the corresponding labeled array $L_M$. The multiline queue $M$ and its labeled array $L_M$ are shown in \cref{fig:MLQ array example}.

    \begin{figure}[h!]
    \centering
\resizebox{!}{1.75cm}{
\begin{tikzpicture}[scale=0.7]
\def \w{1};
\def \h{1};
\def \r{0.25};
    
\node at (-1,1.5) {\large $M=$};
\draw[gray!50,thin,step=\w] (0,0) grid (5*\w,3*\h);
\foreach \xx\yy\i\c in {2/2/3/white,4/2/3/white,1/1/3/white,2/1/3/white,3/1/2/white,0/0/1/white,2/0/3/white,3/0/3/white,4/0/2/white}
    {
    \draw[fill=\c](\w*.5+\w*\xx,\h*.5+\h*\yy) circle (\r cm);
    \node at (\w*.5+\w*\xx,\h*.5+\h*\yy) {\i};
    }
    
\draw[black!50!green] (\w*2.5,\h*2.5-\r)--(\w*2.5,\h*1.5+\r);
\draw[blue,-stealth] (\w*4.5,\h*2.5-\r)--(\w*4.5,\h*2.0)--(\w*5.3,\h*2.0);
\draw[blue] (-.2,\h*2.0)--(\w*1.5,\h*2.0)--(\w*1.5,\h*1.5+\r);
    
\draw[black!50!green] (\w*2.5,\h*1.5-\r)--(\w*2.5,\h*1.02)--(\w*3.5,\h*1.02)--(\w*3.5,\h*.5+\r);
\draw[blue] (\w*1.5,\h*1.5-\r)--(\w*1.5,\h*.85)--(\w*2.5,\h*.85)--(\w*2.5,\h*.5+\r);
\draw[red] (\w*3.5,\h*1.5-\r)--(\w*3.5,\h*1.15)--(\w*4.5,\h*1.15)--(\w*4.5,\h*.5+\r);

\node at (-1+8.5,1.5) {\large $L_M=$};
\draw[gray!50,thin,step=\w] (9,0) grid (5*\w+9,3*\h);
\foreach \xx\yy\i\c in {2/2/3/white,4/2/3/white,1/1/3/white,2/1/3/white,3/1/2/white,0/0/1/white,2/0/3/white,3/0/3/white,4/0/2/white,0/2/0/white,0/1/0/white,1/2/0/white,1/0/0/white,3/2/0/white,4/1/0/white}
    {
    \node at (\w*.5+\w*\xx+9,\h*.5+\h*\yy) {\i};
    }
\end{tikzpicture}
}
    \caption{An example of a labeled multiline queue $M$ with its corresponding array $L_M$.}
    \label{fig:MLQ array example}
\end{figure}
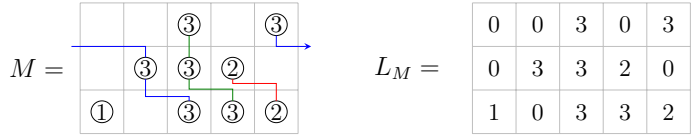

\noindent Reading the labels off the bottom row, we get $\type(M)=(1,0,3,3,2)$ and $\strtype(M)=(1,3,3,2)$. Recording the column numbers of balls in the order defined for row word, we get $rw(M) = 53.432.5431$. Here, labels of each row are separated by ".". Similarly, the column word is given by $cw(M) = 1.2.321.21.31$. 
\end{example}

We treat the pairings on the multiline queue as auxiliary data associated to the ball configuration that can be recovered uniquely through the FM algorithm. In our examples, we will sometimes omit the pairings and labels and show only the underlying ball configuration.

\begin{definition}[Major index of a multiline queue]
    Let $M \in \MLQ_\lambda$ and let $m_{\ell,r}$ be the number of balls labelled $\ell$ in row $r$ that wrap when paired to the ball labelled $\ell$ in row $r-1$. The \textbf{major index} of $M$ is then computed as follows:
        \begin{align*}
            \maj(M) = \sum_{2 \leq \ell \leq L}\sum_{2 \leq r \leq \ell} m_{\ell,r}(\ell - r + 1).
        \end{align*}
\end{definition}

   The following lemma describes the possible intersections of two strands in a non-wrapping multiline queue.
    \begin{lemma}\label{lemma:intersectingstrands}
        Let $s_1, s_2$ be two strands in a non-wrapping multiline queue $M$, with row-1 balls $y_1$ and $y_2$ respectively, and suppose that $s_1$ is weakly longer than $s_2$. Then $s_1$ and $s_2$ can intersect each other at most once. Moreover, if they intersect, then $y_1$ is to the left of $y_2$.
    \end{lemma}
    
    \begin{proof}
        Let $s_1$ and $s_2$ have lengths $k$ and $\ell$, respectively (accordingly, these are also the labels of balls $y_1$ and $y_2$ in $L(M)$). If $k=\ell$, the strands do not intersect by definition according to the FM algorithm.
        
        Now suppose $k>\ell$. For $i\leq \ell$, let $\alpha_i\in \{k,\ell\}^2$ be the pair of labels of the balls in strands $s_1$ and $s_2$ in row $i$ when read left to right. Since the balls with label $k$ pair before balls with label $\ell$ in the FM algorithm, if $\alpha_i = (k,\ell)$, then $\alpha_{i-1} = (\ell,k)$ only if $s_2$ wraps around from row $i$ to $i-1$ -- a contradiction, since $M$ is non-wrapping. Moreover, again due to the pairing order in FM algorithm, if $\alpha_i = (\ell,k)$, then $\alpha_{i-1} = (k,\ell)$ if and only if the two strands intersect when $M$ is non-wrapping. 
        It follows that there is at most one index $i>1$ such that the two strands can intersect from row $i$ to row $i-1$. Furthermore, if there is such an $i$, then for all rows $i\leq j\leq k$, $\alpha_j=(\ell,k)$, and for $1\leq j<i$, we have $\alpha_{j} = (k,\ell)$, with $\alpha_1=(k,\ell)$ confirming $y_1$ is indeed left of $y_2$.
    \end{proof}

As the major index of a multiline queue $M$ encodes the pairings of $M$ that wrap, we say a multiline queue $M$ is \textbf{non-wrapping} if $\maj(M) = 0$. We let $\NMLQ_\lambda$ denote the set of all non-wrapping multiline queues in $\MLQ_\lambda$. Before proceeding, we establish notation for the components of $\MLQ_\lambda$ indexed by $\type$ or $\strtype$. 
\begin{definition}
    For a composition $\alpha$ with $\sort(\alpha) = \lambda$, $\MLQ[\alpha]$ is the set of multiline queues $M$ with $\type(M) = \alpha$. For a strong composition $\gamma$ with $\sort(\gamma) = \lambda$, define $\SMLQ[\gamma]$ to be the set of multiline queues $M$ with $\strtype(M) = \gamma$. Similarly, define $\NMLQ[\alpha]$ and $\SNMLQ[\gamma]$ to be the sets of nonwrapping multiline queues in $\MLQ[\alpha]$ and $\SMLQ[\gamma]$, respectively.
\end{definition}

Define also $M_\alpha\in\NMLQ[\alpha]$ to be the unique multiline queue with column content $\alpha$. We call this the \textbf{straight multiline queue} of type $\alpha$ since all balls in this configuration pair directly below, and all strands are vertical. We will use $M_\alpha$ as a ``representative'' for $\NMLQ[\alpha]$. Multiline queues can be alternatively described as binary matrices with finite support, with row sums corresponding to partitions. It will be useful to work in the larger set of unrestricted binary matrices with finite support:
\begin{definition}
    A \textbf{generalized multiline queue} $M$ is an $L\times n$ matrix with entries in $\{0,1\}$, represented as a tuple $M =(B_1,B_2,...,B_L)$, where $B_i$ is the set of columns with 1 in row $i$. Let $\mathcal{M}_2(L,n)$ denote the set of all generalized multiline queues.
\end{definition}

\begin{remark}The FM algorithm can be extended to generalized multiline queues as well, in order to associate a weak type to each generalized multiline queue \cite{AAMP11, AGS20} and a (generalized) major index \cite[Definition 4.29]{MV24} to each one. However, we will not be using these properties and thus refer the interested reader to the full treatment in \cite{AGS20} and \cite{MV24} (see also \cite{MS25}), where these are referred to as \emph{twisted multiline queues}. In particular, the crystal operators on multiline queues that we describe in \cref{sec:crystals} act on this larger set of objects. 
\end{remark}

\subsection{Schur functions and their refinements}\label{sec:quasischur}

The \textbf{Schur basis} $\{s_\lambda(X)\}$ for $\Sym$ is the unique basis that is orthogonal with respect to the Hall inner product, and upper triangular with respect to the $\{m_\lambda(X)\}$ basis. Classically, it is defined combinatorially as the generating function of \textbf{semistandard Young tableaux (SSYT)}.

An SSYT $T$ of shape $\lambda$ is a filling of the diagram of $\lambda$ with positive integers such that the entries weakly increase from left to right along rows and strictly increase from bottom to top along columns.  The \textbf{content} of a semistandard Young tableau $T$ is the composition $c(T) = (c_1,c_2,c_3,\ldots)$ where $c_i$ counts the occurrences of $i$ in $T$. Let $\SSYT(\lambda)$ denote the set of all semistandard Young tableaux of shape $\lambda$. Let $\SSYT(\lambda,\alpha)$ denote the subset with content $\alpha$. 

We represent the content of a tableau $T$ as the monomial $x^T := x^{c(T)}$. The \textbf{Schur function} $s_\lambda(X)$ is then defined as
    \begin{align*}
        s_\lambda = \sum_{T \in \SSYT(\lambda)}x^T.
    \end{align*}
A weight-preserving bijection between $\SSYT(\lambda)$ and $\NMLQ_\lambda$ was given in \cite[Theorem~A.4]{MV24}. Hence, we may equivalently define Schur functions as the generating function of non-wrapping multiline queues:
    \begin{align*}
        s_\lambda = \sum_{M \in \NMLQ_\lambda}x^M.
    \end{align*}
    
The \textbf{Demazure atoms} $\cA_\alpha$ (letting $\alpha$ range over all compositions) form a basis for the full polynomial ring $\QQ[X]$, and are a nonsymmetric refinement of the Schur basis. These were first studied by Lascoux and Sch\"utzenberger \cite{LS90} and were subsequently given a combinatorial description by Mason \cite{Mas09}. In \cite{Mas09}, Mason gives a description of these polynomials as a generating function of \textbf{semistandard augmented fillings (SSAF)}. There exists a weight-preserving bijection between semistandard augmented fillings of shape $\alpha$ and non-wrapping multiline queues of type $\alpha$ (see \cref{sec:comparison}). Thus, for any weak composition $\alpha$ such that $\sort(\alpha) = \lambda$, the \textbf{Demazure atom} $\cA_\alpha(X)$ is defined as 
    \begin{align*}
        \cA_\alpha(X) = \sum_{M \in \NMLQ[\alpha]}  x^M.
    \end{align*}

Between the Schur functions and the Demazure atoms are the \textbf{quasisymmetric Schur functions} $\QS_\gamma(X)$. Originally introduced by Haglund, Luoto, Mason, and van Willigenburg in \cite{HLMvW11}, they form a basis for $\QSym$ and are a quasisymmetric refinement of the Schur basis. They arise as a sum of Demazure atoms and inherit a combinatorial description as a sum over multiline queues of a prescribed strong type. For any strong composition $\gamma$, the quasisymmetric Schur function $\QS_\gamma(X)$ is defined as
    \begin{align*}
        \QS_\gamma(X) = \sum_{\alpha: \alpha^+ = \gamma} \cA_\alpha(X) = \sum_{\alpha: \alpha^+ = \gamma} \sum_{M \in \NMLQ[\alpha]} x^M = \sum_{M \in \SNMLQ[\gamma]} x^M.
    \end{align*}

One benefit of these multiline queue formulas is that it is immediate that the Demazure atoms and quasisymmetric Schur functions are nonsymmetric and quasisymmetric refinements of the Schur functions respectively. 

There is a similar story for $q$-generalizations of these bases, which we will see in the next section.

\subsection{$q$-Whittaker polynomials and their quasisymmetric and nonsymmetric refinements}\label{sec:qWhittaker}

The nonsymmetric Macdonald polynomials $E_\alpha(X;q,t)$, indexed by compositions $\alpha$, form a basis for the polynomial ring with parameters $q,t$. The number of variables in $E_\alpha(X;q,t)$ is determined by the number of parts of $\alpha$, so when we write $E_\alpha(X;q,t)$, we really mean $E_\alpha(x_1,\ldots,x_{m};q,t)$, where $m$ is the number of parts in $\alpha$. The $E_\alpha(X;q,t)$'s generalize $P_\lambda(X;q,t)$ in the sense that for any $\alpha$ such that $\lambda=\sort(\alpha)$ and any $m\in\NN$, we have
\[
P_{\lambda}(x_1,\ldots,x_m;q,t)=E_{0^m\alpha}(x_1,\ldots,x_m,0,\ldots,0;q,t).
\]
Moreover, the $q=t=\infty$ specialization yields the Demazure atom: $E_\alpha(X;\infty,\infty)=\cA_\alpha(X)$, proved combinatorially by Mason \cite{Mas09}. 

$P_\lambda(X;q,t)$ decomposes into nonsymmetric components in the $E_\alpha$ basis as
\[
P_\lambda(x_1,\ldots,x_n;q,t)=\sum_{\substack{\alpha=(\alpha_1,\ldots,\alpha_n)\\\sort(\alpha)=\lambda}}c_{\lambda\alpha}(q,t)E_\alpha(x_1,\ldots,x_n;q,t),\qquad c_{\lambda\alpha}(q,t)\in\QQ(q,t).
\]
Replacing $(q,t)\mapsto (q^{-1},t^{-1})$ and setting $q=t=0$, the left hand side becomes $s_\lambda(X)$, the coefficients in the sum become $c_{\lambda\alpha}(\infty,\infty)=1$, and the above decomposition reduces to the standard nonsymmetric refinement of Schur functions
\[
s_\lambda(X)=\sum_{\alpha:\sort(\alpha)=\lambda}\cA_\alpha(X)\,.
\]
The \textbf{permuted basement Macdonald polynomials} $E_\alpha^\sigma(X;q,t)$ are a generalization of the $E_\alpha$'s, additionally indexed by a permutation $\sigma\in\mathfrak{S}_m$ of the parts in $\alpha$ (assuming $\alpha$ has $m$ parts). Their properties were studied in \cite{Fer11,Ale19}, with a combinatorial formula for them given as a certain restriction of the original Haglund--Haiman--Loehr formulas in \cite{HHL08}. 

In \cite{CMW22}, a new formula for $P_\lambda(X;q,t)$ was found as a sum over certain enhanced multiline queues introduced by \cite{Mar20} (the multiline queues presented in \cref{sec:MLQ} are the $t=0$ version of the latter). This formula was found following the discovery in \cite{CdGW15} that the specialization at $x_1=\cdots=x_n=q=1$ of $P_\lambda(x_1,\ldots,x_n;1,q)$ is the \textbf{partition function} of the asymmetric simple exclusion process (ASEP) of type $\lambda$ on a circular lattice with $n$ sites. Moreover, the $x_1=\cdots=x_n=q=1$ specialization of the generating function of the multiline queues of \cite{CMW22}, when restricted to a given weak type $\alpha$, compute the stationary probability of the state $\alpha$ of the ASEP of type $(\lambda,n)$. The corresponding polynomials $f_\alpha(X;q,t)$-- obtained as generating functions of the enhanced multiline queues with weak type $\alpha$-- were hence named the \textbf{ASEP polynomials}. Importantly, this connection showed that the set of polynomials $\{f_\alpha(X;q,t)\}$ yield an honest refinement of $P_\lambda(X;q,t)$ \cite[Lemma 3]{CdGW15}:
\[
P_\lambda(x_1,\ldots,x_n;q,t)=\sum_{\substack{\alpha=(\alpha_1,\ldots,\alpha_n)\\\sort(\alpha)=\lambda}} f_\alpha(x_1,\ldots,x_n;q,t).
\]
(Note that a version of this expression can also be found in \cite{Mac96} and \cite[3.3.14]{HHL08}.) 

We obtain elegant formulas for the $q$-Whittaker polynomials $P_\lambda(X;q,0)$ and their nonsymmetric refinements $f_\alpha(X;q,0)$.
\begin{equation}\label{eq:P and F}
    P_\lambda(X;q,0)=\sum_{M\in\MLQ_\lambda}x^Mq^{\maj(M)},\qquad f_\alpha(X;q,t)=\sum_{M\in\MLQ[\alpha]}x^Mq^{\maj(M)}.
\end{equation}

In our convention, $E_\alpha(X;q,t)=E_{\rev(\alpha)}^{w_0}(X;q,t)$, where $w_0$ is the longest permutation and $\rev(\alpha)$ is reversal of $\alpha$; the rest of the $E_\alpha^\sigma$'s are obtained by sequentially applying \textbf{Demazure--Lusztig} operators to $E_\alpha$. It follows that since the ASEP polynomials satisfy the qKZ relations, the $f_\alpha$'s coincide with the nonsymmetric permuted basement Macdonald polynomials with lower index given by the anti-partition $\inc(\alpha)$. More precisely,
\[
f_\alpha(X;q,t)=E_{\inc(\alpha)}^\sigma(X;q,t),
\]
where $\sigma$ is any permutation that sorts $\alpha$ into $\inc(\alpha)$, namely $\alpha_{\sigma(1)}\leq \cdots\leq \alpha_{\sigma(n)}$.

Corteel, Haglund, Mason, Williams, and the first author showed in \cite{CHMMW22} that $f_\alpha(X;0,0)=\cA_\alpha(X)$, making the decomposition of $P_\lambda$ into $f_\alpha$'s a direct parallel to the Schur decomposition into Demazure atoms. This parallel made it natural to define the \textbf{quasisymmetric Macdonald polynomial} $G_\gamma(X;q,t)$, indexed by strong compositions, as the quasisymmetric refinement lying in between the $f_\alpha$'s and the $P_\lambda$'s:
\begin{definition}[{\cite{CHMMW22}}]
For a strong composition $\gamma$, define the quasisymmetric Macdonald polynomial $G_\gamma(X;q,t)$ by
\[G_\gamma(X;q,t):=\sum_{\alpha:\alpha^+=\gamma}f_\alpha(X;q,t).
\]
\end{definition}
Combinatorially, $G_\gamma(X;q,t)$ can be obtained as generating function of a restriction by strong type $\gamma$ of the set of enhanced multiline queues corresponding to $P_\lambda(X;q,t)$ in \cite{CMW22}. Thus at $t=0$, we obtain the following multiline queue formula,
\begin{equation}
    G_\gamma(X;q,0)=\sum_{M\in\SMLQ[\gamma]}x^Mq^{\maj(M)},
\end{equation}
and at $q=t=0$ we recover the quasisymmetric Schur functions:
\[G_\gamma(X;0,0)=\QS(X).\]

\subsection{Crystal operators on multiline queues}\label{sec:crystalops}

Considering a (generalized) multiline queue as a binary matrix, there are two natural tensor product realizations: by its rows or by its columns. Each admits a type $A^{(1)}$ Kirillov--Reshetikhin (KR) crystal structure, which corresponds to acting by crystal operators on the columns of the multiline queue in the first case, or on its rows in the second. We refer to these as the \textbf{column} and \textbf{row} operators. 

We first define the classical crystal operators on words. Kashiwara first described these in \cite{Kashiwara:1989md}, where he describes a basis for the $q$ analogue of the universal enveloping algebra $U_q(\mathfrak{g})$. The operators defined here are the root operators in the setting $\mathfrak{g}=\mathfrak{sl}_n$ and $q=0$. The column and row crystal operators on multiline queues correspond to the action of the classical crystal operators on the column word and the row word of the multiline queue, respectively. Using these, we will define a double crystal structure on $\mathcal{M}_2(L,n)$.

For a positive integer $n$, let $\Wc=[n]^*$ be the set of all words in the alphabet $[n]$.

\begin{definition}[Bracketing rule]
    Let $w\in \Wc$. For $1\leq i\leq n-1$, let $\theta_i(w)$ be the word in the letters $\{(,)\}$ formed by scanning $w$ from left to right, and replacing each $i$ with ")" and each $i+1$ with "(". Next, we iteratively match each opening parenthesis ( with a closing parenthesis ) to its right whenever there are no unmatched parentheses in between them, continuing until no more pairs can be formed. We then call the $i$'s and $i+1$'s corresponding to the matched (\emph{resp.}~unmatched) parentheses in $\theta_i(w)$ matched (\emph{resp.}~unmatched) letters in $w$. This procedure for pairing $i$'s and $i+1$'s is called the \textbf{bracketing} or $\textbf{signature}$ rule. 
\end{definition}

\begin{definition}[Raising and lowering operators]\label{def:wordoperators}
    We define $E_i:\Wc\to\Wc$ as follows: if $\theta_i(w)$ has no unmatched $i+1$'s then $E_i(w) = w$, if it does, then $E_i(w)$ is the word obtained by changing the leftmost unmatched $i+1$ in $w$ to $i$. We define $F_i:\Wc\to \Wc$ as follows: if $\theta_i(w)$ has no unmatched $i$'s then $F_i(w) = w$, if it does, then $F_i(w)$ is the word obtained by changing the rightmost unmatched $i$ in $w$ to $i+1$. We also define $E_i^\star(w)$ to be the word where we change all the unmatched $i+1$'s in $w$ to $i$.
\end{definition}

\begin{example}
    Let $w = 2211323113323$ and $i=2$. We use $\textunderscore$ to denote the entries which are not $i$ and $i+1$, and a hat/red colour to denote unmatched parentheses. We get $\theta_2(w)$:
    \[
    \theta_2(w)= \tcr{\hat)}\;\tcr{\hat)}\;\textunderscore\;\textunderscore\;(\;)\;\tcr{\hat(}\;\textunderscore\;\textunderscore\;\tcr{\hat(}\;(\;)\;\tcr{\hat(}
    \]
    The numbers corresponding to unmatched parentheses are hatted here: $w= \hat{2}\;\hat{2}\;1\;1\;3\;2\;\hat{3}\;1\;1\;\hat{3}\;3\;2\;\hat{3}$. Therefore we have the following, with changes underlined/red:
    \begin{align*}
        E_2(w) &= \hat{2}\;\hat{2}\;1\;1\;3\;2\;{\color{red}\underline{\hat2}}\;1\;1\;\hat{3}\;3\;2\;\hat{3}\\
        F_2(w)&=\hat{2}\;{\color{red}\underline{\hat3}}\;1\;1\;3\;2\;\hat{3}\;1\;1\;\hat{3}\;3\;2\;\hat{3}\\
        E_2^\star(w)&=\hat{2}\;\hat{2}\;1\;1\;3\;2\;{\color{red}\underline{\hat2}}\;1\;1\;{\color{red}\underline{\hat2}}\;3\;2\;\textcolor{red}{\underline{\hat2}}\,.
    \end{align*}
\end{example}

\begin{definition}[Row dropping operator]
    Define the \emph{dropping operator} $\eD_i$ to act on a (generalized) multiline queue $M$ by moving the ball corresponding to the leftmost unmatched entry $i+1$ in $\theta_i(cw(M))$ from row $i+1$ to row $i$; if all $i+1$'s are matched, $\eD_i(M)=M$.   We define $\eDs_i(M)\coloneqq (\eD_i)^{\circ j}(M)$, where $j$ is maximal such that $\eD_i$ acts nontrivially. In other words, $\eDs_i(M)$ is the generalized multiline queue obtained by dropping all the unmatched balls in row $i+1$ to row $i$.
\end{definition}

Similarly, the \emph{row lifting operator} $\fU_i$ acts on $M$ by moving the ball corresponding to the rightmost unmatched entry $i$ in $\theta_i(cw(M))$ from row $i$ to row $i+1$; if there is no such unmatched entry, $\fU_i(M)=M$. We will not use the operators $\fU_i$ in this paper, and direct the reader to \cite{MV24} for further details.

\begin{remark}\label{rem:unmatched}
    While the words "paired" and "matched" are closely related, here we assign them strictly different meanings. We use the words "(un)paired" to refer to balls that are (un)paired at any given step of the FM algorithm. In contrast, we use the words "(un)matched" to refer to balls in the multiline queue $M$ corresponding to the (un)matched entries in the bracket word $\theta_i(w(M))$ according to the usual bracketing rule (here $w(M)$ is either $\rw(M)$ or $\cw(M)$). In particular, \emph{paired} balls are by definition in the same strand and have the same label, while \emph{matched} balls are determined solely from the given bracketing word and ignore the data of the FM algorithm and the labeling $L(M)$.   
\end{remark}

\begin{definition}[Column operators]
We similarly define the \emph{column operators} $\eL_i$ and $\fR_i$, which act on the generalized multiline queue $M$ by moving the ball corresponding to the leftmost unmatched $i+1$ in $\theta_i(\rw(M))$ from column $i+1$ to column $i$, and the rightmost unmatched $i$ in $\theta_i(\rw(M))$ from column $i$ to column $i+1$, respectively. We use the notation $g^\leftrightarrow_i$ to represent an arbitrary column operator $\eL_i$ or $\fR_i$. 
\end{definition}

The operators we have defined can equivalently be characterized as follows:
\[\rw(\eL_i(M))=E_i(\rw(M)),\quad \rw(\fR_i(M))=F_i(\rw(M)),\quad \cw(\eD_i(M))=E_i(\cw(M)).
\]

While column operators are equivalent to applying row operators to a $90^\circ$ rotation of the multiline queue, in our setting they play very different roles: column operators will define the edges of a  multiline queue graded crystal graph, whereas the row dropping operators will project a multiline queue onto the 0-graded component (see \cref{def:graded crystal}).

\begin{definition}[Multiline queue crystal graph]
    The set $\MLQ_\lambda$ (or any subset of it) can be given a structure of a directed graph, where the vertices of the graph are the elements of the set $\MLQ_\lambda$, and we have a directed edge from multiline queue $M$ to $M'$ if and only if there is an $i$ such that $\fR_iM = M'$. For two multiline queues $M, M'\in \MLQ_\lambda$, a \textbf{path} from $M$ to $M'$ is sequence of mutliline queues $M=M_1,M_2,...,M_k=M'$, such that for all $j$, there is an $i$, such that $M_{j+1} = g^\leftrightarrow_i M_j$ and $M_j\not=M_h$ for all $j\not=h$.
\end{definition}
\begin{remark}
    Note that the paths on the crystal graph can traverse the directed edge opposite to the direction of the edge. In such cases, we are using $\eL_i$ instead of $\fR_i$ to traverse the edge.
\end{remark}
In \cref{fig:big figure}, we give an example of the multiline queue crystal graph $\NMLQ_{(3,3,1,1)}$.

\begin{example}
    Take the multiline queue $M$ from \cref{ex:FM algorithm}, which has $\rw(M)=53.432.5431$ and $\cw(M)=1.2.321.21.31$. The following bracketing words show unmatched brackets hatted/red:
    \begin{align*}
        \theta_2(rw(M))= \textunderscore\textcolor{red}{\hat(}\textunderscore(\;)\textunderscore\textunderscore\textcolor{red}{\hat(}\textunderscore\,,\quad
        \theta_3(rw(M))= \textunderscore\textcolor{red}{\hat)}\;(\;)\textunderscore\textunderscore(\;)\textunderscore\,,\quad
        \theta_2(cw(M))=\textunderscore\textcolor{red}{\hat)}\;(\;)\textunderscore\textcolor{red}{\hat)}\textunderscore\textcolor{red}{\hat(}\textunderscore\,.
    \end{align*}
    Therefore, $\fR_3(M), \eL_2(M)$, and $\eD_2(M)$ are shown as follows, with the balls that were moved by the corresponding operators highlighted:
    \begin{center}
\resizebox{!}{1.5cm}
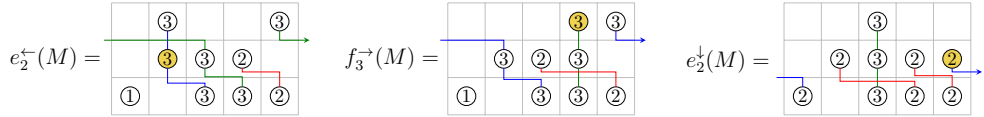
{
\begin{tikzpicture}[scale=0.7]
\def \w{1};
\def \h{1};
\def \r{0.25};
%

%
%
\begin{scope}[xshift=-9cm]    
\node at (-1.5,1.5) {\large $\eL_2(M)=$};
\draw[gray!50,thin,step=\w] (0,0) grid (5*\w,3*\h);
\foreach \xx\yy\i\c in {1/2/3/white,4/2/3/white,1/1/3/\highlight,2/1/3/white,3/1/2/white,0/0/1/white,2/0/3/white,3/0/3/white,4/0/2/white}
    {
    \draw[fill=\c](\w*.5+\w*\xx,\h*.5+\h*\yy) circle (\r cm);
    \node at (\w*.5+\w*\xx,\h*.5+\h*\yy) {\i};
    }
    
\draw[blue] (\w*1.5,\h*2.5-\r)--(\w*1.5,\h*1.5+\r);
\draw[black!50!green,-stealth] (\w*4.5,\h*2.5-\r)--(\w*4.5,\h*2.0)--(\w*5.3,\h*2.0);
\draw[black!50!green] (-.2,\h*2.0)--(\w*2.5,\h*2.0)--(\w*2.5,\h*1.5+\r);
    
\draw[black!50!green] (\w*2.5,\h*1.5-\r)--(\w*2.5,\h*1.02)--(\w*3.5,\h*1.02)--(\w*3.5,\h*.5+\r);
\draw[blue] (\w*1.5,\h*1.5-\r)--(\w*1.5,\h*.85)--(\w*2.5,\h*.85)--(\w*2.5,\h*.5+\r);
\draw[red] (\w*3.5,\h*1.5-\r)--(\w*3.5,\h*1.15)--(\w*4.5,\h*1.15)--(\w*4.5,\h*.5+\r);
\end{scope}

    %
    %
    %
\begin{scope}[xshift=0cm]
\node at (-1.5,1.5) {\large $\fR_3(M)=$};
\draw[gray!50,thin,step=\w] (0,0) grid (5*\w,3*\h);
\foreach \xx\yy\i\c in {3/2/3/\highlight,4/2/3/white,1/1/3/white,2/1/2/white,3/1/3/white,0/0/1/white,2/0/3/white,3/0/3/white,4/0/2/white}
    {
    \draw[fill=\c](\w*.5+\w*\xx,\h*.5+\h*\yy) circle (\r cm);
    \node at (\w*.5+\w*\xx,\h*.5+\h*\yy) {\i};
    }
    
\draw[black!50!green] (\w*3.5,\h*2.5-\r)--(\w*3.5,\h*1.5+\r);
\draw[blue,-stealth] (\w*4.5,\h*2.5-\r)--(\w*4.5,\h*2.0)--(\w*5.3,\h*2.0);
\draw[blue] (-.2,\h*2.0)--(\w*1.5,\h*2.0)--(\w*1.5,\h*1.5+\r);
    
\draw[black!50!green] (\w*3.5,\h*1.5-\r)--(\w*3.5,\h*.5+\r);
\draw[blue] (\w*1.5,\h*1.5-\r)--(\w*1.5,\h*.95)--(\w*2.5,\h*.95)--(\w*2.5,\h*.5+\r);
\draw[red] (\w*2.5,\h*1.5-\r)--(\w*2.5,\h*1.15)--(\w*4.5,\h*1.15)--(\w*4.5,\h*.5+\r);
\end{scope}
%
%
\begin{scope}[xshift=0cm]
\node at (-1+8.5,1.5) {\large $\eD_2(M)=$};
\draw[gray!50,thin,step=\w] (9,0) grid (5*\w+9,3*\h);
\foreach \xx\yy\i\c in {2/2/3/white,1/1/2/white,2/1/3/white,3/1/2/white,4/1/2/\highlight,0/0/2/white,2/0/3/white,3/0/2/white,4/0/2/white}
    {
    \draw[fill=\c](9+\w*.5+\w*\xx,\h*.5+\h*\yy) circle (\r cm);
    \node at (9+\w*.5+\w*\xx,\h*.5+\h*\yy) {\i};
    }
\draw[black!50!green] (9+\w*2.5,\h*2.5-\r)--(9+\w*2.5,\h*1.5+\r);

\draw[black!50!green] (9+\w*2.5,\h*1.5-\r)--(9+\w*2.5,\h*0.5+\r);
\draw[blue,-stealth] (9+\w*4.5,\h*1.5-\r)--(9+\w*4.5,\h*1.15)--(9+\w*5.3,\h*1.15);
\draw[red] (9+\w*3.5,\h*1.5-\r)--(9+\w*3.5,\h*1.05)--(9+\w*4.5,\h*1.05)--(9+\w*4.5,\h*.5+\r);
\draw[red] (9+\w*1.5,\h*1.5-\r)--(9+\w*1.5,\h*0.9)--(9+\w*3.5,\h*0.9)--(9+\w*3.5,\h*.5+\r);
\draw[blue] (9+-.2,\h*1.0)--(9+\w*0.5,\h*1.0)--(9+\w*0.5,\h*0.5+\r);
 \end{scope}
\end{tikzpicture}
}
\end{center}
On the other hand, all brackets are matched in $\theta_4(\rw(M))=(\;\textunderscore\;)\;\textunderscore\;\textunderscore\;(\;)\;\textunderscore\;\textunderscore\;$, so \[\eL_4(M)=\fR_4(M)=M.\] 
\end{example}

The \textbf{collapsing map} was introduced as an analogue for the RSK correspondence on multiline queues in \cite{MV24} to prove the following bijection:
\begin{equation}\label{eqn:collapsebijection}
    \MLQ_\lambda\simeq \bigcup_{\mu \leq \lambda}\NMLQ_\mu\times \SSYT(\mu', \lambda')
\end{equation}
In analogy to the classical RSK bijection, which takes a binary matrix as input and outputs a pair of semistandard tableaux, the collapsing map takes as input a multiline queue of shape $\lambda$, and outputs a non-wrapping multiline queue of shape $\mu\leq \lambda$, together with a semistandard tableau serving as the recording object, to make the operation invertible. It is convenient to define the collapsing map as a composition of dropping operators. 

For $1\leq a\leq b$, define the composition of a contiguous sequence of operators (applied from right to left):
\begin{equation*}
    \eDs_{[b,a]}\coloneq 
        \eDs_a\cdots \eDs_{b-1}\eDs_b,\qquad \eDs_{[a,b]}\coloneq
        \eDs_b\cdots \eDs_{a+1}\eDs_a\,. 
\end{equation*}
For example,  
\[\eDs_{[4,1]}=\eDs_1\eDs_2\eDs_3\eDs_4,\qquad \eDs_{[1,4]}=\eDs_4\eDs_3\eDs_2\eDs_1\,.
\]

\begin{definition}[Collapsing map]\label{def:colmap}
    Given a multiline queue $M=(B_1,\ldots,B_L)$ of shape $\lambda$, we recursively construct a non-wrapping multiline queue $\rho_N(M)$
and a recording tableau $\rho_Q(M)$. 
Initialize $M_1 = M$, and let $Q_1$ be the tableau with one row containing $|B_1|$ boxes, all filled with the entry 1. Now, assuming $(M_i, Q_i)$ has been obtained for some $1\leq i\leq L-1$, define $M_{i+1}= \eDs_{[i,1]} M_i$, and write
\[M_i = (B_1^{(i)}, B_2^{(i)},\ldots, B_L^{(i)}),\qquad M_{i+1} = (B_1^{(i+1)}, B_2^{(i+1)},\ldots, B_L^{(i+1)}).\] 
We obtain $Q_{i+1}$ from $Q_i$ by adding $|B_j^{(i+1)}|-|B_j^{(i)}|$ boxes to row  $j$ for $1\leq j<i+1$, and adding $|B_{i+1}^{(i+1)}|$ boxes to row $i+1$, and filling all the newly added boxes with the entry $i+1$. 

Finally, we define 
\[\rho_N(M)= M_{L},\qquad \rho_Q(M) = Q_{L}.
\]
\end{definition}

The collapsing operator $\rho_N:\MLQ_\lambda\mapsto\NMLQ$, for $\lambda$ with largest part $L$, can equivalently be written as
\begin{equation}
    \label{eq:rho op}
    \rho_N=\eDs_{[L-1,1]}\cdots\eDs_{[2,1]}\eDs_{[1,1]}.
\end{equation}
When the context is clear, we may abuse notation and write $\rho(M)$ for $\rho_N(M)$.

\begin{remark}
    The objects appearing in the intermediate steps of the collapsing map are generalized multiline queues in $\mathcal{M}_{(2)}$. However, the restriction of $M_i$ to rows $\{1,\ldots,i+1\}$ is by construction non-wrapping. In particular, the final output $\rho_N(M)$ is a non-wrapping multiline queue.
\end{remark}

\begin{example}\label{ex:collapsing}
    Let $M= (\{1,3,4,5\}, \{2,3,4,6\},\{2,3,4,5\},\{1,4,6\},\{3,5\})$. Below, we show the image of $M$ under the collapsing operator.
    \begin{center}
\resizebox{!}{3.5cm}
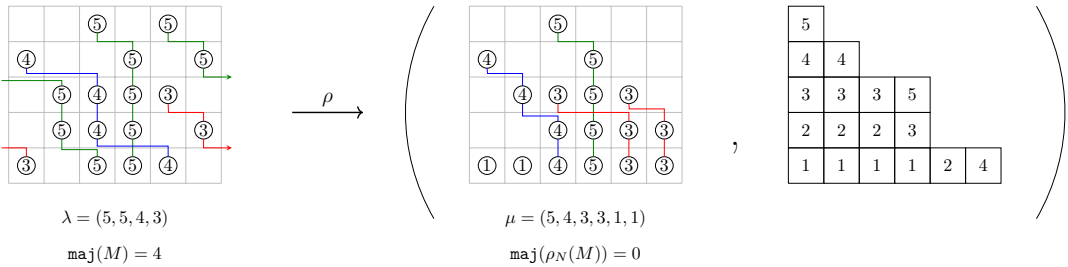
{
\begin{tikzpicture}[scale=0.7]
\def \w{1};
\def \h{1};
\def \r{0.25};

\draw[gray!50,thin,step=\w] (0,0) grid (6*\w,5*\h);
\foreach \xx\yy\i\c in {2/4/5/white,4/4/5/white,0/3/4/white,3/3/5/white,5/3/5/white,1/2/5/white,2/2/4/white,3/2/5/white,4/2/3/white, 1/1/5/white,2/1/4/white,3/1/5/white,5/1/3/white, 0/0/3/white,2/0/5/white, 3/0/5/white, 4/0/4/white}
    {
    \draw[fill=\c](\w*.5+\w*\xx,\h*.5+\h*\yy) circle (\r cm);
    \node at (\w*.5+\w*\xx,\h*.5+\h*\yy) {\i};
    }
    
\draw[black!50!green](\w*2.5,\h*4.5-\r)--(\w*2.5,\h*4)--(\w*3.5,\h*4)--(\w*3.5,\h*3.5+\r);
\draw[black!50!green](\w*4.5,\h*4.5-\r)--(\w*4.5,\h*4)--(\w*5.5,\h*4)--(\w*5.5,\h*3.5+\r);

\draw[black!50!green](\w*3.5,\h*3.5-\r)--(\w*3.5,\h*2.5+\r);
\draw[black!50!green,-stealth](\w*5.5,\h*3.5-\r)--(\w*5.5,\h*3)--(\w*6.3,\h*3);
\draw[black!50!green](-0.2,\h*2.9)--(\w*1.5,\h*2.9)--(\w*1.5,\h*2.5+\r);
\draw[blue](0.5*\w, \h*3.5-\r)--(0.5\w, \h*3.1)--(2.5*\w, \h*3.1)--(2.5*\w,\h*2.5+\r);

\draw[black!50!green](\w*1.5,\h*2.5-\r)--(\w*1.5,\h*1.5+\r);
\draw[blue](\w*2.5,\h*2.5-\r)--(\w*2.5,\h*1.5+\r);
\draw[black!50!green](\w*3.5,\h*2.5-\r)--(\w*3.5,\h*1.5+\r);
\draw[red](4.5*\w, \h*2.5-\r)--(4.5\w, \h*2)--(5.5*\w, \h*2)--(5.5*\w,\h*1.5+\r);

\draw[black!50!green] (\w*3.5,\h*1.5-\r)--(\w*3.5,\h*.5+\r);
\draw[black!50!green] (\w*1.5,\h*1.5-\r)--(\w*1.5,\h*.95)--(\w*2.5,\h*.95)--(\w*2.5,\h*.5+\r);
\draw[blue](\w*2.5,\h*1.5-\r)--(\w*2.5,\h*1.05)--(\w*4.5,\h*1.05)--(\w*4.5,\h*0.5+\r);
\draw[red,-stealth] (\w*5.5,\h*1.5-\r)--(\w*5.5,\h*1)--(\w*6.3,\h*1);
\draw[red](-0.2,\h*1)--(0.5*\w, \h)--(0.5*\w, \h*0.5+\r);
\node at (3,-1) {$\lambda=(5,5,4,3)$};
\node at (3,-2) {$\maj(M)=4$};
\draw[thick,->] (8,2)--(10,2) node[midway,above] {\large $\rho$};

\begin{scope}[xshift=13cm]

\draw[gray!50,thin,step=\w] (0,0) grid (6*\w,5*\h);
\foreach \xx\yy\i\c in {2/4/5/white,0/3/4/white,3/3/5/white,1/2/4/white,2/2/3/white,3/2/5/white,4/2/3/white, 4/1/3/white,2/1/4/white,3/1/5/white,5/1/3/white, 0/0/1/white,1/0/1/white,2/0/4/white, 3/0/5/white, 4/0/3/white,5/0/3/white}
    {
    \draw[fill=\c](\w*.5+\w*\xx,\h*.5+\h*\yy) circle (\r cm);
    \node at (\w*.5+\w*\xx,\h*.5+\h*\yy) {\i};
    }
    
\draw[black!50!green](\w*2.5,\h*4.5-\r)--(\w*2.5,\h*4)--(\w*3.5,\h*4)--(\w*3.5,\h*3.5+\r);

\draw[black!50!green](\w*3.5,\h*3.5-\r)--(\w*3.5,\h*2.5+\r);
\draw[blue](0.5*\w, \h*3.5-\r)--(0.5\w, \h*3.1)--(1.5*\w, \h*3.1)--(1.5*\w,\h*2.5+\r);

\draw[black!50!green](\w*3.5,\h*2.5-\r)--(\w*3.5,\h*1.5+\r);
\draw[blue](\w*1.5,\h*2.5-\r)--(\w*1.5,\h*1.9)--(\w*2.5,\h*1.9)--(\w*2.5,\h*1.5+\r);
\draw[red](2.5*\w, \h*2.5-\r)--(2.5*\w,\h*2)--(4.5*\w,\h*2)--(4.5*\w,\h*1.5+\r);
\draw[red](4.5*\w, \h*2.5-\r)--(4.5\w, \h*2.1)--(5.5*\w, \h*2.1)--(5.5*\w,\h*1.5+\r);

\draw[blue] (\w*2.5,\h*1.5-\r)--(\w*2.5,\h*.5+\r);    
\draw[black!50!green] (\w*3.5,\h*1.5-\r)--(\w*3.5,\h*.5+\r);
\draw[red] (\w*4.5,\h*1.5-\r)--(\w*4.5,\h*.5+\r);
\draw[red] (\w*5.5,\h*1.5-\r)--(\w*5.5,\h*.5+\r);
\node at (3,-1) {$\mu=(5,4,3,3,1,1)$};
\node at (3,-2) {$\maj(\rho_N(M))=0$};


\cell{0}{10}{1};
\cell{0}{11}{1};
\cell{0}{12}{1};
\cell{0}{13}{1};
\cell{0}{14}{2};
\cell{0}{15}{4};

\cell{-1}{10}{2};
\cell{-1}{11}{2};
\cell{-1}{12}{2};
\cell{-1}{13}{3};

\cell{-2}{10}{3};
\cell{-2}{11}{3};
\cell{-2}{12}{3};
\cell{-2}{13}{5};

\cell{-3}{10}{4};
\cell{-3}{11}{4};

\cell{-4}{10}{5};

\end{scope}
\draw (29,-1) arc (-30:30:6);
\draw (12,-1) arc (210:150:6);
\node at (20.5,1) {\Huge ,};
\end{tikzpicture}
}

\end{center}

We also show the partially constructed pairs $(M_r,Q_r)$ in the collapsing procedure at each step in \cref{fig:collapsing steps}.
\begin{figure}[h]
    \centering
    \begin{tikzpicture}[scale=0.3]        
    \def \w{1};
        \def \h{1};
        \def \r{0.25};
        \node at (-1,6) {$r$};
        \node at (-1,2.5) {$M_r$};
        \node at (-1,-2) {$Q_r$};
        \foreach \i in {0,...,5}
        {
        \node at (2.5+7*\i,6) {\i};
        }
        \begin{scope}[xshift=0cm]
        \foreach \i in {0,...,5}
        {
        \draw[gray!50] (0,\i*\h)--(\w*6,\i*\h);
        }
        \foreach \i in {0,...,6}
        {
        \draw[gray!50] (\w*\i,0)--(\w*\i,5*\h);
        }
        \foreach \xx\yy in {0/0,2/0,3/0,4/0,1/1,2/1,3/1,5/1,1/2,2/2,3/2,4/2,0/3,3/3,5/3,2/4,4/4}
        {
        \fill[red!30] (\w*.5+\w*\xx,\h*.5+\h*\yy) circle (\r cm);
        }
        \node at (2.5,-2){\scalebox{0.6}{$\emptyset$}};
        \end{scope}
    
        
        \begin{scope}[xshift=7cm]
        \foreach \i in {0,...,5}
        {
        \draw[gray!50] (0,\i*\h)--(\w*6,\i*\h);
        }
        \foreach \i in {0,...,6}
        {
        \draw[gray!50] (\w*\i,0)--(\w*\i,5*\h);
        }
        \foreach \xx\yy in {1/1,2/1,3/1,5/1,1/2,2/2,3/2,4/2,0/3,3/3,5/3,2/4,4/4}
        {
        \fill[red!30] (\w*.5+\w*\xx,\h*.5+\h*\yy) circle (\r cm);
        }
        \foreach \xx\yy in {0/0,2/0,3/0,4/0}
        {
        \fill[black] (\w*.5+\w*\xx,\h*.5+\h*\yy) circle (\r cm);
        }
        \node at (2.5,-2) {\scalebox{0.6}{\tableau{1&1&1&1\\}}};
        \end{scope}
    
        
        \begin{scope}[xshift=14cm]
        \foreach \i in {0,...,5}
        {
        \draw[gray!50] (0,\i*\h)--(\w*6,\i*\h);
        }
        \foreach \i in {0,...,6}
        {
        \draw[gray!50] (\w*\i,0)--(\w*\i,5*\h);
        }
        \foreach \xx\yy in {1/2,2/2,3/2,4/2,0/3,3/3,5/3,2/4,4/4}
        {
        \fill[red!30] (\w*.5+\w*\xx,\h*.5+\h*\yy) circle (\r cm);
        }
        \foreach \xx\yy in {0/0,2/0,3/0,4/0,1/1,2/1,3/1,5/0}
        {
        \fill[black] (\w*.5+\w*\xx,\h*.5+\h*\yy) circle (\r cm);
        }
        \node at (2.5,-2) {\scalebox{0.6}{\tableau{2&2&2\\1&1&1&1&2}}};
        \draw[blue](1.5, 1.5-\r)--(1.5,0.8)--(2.5,0.8)--(2.5,0.5+\r);
        \draw[black!50!green](2.5, 1.5-\r)--(2.5,0.9)--(3.5,0.9)--(3.5,0.5+\r);
        \draw[blue](3.5, 1.5-\r)--(3.5,1)--(4.5,1)--(4.5,0.5+\r);
        \end{scope}
    
        
        \begin{scope}[xshift=21cm]
        \foreach \i in {0,...,5}
        {
        \draw[gray!50] (0,\i*\h)--(\w*6,\i*\h);
        }
        \foreach \i in {0,...,6}
        {
        \draw[gray!50] (\w*\i,0)--(\w*\i,5*\h);
        }
        \foreach \xx\yy in {0/3,3/3,5/3,2/4,4/4}
        {
        \fill[red!30] (\w*.5+\w*\xx,\h*.5+\h*\yy) circle (\r cm);
        }
        \foreach \xx\yy in {0/0,2/0,3/0,4/0,1/1,2/1,3/1,5/0,1/2,2/2,3/2,4/1}
        {
        \fill[black] (\w*.5+\w*\xx,\h*.5+\h*\yy) circle (\r cm);
        }
        
        \node at (2.5,-2.5) {\scalebox{0.6}{\tableau{3&3&3\\2&2&2&3\\1&1&1&1&2}}};
        \draw[blue](1.5, 1.5-\r)--(1.5,0.8)--(2.5,0.8)--(2.5,0.5+\r);
        \draw[black!50!green](2.5, 1.5-\r)--(2.5,0.9)--(3.5,0.9)--(3.5,0.5+\r);
        \draw[blue](3.5, 1.5-\r)--(3.5,1)--(4.5,1)--(4.5,0.5+\r);
        \draw[black!50!green](4.5, 1.5-\r)--(4.5,1.1)--(5.5,1.1)--(5.5,0.5+\r);
        \draw[blue](1.5,2.5-\r)--(1.5,1.5+\r);
        \draw[black!50!green](2.5,2.5-\r)--(2.5,1.5+\r);
        \draw[blue](3.5,2.5-\r)--(3.5,1.5+\r);
        \end{scope}
    
        
        \begin{scope}[xshift=28cm]
        \foreach \i in {0,...,5}
        {
        \draw[gray!50] (0,\i*\h)--(\w*6,\i*\h);
        }
        \foreach \i in {0,...,6}
        {
        \draw[gray!50] (\w*\i,0)--(\w*\i,5*\h);
        }
        \foreach \xx\yy in {2/4,4/4}
        {
        \fill[red!30] (\w*.5+\w*\xx,\h*.5+\h*\yy) circle (\r cm);
        }
        \foreach \xx\yy in {0/0,2/0,3/0,4/0,1/0,2/1,3/1,5/0,1/2,2/2,3/2,4/1,0/3,3/3,5/1}
        {
        \fill[black] (\w*.5+\w*\xx,\h*.5+\h*\yy) circle (\r cm);
        }
        \node at (2.5,-2.5) {\scalebox{0.6}{\tableau{4&4\\3&3&3\\2&2&2&3\\1&1&1&1&2&4}}};

        \draw[blue](2.5, 1.5-\r)--(2.5, 0.5+\r);
        \draw[blue](3.5, 1.5-\r)--(3.5, 0.5+\r);
        \draw[black!50!green](4.5, 1.5-\r)--(4.5, 0.5+\r);
        \draw[black!50!green](5.5, 1.5-\r)--(5.5, 0.5+\r);
        
        \draw[blue](1.5, 2.5-\r)--(1.5,1.8)--(2.5,1.8)--(2.5,1.5+\r);
        \draw[black!50!green](2.5, 2.5-\r)--(2.5,1.9)--(4.5,1.9)--(4.5,1.5+\r);
        \draw[blue](3.5, 2.5-\r)--(3.5,1.5+\r);

        \draw[blue](0.5,3.5-\r)--(0.5, 3)--(1.5,3)--(1.5,2.5+\r);
        \draw[blue](3.5,3.5-\r)--(3.5,2.5+\r);
        \end{scope}

        \begin{scope}[xshift=35cm]
        \foreach \i in {0,...,5}
        {
        \draw[gray!50] (0,\i*\h)--(\w*6,\i*\h);
        }
        \foreach \i in {0,...,6}
        {
        \draw[gray!50] (\w*\i,0)--(\w*\i,5*\h);
        }
        \foreach \xx\yy in {0/0,2/0,3/0,4/0,1/0,2/1,3/1,5/0,1/2,2/2,3/2,4/1,0/3,3/3,5/1,2/4,4/2}
        {
        \fill[black] (\w*.5+\w*\xx,\h*.5+\h*\yy) circle (\r cm);
        }
        \node at (2.5,-2.5) {\scalebox{0.6}{\tableau{5\\4&4\\3&3&3&5\\2&2&2&3\\1&1&1&1&2&4}}};

        \draw[blue](2.5, 1.5-\r)--(2.5, 0.5+\r);
        \draw[blue](3.5, 1.5-\r)--(3.5, 0.5+\r);
        \draw[black!50!green](4.5, 1.5-\r)--(4.5, 0.5+\r);
        \draw[black!50!green](5.5, 1.5-\r)--(5.5, 0.5+\r);
        
        \draw[blue](1.5, 2.5-\r)--(1.5,1.8)--(2.5,1.8)--(2.5,1.5+\r);
        \draw[black!50!green](2.5, 2.5-\r)--(2.5,1.9)--(4.5,1.9)--(4.5,1.5+\r);
        \draw[blue](3.5, 2.5-\r)--(3.5,1.5+\r);
        \draw[black!50!green](4.5, 2.5-\r)--(4.5,2.1)--(5.5,2.1)--(5.5,1.5+\r);

        \draw[blue](0.5,3.5-\r)--(0.5, 3)--(1.5,3)--(1.5,2.5+\r);
        \draw[blue](3.5,3.5-\r)--(3.5,2.5+\r);

        \draw[blue](2.5,4.5-\r)--(2.5,4)--(3.5,4)--(3.5,3.5+\r);
        \end{scope}
    \end{tikzpicture}
    \caption{Starting with $M_0=M$ in \cref{ex:collapsing}, we show the multiline queues $M_r$ (in black) along with the recording tableaux $Q_r$ at each step of the collapsing procedure of \cref{def:colmap}.}\label{fig:collapsing steps}
\end{figure}
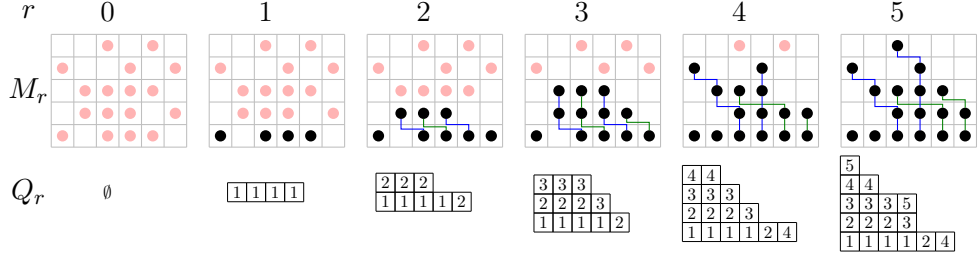

\end{example}

A key property of the row and column crystal operators is that they commute.
\begin{lemma}[{\cite{vanleeuwen2006double}, \cite{MV24}}]\label{prop:opcommute}
    Let $M$ be a multiline queue, $g^\leftrightarrow_i\in\{\eL_i,\fR_i\}$ be a column crystal operator, and let $\eD_j$ be a row crystal operator. Then
    \begin{equation*}
        g^\leftrightarrow_i\circ \eD_j(M) = \eD_j\circ g^\leftrightarrow_i(M).
    \end{equation*}
\end{lemma}
\noindent This implies two fundamental properties: (1) the column operators commute with $\rho_N$ and (2) the column operators preserve the major index (energy) of a multiline queue. 

\begin{prop}[{\cite[Proof of Lemma 5.10]{MV24}}]\label{prop:majpreserver}
    Let $M$ be a multiline queue, and let $g^\leftrightarrow_i$
    be a column crystal operator. Then 
    \[g^\leftrightarrow_i(\rho_N(M)) = \rho_N(g^\leftrightarrow_i(M))
    \]
    and 
    \[\maj(M) = \maj(g^\leftrightarrow_i(M)).
    \]
\end{prop}

This immediately implies two important properties that form the basis of our arguments.

\begin{cor}\label{cor:sequence}
    If $M,M'\in\MLQ_\lambda$ satisfy $\rho_Q(M)=\rho_Q(M')$, then $M$ and $M'$ are in the same connected component of $\MLQ_\lambda$, and hence there exists a sequence of operators $\sbf=g^\leftrightarrow_{i_\ell}\cdots g^\leftrightarrow_{i_1}$ such that $M'=\sbf M$.
\end{cor}

\begin{cor}\label{prop:num of balls preserved}
    Let $M$ be a multiline queue such that $\theta_i(rw(M))$ contains $a$ and $b$ unmatched $i$'s and $i+1$'s, respectively. Then $\theta_i(rw(\rho_N(M)))$ also contains $a$ and $b$ unmatched $i$'s and $i+1$'s.
\end{cor}

\begin{proof}
    There are $a$ unmatched $i$'s in $\theta_i(rw(M))$ if and only if $(\fR_i)^{k-1}(M) \neq (\fR_i)^{k}(M)$ for $1\leq k \leq a$ and $(\fR_i)^a(M) = (\fR_i)^{a+1}(M)$, with an analogous statement for $i+1$. The claim then follows from \cref{prop:majpreserver}.
\end{proof}

\noindent Since $\maj$ is preserved under the crystal operators, it induces a natural grading on the crystal graph.
\begin{definition}[Graded crystal graph]\label{def:graded crystal}
    Let $\MLQ_\lambda$ be the crystal of multiline queues of shape $\lambda$. We grade $\MLQ_\lambda$ by the $\maj$ statistic, with the homogeneous component with grading $k$ consisting of all multiline queues $M$ such that $\maj(M) = k$. In particular, $\NMLQ_\lambda$ is the 0-graded component of $\MLQ_\lambda$.
\end{definition}

For example, the graph in \cref{fig:big figure} is the $0$-graded component of $\MLQ_{(3,3,1,1)}$ for $n=4$.

\subsection{Schur positivity of $q$-Whittaker polynomials}

The $q$-Whittaker and Hall--Littlewood polynomials are related to the Schur basis via the Kostka--Foulkes polynomials $K_{\lambda,\mu}(q)$, as follows:
\begin{equation}
P_\lambda(X;q,0)=\sum_{\mu\leq\lambda}K_{\mu',\lambda'}(q)s_\mu(X)
\end{equation}
and
\begin{equation}
    s_\lambda(X)=\sum_{\mu\leq\lambda} K_{\lambda,\mu}(t)P_\mu(X;0,t).
\end{equation}

In their seminal paper \cite{LS78}, Lascoux and Sch\"utzenberger gave a formula for $K_{\lambda,\mu}(q)$ in terms of the \textbf{charge} statistic on semistandard tableaux.

\begin{theorem}[{\cite{LS78}}]\label{thm:LS}
    For partitions $\lambda,\mu$, the Kostka--Foulkes polynomial is given by
    \begin{equation}\label{eqn:KF_coeff}
        K_{\lambda,\mu}(q)=\sum_{T\in\SSYT(\lambda,\mu)}q^{\charge(T)}
    \end{equation}
    
\end{theorem}

Since our only use of charge is through the following theorem, we omit the definition and instead refer the reader to \cite[Chapter 5]{Lothaire_2002}. 

\begin{theorem}[{\cite[Theorem 4.23]{MV24}}]\label{theorem:collthm}
    The operator $\rho:\MLQ_\lambda\to \bigcup_{\mu\leq \lambda}\NMLQ_\mu\times \SSYT(\mu',\lambda')$ defined by:
    \begin{equation*}
        \rho(M) = (\rho_N(M),\rho_Q(M))
    \end{equation*}
    is a bijection such that $x^M = x^{\rho_N(M)}$ and $\maj(M) = \charge(Q)$.
\end{theorem}
\begin{example}
    Continuing \cref{ex:collapsing}, we have that $\rho_N(M)\in\NMLQ_\mu$ with the same content as $M$ and $\maj$ equal to 0, and $\rho_Q(M)\in\SSYT(\mu',\lambda')$ with $\charge(\rho_Q(M))=\maj(M)=4$. 
\end{example}

We prove \cref{thm:LS}, using collapsing on the multiline queue crystal. This will serve as a warm-up for our proof for the nonsymmetric and quasisymmetric cases, highlighting the main obstacle of a disconnected crystal in these settings.
  
The key property we use is that the collapsing operator $\rho$ is a composition of row crystal operators, which commute with the column crystal operators by \cref{prop:opcommute}. The numbers of balls that collapse at each step of the collapsing procedure  determine the entries of the recording tableau $\rho_Q(M)$. By the commutativity property, these numbers remain unchanged under any application of column crystal operators. Thus we obtain the following lemma.
\begin{lemma}[{Proof of \cite[Lemma 5.10]{MV24}}]\label{lem:charge}
    Let $M$ be a multiline queue, and let $g^\leftrightarrow_i\in\{\eL_i,\fR_i\}$ be a column crystal operator. Then we have:
    \begin{equation*}
        \rho_Q(g^\leftrightarrow_i(M)) = \rho_Q(M)
    \end{equation*}
    In particular, $\charge(\rho_Q(g^\leftrightarrow_i(M))=\charge(\rho_Q(M))$.
\end{lemma}
\begin{remark}\label{rem:connected}
    For a partition $\mu$, the bijection \cite[Theorem A.4]{MV24} between $\SSYT(\mu)$ and $\NMLQ_\mu$ is a crystal isomorphism. In particular, the crystal graph of $\NMLQ_\mu$ is connected. 
\end{remark}

From \cref{theorem:collthm}, we get that each connected component of $\MLQ_\lambda$ is associated to a recording tableau given by the image of $\rho_Q$ acting on that component; we may call this the \emph{fiber} of $\rho_Q$ corresponding to that recording tableau. The set of connected components is therefore indexed by the semistandard tableaux of content $\lambda'$.

From \cref{theorem:collthm} and \cref{lem:charge}, we have that for any $N_1, N_2\in\NMLQ_\mu$,
\begin{equation}\label{eq:iff}
        \rho^{-1}(N_1,Q)\in \MLQ_\lambda\iff \rho^{-1}(N_2,Q)\in \MLQ_\lambda\,.
\end{equation}
Combining \eqref{eq:iff} with \cref{rem:connected}, we obtain an elementary proof for \cref{thm:LS}. Recall the expression for the $q$-Kostka polynomials in \eqref{eqn:KF_coeff}, from which we have:
\begin{align*}
    \sum_\mu K_{\mu',\lambda'}(q)s_\mu &= \sum_\mu \left(\sum_{Q\in \SSYT(\mu', \lambda')}q^{\charge(Q)}\right)\left(\sum_{N\in \NMLQ_\mu}x^N\right)\\
    &=\sum_\mu \sum_{\substack{Q\in \SSYT(\mu', \lambda'),\\N\in \text{NMLQ}_\mu}}q^{\charge(Q)}x^N\\
    &=\sum_\mu \sum_{\substack{Q\in \SSYT(\mu', \lambda'),\\N\in \text{NMLQ}_\mu}}q^{\charge(Q)}x^{\rho^{-1}(N,Q)}\\
    &= \sum_{M\in \text{MLQ}_\lambda}q^{\maj(M)}x^M = P_\lambda(X;q,0)\,  
\end{align*}
A multiline queue formula for $K_{\mu',\lambda'}$ is then obtained in terms of the pre-image of $\rho$. Recall that for a partition $\mu$, $M_\mu$ is the unique multiline queue with $\mu_i$ balls in column $i$; this will serve as a representative of $\NMLQ_\mu$.
\begin{align*}
    K_{\mu',\lambda'}(q) &= \sum_{Q\in \SSYT(\mu',\lambda')} q^{\charge(Q)}\\
    &= \sum_{Q\in \SSYT(\mu',\lambda')}q^{\maj(\rho^{-1}(M_\mu, Q))}
    &= \sum_{\substack{N\in \text{MLQ}_\lambda\\\rho_N(N) = M_\mu}}q^{\maj(N)}\,.
\end{align*}

Our goal now is to use the collapsing map in a similar way to prove the Demazure atom-positivity of the ASEP polynomials and the quasi-Schur positivity of the quasisymmetric Macdonald polynomials at $t=0$.

To show the Demazure atom-positivity of ASEP polynomials at $t=0$, we restrict the collapsing map to $\MLQ(\alpha)$ with $\sort(\alpha) = \lambda$. Here, we would like to show that the image of the collapsing map decomposes properly, that is, if some multiline queue from a particular nonsymmetric component shows up in the image, then the entire nonsymmetric component is in the image. Stating this precisely, we show that there exist sets $S_{\alpha,\beta}\subseteq \SSYT(\sort(\beta)',\sort(\alpha)')$ such that the collapsing map, when restricted to $\MLQ(\alpha)$, gives the following bijection:
\[\MLQ(\alpha) \simeq \bigcup_\beta \NMLQ(\beta) \times S_{\alpha,\beta}\,.\]
This is equivalent to showing the following theorem:

\begin{restatable}{theorem}{weaktype}\label{theorem:weaktype}
Let $\lambda$ be a partition and let $M, M' \in\MLQ_\lambda$ be multiline queues in the same connected component of the graded crystal graph (such that $\rho_Q(M)=\rho_Q(M')$). Then, 
            \begin{equation}\label{eq:toshow}
                    \type(\rho(M)) = \type(\rho(M'))  \implies \type(M) = \type(M').
            \end{equation}    
\end{restatable}

We will prove \cref{theorem:weaktype} in \cref{sec:proof of main}. We also derive from \cref{theorem:weaktype} the quasisymmetric analogue as seen in the following proposition:

\begin{restatable}{prop}{sameQ}\label{prop:sameQ}
    Let $\lambda$ be a partition and let $M, M' \in\MLQ_\lambda$ be multiline queues in the same connected component of the graded crystal graph (such that $\rho_Q(M)=\rho_Q(M')$). Then, 
            \begin{equation}\label{eq:toshow qsym}
                    \strtype(\rho(M)) = \strtype(\rho(M'))  \implies \strtype(M) = \strtype(M').
            \end{equation}   
\end{restatable}

We sketch the proof of \cref{theorem:weaktype}, which will be the focus of \cref{sec:induction}. Write $N,N'=\rho(M),\rho(M')\in\NMLQ_\mu$. First, as $M,M'$ are in the same connected component of $\MLQ_\lambda$, there exists a sequence of column operators $\sbf=g^\leftrightarrow_{j_m}\cdots g^\leftrightarrow_{j_1}$ such that $M' = \sbf\,M$. As $\rho$ commutes with $\sbf$, we also have $N'=\sbf\,N$. To parallel \eqref{eq:iff} for the symmetric case, we must show \eqref{eq:toshow}, meaning that if $N$ and $N'$ lie in the same nonsymmetric component $\NMLQ[\alpha]$ of $\NMLQ_\mu$, then $M$ and $M'$ lie in the same nonsymmetric component of $\MLQ_\lambda$. However, there are two fundamental obstructions.

First, crystal operators do not preserve nonsymmetric type. As a result, along the path $g^\leftrightarrow_{j_1},\ldots,g^\leftrightarrow_{j_m}$ from $N$ to $N'$ in $\NMLQ_\mu$, even if both endpoints $N$ and $N'$ lie in $\NMLQ[\alpha]$, the intermediate vertices may be outside of this set.

Second, it is not possible to circumvent this by choosing a different path from $N$ to $N'$, since the restriction of $\NMLQ_\mu$ to $\NMLQ[\alpha]$ in general produces a \emph{disconnected graph}. Thus there need not exist any path between $N$ and $N'$ that is entirely contained within $\NMLQ[\alpha]$,  and in general, any sequence of multiline queues along a path from $N\in\NMLQ[\alpha]$ to $N'\in\NMLQ[\alpha]$ may be forced to pass through different nonsymmetric components before returning to $\NMLQ[\alpha]$. See \cref{fig:big figure} for such an example. 

Thus it does not seem possible to construct an argument based on preserving types or global connectivity. Instead, we shall construct a special path whose behavior with respect to the nonsymmetric components it passes through can be analyzed directly.  Our approach is to consider a particular path from $N$ to $M_\alpha$, the unique straight multiline queue of type $\alpha$, and to simultaneously track its preimage under $\rho$. In particular, this path passes through the lowest weight element of $\NMLQ_\mu$. In \cref{sec:crystals}, we analyze the action of a single column operator on a multiline queue and on its image under $\rho$, using the combinatorial characterization of \cref{cor:Harper's lemma} that determines exactly when the column crystal operator changes the nonsymmetric type. This local analysis will allow us to track how type evolves along the path in order to conclude our results.

We illustrate \cref{theorem:weaktype} and the description above with \cref{ex:path}.

   \begin{example}\label{ex:path}
        Let $M,M'\in\MLQ[(0,5,3,2)]$ and $N,N'\in\NMLQ[(1,4,3,2])$ for 
        $N=\rho(M)$, $N'=\rho(M')$, and \[\rho_Q(M)=\rho_Q(M')=\begin{ytableau}
        4\\3&5\\2&2&2\\1&1&1&3
    \end{ytableau}\,.\] 
      The sequence of column operators $\eL_2\fR_1(\fR_2)^2$ from $M$ to $M'$ passes through components of types $(0,5,3,2) \to (0,3,5,2) \to (0,5,3,2)$, as shown below. By \cref{prop:opcommute}, the same sequence takes $N$ to $N'$, traversing the types $(1,4,3,2) \to (1,3,4,2) \to (1,4,3,2)$. 
       
        \begin{center}
        \resizebox{1\linewidth}{!}
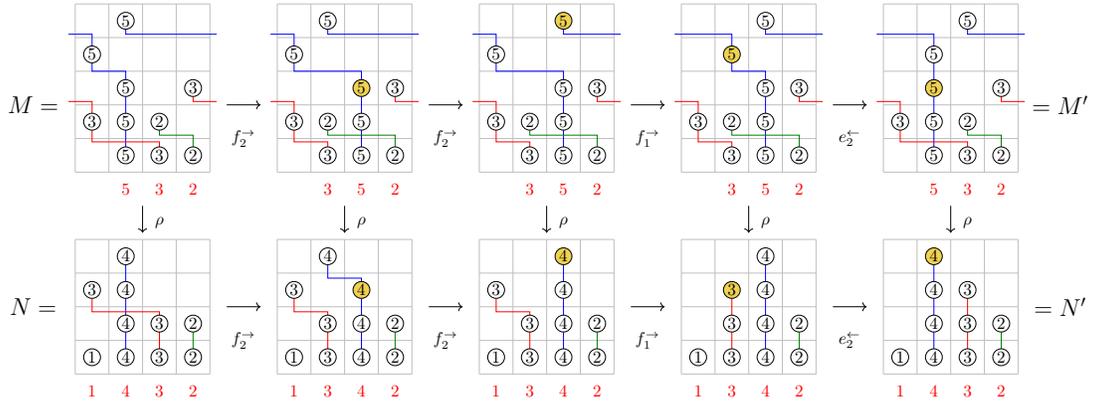
{
        \begin{tikzpicture}[scale=0.7]
            \def \w{1};
        \def \h{1};
        \def \r{0.25};
        \def \d{6};
        
        \begin{scope}[xshift=-\d cm]
        \node at (-1.25, 2) {\Large $N=$};
        \foreach \i in {0,...,4}
        {
        \draw[gray!50] (0,\i*\h)--(\w*4,\i*\h);
        }
        \foreach \i in {0,...,4}
        {
        \draw[gray!50] (\w*\i,0)--(\w*\i,4*\h);
        }

        \draw[blue] (\w*1.5,\h*3.5-\r)--(\w*1.5,\h*.5+\r);
        \draw[red] (\w*.5,\h*2.5-\r)--(\w*.5,\h*1.85)--(\w*2.5,\h*1.85)--(\w*2.5,\h*.5+\r);
        \draw[black!50!green] (\w*3.5,\h*1.5-\r)--(\w*3.5,\h*.5+\r);

        \draw[fill=white] (\w*1.5,\h*3.5) circle (\r cm) node[color=black] {4};
        \draw[fill=white] (\w*0.5,\h*2.5) circle (\r cm) node[color=black] {3};
        \draw[fill=white] (\w*1.5,\h*2.5) circle (\r cm) node[color=black] {4};
        \draw[fill=white] (\w*1.5,\h*1.5) circle (\r cm) node[color=black] {4};
        \draw[fill=white] (\w*2.5,\h*1.5) circle (\r cm) node[color=black] {3};
        \draw[fill=white] (\w*3.5,\h*1.5) circle (\r cm) node[color=black] {2};
        \draw[fill=white] (\w*0.5,\h*.5) circle (\r cm) node[color=black] {1};
        \draw[fill=white] (\w*1.5,\h*.5) circle (\r cm) node[color=black] {4};
        \draw[fill=white] (\w*2.5,\h*.5) circle (\r cm) node[color=black] {3};
        \draw[fill=white] (\w*3.5,\h*.5) circle (\r cm) node[color=black] {2};
        
        \node[red] at (\w*0.5,-.5*\h) {1};
        \node[red] at (\w*1.5,-.5*\h) {4};
        \node[red] at (\w*2.5,-.5*\h){3};
        \node[red] at (\w*3.5,-.5*\h) {2};
        \end{scope}

        \begin{scope}[xshift=0cm]
        \draw[->] (-1.5,2)--(-.5,2);
        \node at (-1,1) {$\fR_2$};
        
        \foreach \i in {0,...,4}
        {
        \draw[gray!50] (0,\i*\h)--(\w*4,\i*\h);
        }     
        \foreach \i in {0,...,4}
        {
        \draw[gray!50] (\w*\i,0)--(\w*\i,4*\h);
        }

        \draw[blue] (\w*1.5,\h*3.5-\r)--(\w*1.5,\h*2.85)--(\w*2.5,\h*2.85)--(\w*2.5,\h*2.5+\r) (\w*2.5,\h*2.5-\r)--(\w*2.5,\h*.5+\r);
        \draw[red] (\w*.5,\h*2.5-\r)--(\w*.5,\h*1.85)--(\w*1.5,\h*1.85)--(\w*1.5,\h*.5+\r);
        \draw[black!50!green] (\w*3.5,\h*1.5-\r)--(\w*3.5,\h*.5+\r);

        \draw[fill=white] (\w*1.5,\h*3.5) circle (\r cm) node[color=black] {4};
        \draw[fill=white] (\w*0.5,\h*2.5) circle (\r cm) node[color=black] {3};
        \draw[fill=\highlight] (\w*2.5,\h*2.5) circle (\r cm) node[color=black] {4};
        \draw[fill=white] (\w*1.5,\h*1.5) circle (\r cm) node[color=black] {3};
        \draw[fill=white] (\w*2.5,\h*1.5) circle (\r cm) node[color=black] {4};
        \draw[fill=white] (\w*3.5,\h*1.5) circle (\r cm) node[color=black] {2};
        \draw[fill=white] (\w*0.5,\h*.5) circle (\r cm) node[color=black] {1};
        \draw[fill=white] (\w*1.5,\h*.5) circle (\r cm) node[color=black] {3};
        \draw[fill=white] (\w*2.5,\h*.5) circle (\r cm) node[color=black] {4};
        \draw[fill=white] (\w*3.5,\h*.5) circle (\r cm) node[color=black] {2};
        
        \node[red] at (\w*0.5,-.5*\h) {1};
        \node[red] at (\w*1.5,-.5*\h) {3};
        \node[red] at (\w*2.5,-.5*\h){4};
        \node[red] at (\w*3.5,-.5*\h) {2};
        \end{scope}

        \begin{scope}[xshift=\d cm]
        \draw[->] (-1.5,2)--(-.5,2);
        \node at (-1,1) {$\fR_2$};
        
        \foreach \i in {0,...,4}
        {
        \draw[gray!50] (0,\i*\h)--(\w*4,\i*\h);
        }
        \foreach \i in {0,...,4}
        {
        \draw[gray!50] (\w*\i,0)--(\w*\i,4*\h);
        }

        \draw[blue] (\w*2.5,\h*3.5-\r)--(\w*2.5,\h*.5+\r);
        \draw[red] (\w*.5,\h*2.5-\r)--(\w*.5,\h*1.85)--(\w*1.5,\h*1.85)--(\w*1.5,\h*.5+\r);
        \draw[black!50!green] (\w*3.5,\h*1.5-\r)--(\w*3.5,\h*.5+\r);

        \draw[fill=\highlight] (\w*2.5,\h*3.5) circle (\r cm) node[color=black] {4};
        \draw[fill=white] (\w*0.5,\h*2.5) circle (\r cm) node[color=black] {3};
        \draw[fill=white] (\w*2.5,\h*2.5) circle (\r cm) node[color=black] {4};
        \draw[fill=white] (\w*1.5,\h*1.5) circle (\r cm) node[color=black] {3};
        \draw[fill=white] (\w*2.5,\h*1.5) circle (\r cm) node[color=black] {4};
        \draw[fill=white] (\w*3.5,\h*1.5) circle (\r cm) node[color=black] {2};
        \draw[fill=white] (\w*0.5,\h*.5) circle (\r cm) node[color=black] {1};
        \draw[fill=white] (\w*1.5,\h*.5) circle (\r cm) node[color=black] {3};
        \draw[fill=white] (\w*2.5,\h*.5) circle (\r cm) node[color=black] {4};
        \draw[fill=white] (\w*3.5,\h*.5) circle (\r cm) node[color=black] {2};
        
        \node[red] at (\w*0.5,-.5*\h) {1};
        \node[red] at (\w*1.5,-.5*\h) {3};
        \node[red] at (\w*2.5,-.5*\h){4};
        \node[red] at (\w*3.5,-.5*\h) {2};
        \end{scope}

        \begin{scope}[xshift=2*\d cm]
        \draw[->] (-1.5,2)--(-.5,2);
        \node at (-1,1) {$\fR_1$};
        \foreach \i in {0,...,4}
        {
        \draw[gray!50] (0,\i*\h)--(\w*4,\i*\h);
        }
        \foreach \i in {0,...,4}
        {
        \draw[gray!50] (\w*\i,0)--(\w*\i,4*\h);
        }
        
        \draw[blue] (\w*2.5,\h*3.5-\r)--(\w*2.5,\h*.5+\r);
        \draw[red] (\w*1.5,\h*2.5-\r)--(\w*1.5,\h*.5+\r);
        
        \draw[black!50!green] (\w*3.5,\h*1.5-\r)--(\w*3.5,\h*.5+\r);

        \draw[fill=white] (\w*2.5,\h*3.5) circle (\r cm) node[color=black] {4};
        \draw[fill=\highlight] (\w*1.5,\h*2.5) circle (\r cm) node[color=black] {3};
        \draw[fill=white] (\w*2.5,\h*2.5) circle (\r cm) node[color=black] {4};
        \draw[fill=white] (\w*1.5,\h*1.5) circle (\r cm) node[color=black] {3};
        \draw[fill=white] (\w*2.5,\h*1.5) circle (\r cm) node[color=black] {4};
        \draw[fill=white] (\w*3.5,\h*1.5) circle (\r cm) node[color=black] {2};
        \draw[fill=white] (\w*0.5,\h*.5) circle (\r cm) node[color=black] {1};
        \draw[fill=white] (\w*1.5,\h*.5) circle (\r cm) node[color=black] {3};
        \draw[fill=white] (\w*2.5,\h*.5) circle (\r cm) node[color=black] {4};
        \draw[fill=white] (\w*3.5,\h*.5) circle (\r cm) node[color=black] {2};
        
        \node[red] at (\w*0.5,-.5*\h) {1};
        \node[red] at (\w*1.5,-.5*\h) {3};
        \node[red] at (\w*2.5,-.5*\h){4};
        \node[red] at (\w*3.5,-.5*\h) {2};
        \end{scope}

        \begin{scope}[xshift=3*\d cm]
        \node at (5.25, 2) {\Large $=N'$};
        \draw[->] (-1.5,2)--(-.5,2);
        \node at (-1,1) {$\eL_2$};
        \foreach \i in {0,...,4}
        {
        \draw[gray!50] (0,\i*\h)--(\w*4,\i*\h);
        }
        \foreach \i in {0,...,4}
        {
        \draw[gray!50] (\w*\i,0)--(\w*\i,4*\h);
        }
        
        \draw[blue] (\w*1.5,\h*3.5-\r)--(\w*1.5,\h*.5+\r);
        \draw[red] (\w*2.5,\h*2.5-\r)--(\w*2.5,\h*.5+\r);
        
        \draw[black!50!green] (\w*3.5,\h*1.5-\r)--(\w*3.5,\h*.5+\r);

        \draw[fill=\highlight] (\w*1.5,\h*3.5) circle (\r cm) node[color=black] {4};
        \draw[fill=white] (\w*2.5,\h*2.5) circle (\r cm) node[color=black] {3};
        \draw[fill=white] (\w*1.5,\h*2.5) circle (\r cm) node[color=black] {4};
        \draw[fill=white] (\w*2.5,\h*1.5) circle (\r cm) node[color=black] {3};
        \draw[fill=white] (\w*1.5,\h*1.5) circle (\r cm) node[color=black] {4};
        \draw[fill=white] (\w*3.5,\h*1.5) circle (\r cm) node[color=black] {2};
        \draw[fill=white] (\w*0.5,\h*.5) circle (\r cm) node[color=black] {1};
        \draw[fill=white] (\w*2.5,\h*.5) circle (\r cm) node[color=black] {3};
        \draw[fill=white] (\w*1.5,\h*.5) circle (\r cm) node[color=black] {4};
        \draw[fill=white] (\w*3.5,\h*.5) circle (\r cm) node[color=black] {2};
        
        \node[red] at (\w*0.5,-.5*\h) {1};
        \node[red] at (\w*1.5,-.5*\h) {4};
        \node[red] at (\w*2.5,-.5*\h){3};
        \node[red] at (\w*3.5,-.5*\h) {2};
        \end{scope}

        \begin{scope}[xshift=-\d cm,yshift=6cm]
        \node at (-1.25, 2) {\Large $M=$};
        \foreach \i in {0,...,5}
        {
        \draw[gray!50] (0,\i*\h)--(\w*4,\i*\h);
        }
        \foreach \i in {0,...,4}
        {
        \draw[gray!50] (\w*\i,0)--(\w*\i,5*\h);
        }

        \draw[blue] (\w*1.5,\h*4.5-\r)--(\w*1.5,\h*4.1)--(\w*4.2,\h*4.1) (-.2*\w,\h*4.1)--(\w*.5,\h*4.1)--(\w*.5,\h*3)--(\w*1.5,\h*3)--(\w*1.5,\h*.5+\r);
        \draw[red]  (\w*3.5,\h*2.5-\r)--(\w*3.5,\h*2.1)--(\w*4.2,\h*2.1) (-.2*\w,\h*2.1)--(\w*.5,\h*2.1)--(\w*.5,\h*.9)--(\w*2.5,\h*.9)--(\w*2.5,\h*.5+\r);
        \draw[black!50!green] (\w*2.5,\h*1.5-\r)--(\w*2.5,\h*1.1)--(\w*3.5,\h*1.1)--(\w*3.5,\h*.5+\r);

        \draw[fill=white] (\w*1.5,\h*4.5) circle (\r cm) node[color=black] {5};
        \draw[fill=white] (\w*0.5,\h*3.5) circle (\r cm) node[color=black] {5};
        \draw[fill=white] (\w*1.5,\h*2.5) circle (\r cm) node[color=black] {5};
        \draw[fill=white] (\w*3.5,\h*2.5) circle (\r cm) node[color=black] {3};
        \draw[fill=white] (\w*0.5,\h*1.5) circle (\r cm) node[color=black] {3};
        \draw[fill=white] (\w*1.5,\h*1.5) circle (\r cm) node[color=black] {5};
        \draw[fill=white] (\w*2.5,\h*1.5) circle (\r cm) node[color=black] {2};
        \draw[fill=white] (\w*1.5,\h*.5) circle (\r cm) node[color=black] {5};
        \draw[fill=white] (\w*2.5,\h*.5) circle (\r cm) node[color=black] {3};
        \draw[fill=white] (\w*3.5,\h*.5) circle (\r cm) node[color=black] {2};
        
        \node[red] at (\w*1.5,-.5*\h) {5};
        \node[red] at (\w*2.5,-.5*\h){3};
        \node[red] at (\w*3.5,-.5*\h) {2};
        
        \draw[->] (2*\w,-1)--(2*\w,-1.8);
        \node at (2.5*\w,-1.5) {$\rho$};
        \end{scope}

         \begin{scope}[xshift=0cm,yshift=6cm]
        \draw[->] (-1.5,2)--(-.5,2);
        \node at (-1,1) {$\fR_2$};
        
        \foreach \i in {0,...,5}
        {
        \draw[gray!50] (0,\i*\h)--(\w*4,\i*\h);
        }
        \foreach \i in {0,...,4}
        {
        \draw[gray!50] (\w*\i,0)--(\w*\i,5*\h);
        }

        \draw[blue] (\w*1.5,\h*4.5-\r)--(\w*1.5,\h*4.1)--(\w*4.2,\h*4.1) (-.2*\w,\h*4.1)--(\w*.5,\h*4.1)--(\w*.5,\h*3)--(\w*2.5,\h*3)--(\w*2.5,\h*.5+\r);
        \draw[red]  (\w*3.5,\h*2.5-\r)--(\w*3.5,\h*2.1)--(\w*4.2,\h*2.1) (-.2*\w,\h*2.1)--(\w*.5,\h*2.1)--(\w*.5,\h*.9)--(\w*1.5,\h*.9)--(\w*1.5,\h*.5+\r);
        \draw[black!50!green] (\w*1.5,\h*1.5-\r)--(\w*1.5,\h*1.1)--(\w*3.5,\h*1.1)--(\w*3.5,\h*.5+\r);

        \draw[fill=white] (\w*1.5,\h*4.5) circle (\r cm) node[color=black] {5};
        \draw[fill=white] (\w*0.5,\h*3.5) circle (\r cm) node[color=black] {5};
        \draw[fill=\highlight] (\w*2.5,\h*2.5) circle (\r cm) node[color=black] {5};
        \draw[fill=white] (\w*3.5,\h*2.5) circle (\r cm) node[color=black] {3};
        \draw[fill=white] (\w*0.5,\h*1.5) circle (\r cm) node[color=black] {3};
        \draw[fill=white] (\w*1.5,\h*1.5) circle (\r cm) node[color=black] {2};
        \draw[fill=white] (\w*2.5,\h*1.5) circle (\r cm) node[color=black] {5};
        \draw[fill=white] (\w*1.5,\h*.5) circle (\r cm) node[color=black] {3};
        \draw[fill=white] (\w*2.5,\h*.5) circle (\r cm) node[color=black] {5};
        \draw[fill=white] (\w*3.5,\h*.5) circle (\r cm) node[color=black] {2};
        
        \node[red] at (\w*1.5,-.5*\h) {3};
        \node[red] at (\w*2.5,-.5*\h){5};
        \node[red] at (\w*3.5,-.5*\h) {2};
        
        \draw[->] (2*\w,-1)--(2*\w,-1.8);
        \node at (2.5*\w,-1.5) {$\rho$};
        \end{scope}
        
        \begin{scope}[xshift=\d cm,yshift=6cm]
        \draw[->] (-1.5,2)--(-.5,2);
        \node at (-1,1) {$\fR_2$};
        
        \foreach \i in {0,...,5}
        {
        \draw[gray!50] (0,\i*\h)--(\w*4,\i*\h);
        }
        \foreach \i in {0,...,4}
        {
        \draw[gray!50] (\w*\i,0)--(\w*\i,5*\h);
        }

        \draw[blue] (\w*2.5,\h*4.5-\r)--(\w*2.5,\h*4.1)--(\w*4.2,\h*4.1) (-.2*\w,\h*4.1)--(\w*.5,\h*4.1)--(\w*.5,\h*3)--(\w*2.5,\h*3)--(\w*2.5,\h*.5+\r);
        \draw[red]  (\w*3.5,\h*2.5-\r)--(\w*3.5,\h*2.1)--(\w*4.2,\h*2.1) (-.2*\w,\h*2.1)--(\w*.5,\h*2.1)--(\w*.5,\h*.9)--(\w*1.5,\h*.9)--(\w*1.5,\h*.5+\r);
        \draw[black!50!green] (\w*1.5,\h*1.5-\r)--(\w*1.5,\h*1.1)--(\w*3.5,\h*1.1)--(\w*3.5,\h*.5+\r);

        \draw[fill=\highlight] (\w*2.5,\h*4.5) circle (\r cm) node[color=black] {5};
        \draw[fill=white] (\w*0.5,\h*3.5) circle (\r cm) node[color=black] {5};
        \draw[fill=white] (\w*2.5,\h*2.5) circle (\r cm) node[color=black] {5};
        \draw[fill=white] (\w*3.5,\h*2.5) circle (\r cm) node[color=black] {3};
        \draw[fill=white] (\w*0.5,\h*1.5) circle (\r cm) node[color=black] {3};
        \draw[fill=white] (\w*1.5,\h*1.5) circle (\r cm) node[color=black] {2};
        \draw[fill=white] (\w*2.5,\h*1.5) circle (\r cm) node[color=black] {5};
        \draw[fill=white] (\w*1.5,\h*.5) circle (\r cm) node[color=black] {3};
        \draw[fill=white] (\w*2.5,\h*.5) circle (\r cm) node[color=black] {5};
        \draw[fill=white] (\w*3.5,\h*.5) circle (\r cm) node[color=black] {2};
        
        \node[red] at (\w*1.5,-.5*\h) {3};
        \node[red] at (\w*2.5,-.5*\h){5};
        \node[red] at (\w*3.5,-.5*\h) {2};
        
        \draw[->] (2*\w,-1)--(2*\w,-1.8);
        \node at (2.5*\w,-1.5) {$\rho$};
        \end{scope}
        
        \begin{scope}[xshift=2*\d cm,yshift=6cm]
        \draw[->] (-1.5,2)--(-.5,2);
        \node at (-1,1) {$\fR_1$};
        
        \foreach \i in {0,...,5}
        {
        \draw[gray!50] (0,\i*\h)--(\w*4,\i*\h);
        }
        \foreach \i in {0,...,4}
        {
        \draw[gray!50] (\w*\i,0)--(\w*\i,5*\h);
        }

        \draw[blue] (\w*2.5,\h*4.5-\r)--(\w*2.5,\h*4.1)--(\w*4.2,\h*4.1) (-.2*\w,\h*4.1)--(\w*1.5,\h*4.1)--(\w*1.5,\h*3)--(\w*2.5,\h*3)--(\w*2.5,\h*.5+\r);
        \draw[red]  (\w*3.5,\h*2.5-\r)--(\w*3.5,\h*2.1)--(\w*4.2,\h*2.1) (-.2*\w,\h*2.1)--(\w*.5,\h*2.1)--(\w*.5,\h*.9)--(\w*1.5,\h*.9)--(\w*1.5,\h*.5+\r);
        \draw[black!50!green] (\w*1.5,\h*1.5-\r)--(\w*1.5,\h*1.1)--(\w*3.5,\h*1.1)--(\w*3.5,\h*.5+\r);

        \draw[fill=white] (\w*2.5,\h*4.5) circle (\r cm) node[color=black] {5};
        \draw[fill=\highlight] (\w*1.5,\h*3.5) circle (\r cm) node[color=black] {5};
        \draw[fill=white] (\w*2.5,\h*2.5) circle (\r cm) node[color=black] {5};
        \draw[fill=white] (\w*3.5,\h*2.5) circle (\r cm) node[color=black] {3};
        \draw[fill=white] (\w*0.5,\h*1.5) circle (\r cm) node[color=black] {3};
        \draw[fill=white] (\w*1.5,\h*1.5) circle (\r cm) node[color=black] {2};
        \draw[fill=white] (\w*2.5,\h*1.5) circle (\r cm) node[color=black] {5};
        \draw[fill=white] (\w*1.5,\h*.5) circle (\r cm) node[color=black] {3};
        \draw[fill=white] (\w*2.5,\h*.5) circle (\r cm) node[color=black] {5};
        \draw[fill=white] (\w*3.5,\h*.5) circle (\r cm) node[color=black] {2};
        
        \node[red] at (\w*1.5,-.5*\h) {3};
        \node[red] at (\w*2.5,-.5*\h){5};
        \node[red] at (\w*3.5,-.5*\h) {2};
        
        \draw[->] (2*\w,-1)--(2*\w,-1.8);
        \node at (2.5*\w,-1.5) {$\rho$};
        \end{scope}

        \begin{scope}[xshift=3*\d cm,yshift=6cm]
        \node at (5.25, 2) {\Large $=M'$};
        \draw[->] (-1.5,2)--(-.5,2);
        \node at (-1,1) {$\eL_2$};
        
        \foreach \i in {0,...,5}
        {
        \draw[gray!50] (0,\i*\h)--(\w*4,\i*\h);
        }
        \foreach \i in {0,...,4}
        {
        \draw[gray!50] (\w*\i,0)--(\w*\i,5*\h);
        }

        \draw[blue] (\w*2.5,\h*4.5-\r)--(\w*2.5,\h*4.1)--(\w*4.2,\h*4.1) (-.2*\w,\h*4.1)--(\w*1.5,\h*4.1)--(\w*1.5,\h*.5+\r);
        \draw[red]  (\w*3.5,\h*2.5-\r)--(\w*3.5,\h*2.1)--(\w*4.2,\h*2.1) (-.2*\w,\h*2.1)--(\w*.5,\h*2.1)--(\w*.5,\h*.9)--(\w*2.5,\h*.9)--(\w*2.5,\h*.5+\r);
        \draw[black!50!green] (\w*2.5,\h*1.5-\r)--(\w*2.5,\h*1.1)--(\w*3.5,\h*1.1)--(\w*3.5,\h*.5+\r);

        \draw[fill=white] (\w*2.5,\h*4.5) circle (\r cm) node[color=black] {5};
        \draw[fill=white] (\w*1.5,\h*3.5) circle (\r cm) node[color=black] {5};
        \draw[fill=\highlight] (\w*1.5,\h*2.5) circle (\r cm) node[color=black] {5};
        \draw[fill=white] (\w*3.5,\h*2.5) circle (\r cm) node[color=black] {3};
        \draw[fill=white] (\w*0.5,\h*1.5) circle (\r cm) node[color=black] {3};
        \draw[fill=white] (\w*1.5,\h*1.5) circle (\r cm) node[color=black] {5};
        \draw[fill=white] (\w*2.5,\h*1.5) circle (\r cm) node[color=black] {2};
        \draw[fill=white] (\w*1.5,\h*.5) circle (\r cm) node[color=black] {5};
        \draw[fill=white] (\w*2.5,\h*.5) circle (\r cm) node[color=black] {3};
        \draw[fill=white] (\w*3.5,\h*.5) circle (\r cm) node[color=black] {2};
        
        \node[red] at (\w*1.5,-.5*\h) {5};
        \node[red] at (\w*2.5,-.5*\h){3};
        \node[red] at (\w*3.5,-.5*\h) {2};
        
        \draw[->] (2*\w,-1)--(2*\w,-1.8);
        \node at (2.5*\w,-1.5) {$\rho$};
        \end{scope}
        \end{tikzpicture}
        }
        \end{center}
    \end{example}

    \begin{remark}
         An observant reader may find that in \cref{ex:path}, there exists a sequence of column operators from $N$ to $N'$ that stays within the subgraph $\NMLQ[(1,4,3,2)]$ (for instance, $(\fR_1)^2\fR_2\eL_1$ is such a sequence). However, such a sequence is not guaranteed to exist in general. The smallest example of two nonwrapping multiline queues of the same type $\alpha$ lying in distinct connected components of $\NMLQ[\alpha]$ occurs in $\NMLQ_{(3,3,1,1)}$ for $\alpha=(1,3,1,3)$. As shown in \cref{fig:big figure},  the component $\NMLQ[(1,3,1,3)]$ is  disconnected.  Note that for nonwrapping multiline queues, each component of $\NMLQ[\alpha]$ is supported on exactly as many columns as the number of parts of $\alpha$; however, this is no longer the case for $\SNMLQ[\gamma]$. Thus, to exhibit a disconnected quasisymmetric component, one would need to consider  $\NMLQ_{(3,3,1,1)}$ on eight columns, which is a graph on 3135 vertices. One may verify with computer code that the component $\SNMLQ[(1,3,1,3)]$ is disconnected.
    \end{remark}

    \begin{figure}[htp]
        \centering
        \resizebox{\linewidth}{!}{\input{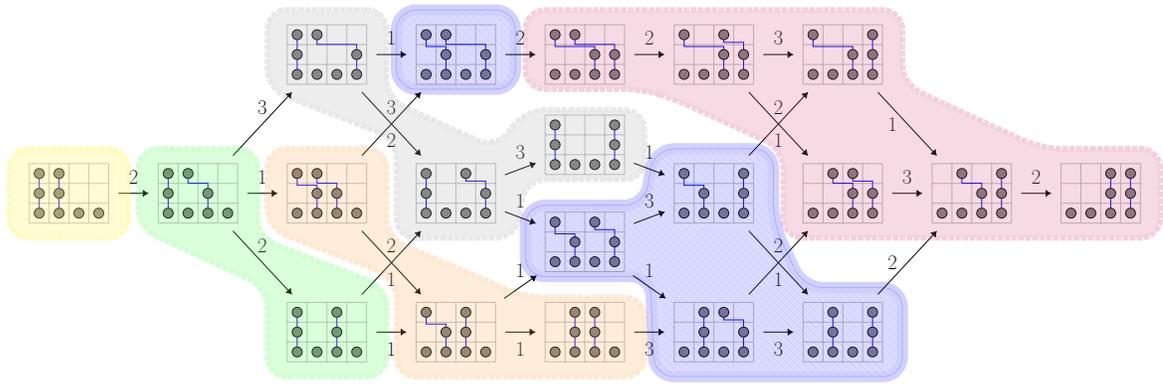}}
        \caption{We show the crystal $\NMLQ_{(3,3,1,1)}$. The different nonsymmetric components $\NMLQ[\alpha]$, for $\sort(\alpha)=(3,3,1,1)$, are highlighted. In particular, note that the component $\NMLQ[(1,3,1,3)]$, highlighted with dashed lines, is disconnected.}
        \label{fig:big figure}
    \end{figure}

\section{Crystal operators and type}\label{sec:crystals}

We examine how collapsing and column crystal operators interact with the $\type$ of a multiline queue. 

\subsection{Column crystal operators and $\type$ change}\label{sec:typechange}

In this section, we give an explicit condition in \cref{cor:Harper's lemma} for when a column crystal operator acting on a multiline queue $M$ induces a $\type$ change; this will serve as a key step in the proofs of our main results.

\begin{definition}\label{def:ei full}
    A multiline queue $M$ is $\mathbf{\eL_i}$\textbf{-full} if $\type(M_i)<\type(M)_{i+1}$ and there is exactly one unpaired $i+1$ in $\theta_i(\rw(M))$.  $M$ is $\mathbf{\fR_i}$\textbf{-full} if $\fR_i(M)$ is $\eL_i$-full.
\end{definition}

\begin{lemma}\label{cor:Harper's lemma}
    For any multiline queue $M$ and column crystal operator $g^\leftrightarrow_i \in \{\eL_i,\fR_i\}$,    
        \[
            \type(g^\leftrightarrow_i(M)) = 
            \begin{cases}
                 s_i \cdot \type(M) & \text{if $M$ is $g^\leftrightarrow_i$-full,} \\
                 \type(M) & \text{otherwise.}
            \end{cases}
        \]
\end{lemma}

To prove \cref{cor:Harper's lemma}, we will give an explicit combinatorial characterization for $\eL_i$-fullness in terms of the multiline queue configuration. Recall \cref{def:FM} of the label array $L_M$ corresponding to a multiline queue $M$. We begin with some preliminary lemmas.

The following lemma states that the labels on any sequence of balls within the same column in consecutive rows must be weakly increasing from top to bottom.
\begin{lemma} \label{lemma:vertical}
    Fix a column $i$, and let $r\geq t$ be such that $L_M(s,i) > 0$ for all $r \geq s > t$. Then, $L_M(r,i) \leq L_M(r-1,i) \leq \cdots \leq L_M(t,i)$.
\end{lemma}

\begin{proof}
    For any $r\geq s>t$, let $x$ and $y$ be the balls at sites $(s,i)$ and $(s-1,i)$ respectively. Then, during the pairing process from row $s$, either $x$ pairs to $y$, in which case $L_M(k,i) = L_M(k-1,i)$, or $y$ was already paired, in which case $L_M(s,i) \leq L_M(s-1,i)$ since balls are paired in order from larger labels to smaller labels. 
\end{proof}

The following lemma states that a ball labeled $\ell > r$ in row $r$ cannot have a smaller labeled ball directly to its left unless there is a ball with the same label directly above.
\begin{lemma} \label{lemma:horizontal}
    If $L_M(r,i) > r$ and $L_M(r+1,i) \neq L_M(r,i)$, then either $L_M(r,i-1) = 0$ or $L_M(r,i-1) \geq L_M(r,i)$.
\end{lemma}

\begin{proof}
    Let $x$ be the ball in row $r+1$ paired to the ball $y$ at site $(r,i)$. By assumption, $L_M(r+1,i) \neq L_M(r,i)$, so $x$ is not at site $(r+1,i)$. Suppose $0<L_M(r,i-1) <L_M(r,i)$. Then the ball at site $(r,i-1)$ was unpaired when ball $y$ was paired-- a contradiction, since $x$ must necessarily skip over site $(r,i-1)$ when pairing to ball $y$. This proves the claim.
\end{proof}

Next we examine the multiline queue $\eL_i(M)$ whenever $\eL_i$ acts nontrivially. As we will see, $L_{\eL_i(M)}$ only differs from $L_M$ in a specific subset of its sites, which we call the \emph{$i$-active region}.

\begin{definition} \label{def:active region}
    Suppose $\eL_i(M) \neq M$, and let $x$ be the ball at site $(r,i+1)$ that moves to site $(r,i)$ in $\eL_i(M)$; let $\ell$ be the label of $x$. Let $p \leq r$ be the maximal row such that $L_M(p-1,i) = 0$ or $L_M(p-1,i) \geq \ell$. If no such $p$ exists, set $p=1$.  Define the $\mathbf{i}$\textbf{-active region} of $M$ as the collection of sites
    \[
        \act_i(M) = \{(s,j): p \leq s \leq r, j \in\{i, i+1\}\},
    \] 
    and define an $\mathbf{i}$\textbf{-active row} as a row containing a site in $\act_i(M)$.
\end{definition}

Note that by construction, if $\act_i(M)$ has topmost row $r$, then $M$ has no vacancies for any sites in $\act_i(M)$ below row $r$. See \cref{ex:iactive} for examples.

\begin{lemma}
    A multiline queue $M$ is $\eL_i$-full if and only if $\act_i(M)$ contains row 1.
\end{lemma}

\begin{proof}
    If $\act_i(M)$ contains row 1, then $L_M(1,i) < L_M(1,i+1)$ or equivalently, $\type(M)_i < \type(M)_{i+1}$. Let $x$ be the ball at site $(r,i+1)$ that moves to site $(r,i)$ in $\eL_i(M)$. Since $x$ is the highest unmatched ball in column $i+1$ of $M$ and $M$ has no vacancies in $\act_i(M)$ below row $r$, $x$ is the \textit{only} unmatched ball in column $i+1$ in $\theta_i(rw(M))$. Thus, $M$ is $\eL_i$-full. 

    For the converse, if $M$ is $\eL_i$-full, then $L_M(1,i) < L_M(1,i+1)$. We claim that this ascent in ball labels in row 1 implies that there exists a row $r$ such that $L_M(r,i)=0$, $L_M(r,i+1)>0$, and $L_M(h,i)>0$ and $L_M(h,i+1)>0$ for all $1\leq h<r$. We further claim that the ball at site $(r,i+1)$ has label $L_M(1,i+1)$. This claimed configuration is visualized below, with the ball with claimed label $L_M(1,i+1)$ highlighted, and the `$\ast$' above row $r$ indicating arbitrary content.
    
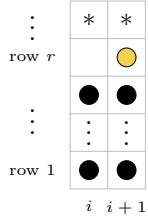
\[
    \begin{tikzpicture}[scale=0.5]
            \def \w{1};
        \def \h{1};
        \def \r{0.25};
        \foreach \i in {0,...,5}
        {
        \draw[gray!50] (0,\i*\h)--(\w*2,\i*\h);
        }
        \foreach \i in {0,1,2}
        {
        \draw[gray!50] (\w*\i,0)--(\w*\i,5*\h);
        }
        \node at (-1,.5*\h) {\tiny row $1$};
        \node at (-1,2.5*\h-.5*\h) { $\vdots$};
        \node at (-1,4*\h-.5*\h) {\tiny row $r$};
        \node at (-1,5*\h-.5*\h) { $\vdots$};
        \node at (0.5*\w,-0.45) {\tiny$i$};
        \node at (1.5*\w,-0.5) {\tiny $i+1$};

        \foreach \i\j in {0/0,1/0,0/2,1/2}
        {
        \draw[fill=black] (\i+\w*0.5,\j+\h*.5) circle (\r cm) node[color=black] {};
        }
        \draw[fill=\highlight] (\w*1.5,\h*3.5) circle (\r cm) node[color=black] {};
        \node at (\w*.5,\h*4.5) {$\ast$};
        \node at (\w*1.5,\h*4.5) {$\ast$};
        \node at (\w*.5,1.7*\h) { $\vdots$};
        \node at (\w*1.5,1.7*\h) { $\vdots$};
    \end{tikzpicture}
    \]

    If $L_M(1,i) = 0$, this claim holds trivially. So suppose that $L_M(h,i+1)>L_M(h,i)>0$ for some $h\geq 1$ (we are guaranteed to have this for $h=1$). Then, there is a ball $y$ in row $h+1$ that pairs with the ball at site $(h,i+1)$ with higher priority than the ball that pairs with the ball at site $(h,i)$, which can only occur if ball $y$ is in column $i+1$. Then if $L_M(h+1,i)>0$, we have $L_M(h+1,i)\leq L_M(h,i)<L_M(h,i+1)=L_M(h+1,i+1)$. 

    Iteratively applying the statement above for $h=1,2,\ldots$, we have that either $L_M(h,i)>0$ and $L_M(h,i+1)=L_M(1,i+1)$ for $1\leq h\leq L$, or there is some row $r$ such that $L_M(r,i)=0$, $L_M(r,i+1)=L_M(1,i+1)$, with $L_M(h,i)>0$ and $L_M(h,i+1)=L_M(1,i+1)$ for all $1\leq h<r$. In the latter case, we have proved our claim, and the former case implies $L_M(1,i)=L$, a contradiction. 

    Now, observe that the ball at site $(r,i+1)$ which is highlighted in yellow above is unmatched in $\theta_i(rw(M))$. Since $M$ is $\eL_i$-full, this is the only such unmatched ball and hence is the ball moved by $\eL_i$. Lastly, since $L_M(r,i+1) = L_M(1,i+1) > L_M(1,i) \geq L_M(h,i)$ for all $1 \leq h < r$ by \cref{lemma:vertical}, $\act_i(M)$ is precisely the configuration pictured above. In particular, $\act_i(M)$ contains sites in row 1.
\end{proof}

\begin{example}\label{ex:iactive}
    We show the active regions for $i=2,5,7$ in the multiline queue $M$ shown below on the left, with the balls moved by $\eL_2, \eL_5, \eL_7$ respectively highlighted in yellow. 
    The $2$-active region $\act_2(M)$ is shown in red and, as it extends down to the first row, $M$ is $\eL_2$-full. The region $\act_5(M)$ is shown in blue and stops at row three since there is no ball at site $(5,2)$. The region $\act_7(M)$ is shown in green and also stops at row three since there is a ball with label $3$ at site $(7,2)$, which is \textit{weakly larger} than the label of the ball at site $(8,3)$ (the one moved by $\eL_7$). 
\smallskip

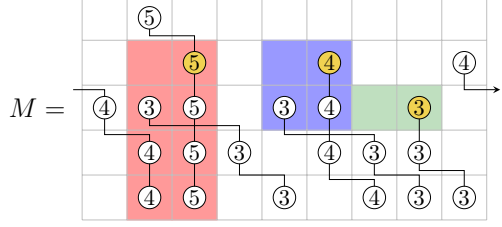
\begin{center}
        \resizebox{!}{3cm}{
\begin{tikzpicture}[scale=0.7]
\def \w{1};
\def \h{1};
\def \r{0.25};
    
\node at (-1,2.5) {\large $M=$};

\def \one{red!40!white};
\def \two{blue!40!white};
\def \three{black!50!green!25!white};

\foreach \xxx\yyy\c in {
1/3/\one,2/3/\one,1/2/\one,2/2/\one,1/1/\one,2/1/\one,1/0/\one,2/0/\one,
4/3/\two,5/3/\two,4/2/\two,5/2/\two,
6/2/\three,7/2/\three}
    {
    \draw[\c, fill=\c](\xxx,\yyy) rectangle (\xxx+\w,\yyy+\h);
    }


\draw[gray!50,thin,step=\w] (0,0) grid (9*\w,5*\h);
\foreach \xx\yy\i\c in {0/2/4/white,
1/0/4/white,1/1/4/white,1/2/3/white,1/4/5/white,
2/0/5/white,2/1/5/white,2/2/5/white,2/3/5/\highlight,
3/1/3/white,
4/0/3/white,4/2/3/white,
5/1/4/white,5/2/4/white,5/3/4/\highlight,
6/0/4/white,6/1/3/white,
7/0/3/white,7/1/3/white,7/2/3/\highlight,
8/0/3/white,8/3/4/white}
    {
    \draw[fill=\c](\w*.5+\w*\xx,\h*.5+\h*\yy) circle (\r cm);
    \node at (\w*.5+\w*\xx,\h*.5+\h*\yy) {\i};
    }




\draw[black](\w*1.5,\h*4.5-\r)--(\w*1.5,\h*4.1)--(\w*2.5,\h*4.1)--(\w*2.5,\h*3.5+\r);
\draw[black](\w*2.5,\h*3.5-\r)--(\w*2.5,\h*2.5+\r);
\draw[black](\w*2.5,\h*2.5-\r)--(\w*2.5,\h*1.5+\r);
\draw[black](\w*2.5,\h*1.5-\r)--(\w*2.5,\h*0.5+\r);

\draw[black](\w*5.5,\h*3.5-\r)--(\w*5.5,\h*2.5+\r);
\draw[black](\w*5.5,\h*2.5-\r)--(\w*5.5,\h*1.5+\r);
\draw[black](\w*5.5,\h*1.5-\r)--(\w*5.5,\h*0.9)--(\w*6.5,\h*0.9)--(\w*6.5,\h*0.5+\r);

\draw[black,-stealth](\w*8.5,\h*3.5-\r)--(\w*8.5,\h*2.9)--(\w*9.3,\h*2.9);
\draw[black](-0.2,\h*2.9)--(\w*0.5,\h*2.9)--(\w*0.5,\h*2.5+\r);
\draw[black](\w*0.5,\h*2.5-\r)--(\w*0.5,\h*1.9)--(\w*1.5,\h*1.9)--(\w*1.5,\h*1.5+\r);
\draw[black](\w*1.5,\h*1.5-\r)--(\w*1.5,\h*0.5+\r);

\draw[black](\w*1.5,\h*2.5-\r)--(\w*1.5,\h*2.1)--(\w*3.5,\h*2.1)--(\w*3.5,\h*1.5+\r);
\draw[black](\w*3.5,\h*1.5-\r)--(\w*3.5,\h*1)--(\w*4.5,\h*1)--(\w*4.5,\h*0.5+\r);

\draw[black](\w*4.5,\h*2.5-\r)--(\w*4.5,\h*1.9)--(\w*6.5,\h*1.9)--(\w*6.5,\h*1.5+\r);
\draw[black](\w*6.5,\h*1.5-\r)--(\w*6.5,\h*1)--(\w*7.5,\h*1)--(\w*7.5,\h*0.5+\r);

\draw[black](\w*7.5,\h*2.5-\r)--(\w*7.5,\h*1.5+\r);
\draw[black](\w*7.5,\h*1.5-\r)--(\w*7.5,\h*1.1)--(\w*8.5,\h*1.1)--(\w*8.5,\h*0.5+\r);
\end{tikzpicture}
}
\end{center}

    \smallskip
   Below, we show $\eL_2(M)$ and $\eL_5(M)$: in both cases, the labeled balls within the corresponding active regions of $M$ (shown in red and blue, respectively) swap columns. Note that since $M$ is $\eL_2$-full, $\type(e_2(M))=s_2 \cdot \type(M)$, whereas $M$ is not $\eL_5$-full, so $\type(\eL_5(M))=\type(M)$. 
    \smallskip
    
        \centering
        \resizebox{!}{3cm}
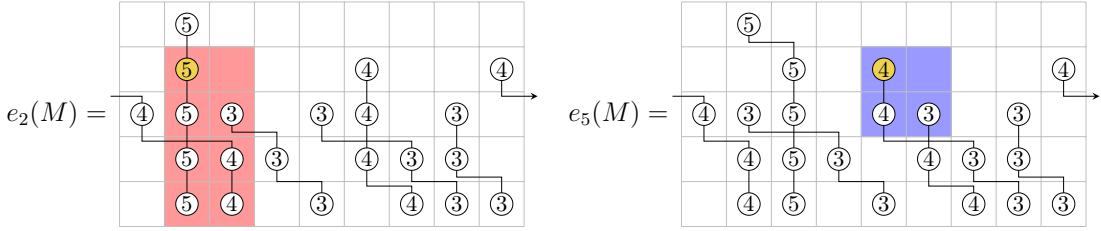
{
\begin{tikzpicture}[scale=0.7]
\def \w{1};
\def \h{1};
\def \r{0.25};
    
\node at (-1.4,2.5) {\large $e_2(M)=$};

\def \one{red!40!white};
\def \two{blue!40!white};
\def \three{black!50!green!25!white};

\foreach \xxx\yyy\c in {
1/3/\one,2/3/\one,1/2/\one,2/2/\one,1/1/\one,2/1/\one,1/0/\one,2/0/\one,
}
    {
    \draw[\c, fill=\c](\xxx,\yyy) rectangle (\xxx+\w,\yyy+\h);
    }

\draw[gray!50,thin,step=\w] (0,0) grid (9*\w,5*\h);
\foreach \xx\yy\i\c in {0/2/4/white,
1/0/5/white,1/1/5/white,1/2/5/white,1/3/5/\highlight,1/4/5/white,
2/0/4/white,2/1/4/white,2/2/3/white,
3/1/3/white,
4/0/3/white,4/2/3/white,
5/1/4/white,5/2/4/white,5/3/4/white,
6/0/4/white,6/1/3/white,
7/0/3/white,7/1/3/white,7/2/3/white,
8/0/3/white,8/3/4/white}
    {
    \draw[fill=\c](\w*.5+\w*\xx,\h*.5+\h*\yy) circle (\r cm);
    \node at (\w*.5+\w*\xx,\h*.5+\h*\yy) {\i};
    }

\draw[black](\w*1.5,\h*4.5-\r)--(\w*1.5,\h*3.5+\r) (\w*1.5,\h*3.5-\r)--(\w*1.5,\h*2.5+\r) (\w*1.5,\h*2.5-\r)--(\w*1.5,\h*1.5+\r); 

\draw[black](\w*2.5,\h*1.5-\r)--(\w*2.5,\h*0.5+\r);

\draw[black](\w*5.5,\h*3.5-\r)--(\w*5.5,\h*2.5+\r);
\draw[black](\w*5.5,\h*2.5-\r)--(\w*5.5,\h*1.5+\r);
\draw[black](\w*5.5,\h*1.5-\r)--(\w*5.5,\h*0.9)--(\w*6.5,\h*0.9)--(\w*6.5,\h*0.5+\r);

\draw[black,-stealth](\w*8.5,\h*3.5-\r)--(\w*8.5,\h*2.9)--(\w*9.3,\h*2.9);
\draw[black](-0.2,\h*2.9)--(\w*0.5,\h*2.9)--(\w*0.5,\h*2.5+\r);
\draw[black](\w*0.5,\h*2.5-\r)--(\w*0.5,\h*1.9)--(\w*2.5,\h*1.9)--(\w*2.5,\h*1.5+\r);
\draw[black](\w*1.5,\h*1.5-\r)--(\w*1.5,\h*0.5+\r);

\draw[black](\w*2.5,\h*2.5-\r)--(\w*2.5,\h*2.1)--(\w*3.5,\h*2.1)--(\w*3.5,\h*1.5+\r);
\draw[black](\w*3.5,\h*1.5-\r)--(\w*3.5,\h*1)--(\w*4.5,\h*1)--(\w*4.5,\h*0.5+\r);

\draw[black](\w*4.5,\h*2.5-\r)--(\w*4.5,\h*1.9)--(\w*6.5,\h*1.9)--(\w*6.5,\h*1.5+\r);
\draw[black](\w*6.5,\h*1.5-\r)--(\w*6.5,\h*1)--(\w*7.5,\h*1)--(\w*7.5,\h*0.5+\r);

\draw[black](\w*7.5,\h*2.5-\r)--(\w*7.5,\h*1.5+\r);
\draw[black](\w*7.5,\h*1.5-\r)--(\w*7.5,\h*1.1)--(\w*8.5,\h*1.1)--(\w*8.5,\h*0.5+\r);

\begin{scope}[xshift=12.5cm]   
\node at (-1.4,2.5) {\large $e_5(M)=$};

\def \one{red!40!white};
\def \two{blue!40!white};
\def \three{black!50!green!25!white};

\foreach \xxx\yyy\c in {
4/3/\two,5/3/\two,4/2/\two,5/2/\two,
}
    {
    \draw[\c, fill=\c](\xxx,\yyy) rectangle (\xxx+\w,\yyy+\h);
    }


\draw[gray!50,thin,step=\w] (0,0) grid (9*\w,5*\h);
\foreach \xx\yy\i\c in {0/2/4/white,
1/0/4/white,1/1/4/white,1/2/3/white,1/4/5/white,
2/0/5/white,2/1/5/white,2/2/5/white,2/3/5/white,
3/1/3/white,
4/0/3/white,5/2/3/white,
5/1/4/white,4/2/4/white,4/3/4/\highlight,
6/0/4/white,6/1/3/white,
7/0/3/white,7/1/3/white,7/2/3/white,
8/0/3/white,8/3/4/white}
    {
    \draw[fill=\c](\w*.5+\w*\xx,\h*.5+\h*\yy) circle (\r cm);
    \node at (\w*.5+\w*\xx,\h*.5+\h*\yy) {\i};
    }

\draw[black](\w*1.5,\h*4.5-\r)--(\w*1.5,\h*4.1)--(\w*2.5,\h*4.1)--(\w*2.5,\h*3.5+\r);
\draw[black](\w*2.5,\h*3.5-\r)--(\w*2.5,\h*2.5+\r);
\draw[black](\w*2.5,\h*2.5-\r)--(\w*2.5,\h*1.5+\r);
\draw[black](\w*2.5,\h*1.5-\r)--(\w*2.5,\h*0.5+\r);
\draw[black](\w*4.5,\h*3.5-\r)--(\w*4.5,\h*2.5+\r);
\draw[black](\w*4.5,\h*2.5-\r)--(\w*4.5,\h*1.9)--(\w*5.5,\h*1.9)--(\w*5.5,\h*1.5+\r);
\draw[black](\w*5.5,\h*1.5-\r)--(\w*5.5,\h*0.9)--(\w*6.5,\h*0.9)--(\w*6.5,\h*0.5+\r);
\draw[black](\w*5.5,\h*2.5-\r)--(\w*5.5,\h*1.9)--(\w*6.5,\h*1.9)--(\w*6.5,\h*1.5+\r);
\draw[black,-stealth](\w*8.5,\h*3.5-\r)--(\w*8.5,\h*2.9)--(\w*9.3,\h*2.9);
\draw[black](-0.2,\h*2.9)--(\w*0.5,\h*2.9)--(\w*0.5,\h*2.5+\r);
\draw[black](\w*0.5,\h*2.5-\r)--(\w*0.5,\h*1.9)--(\w*1.5,\h*1.9)--(\w*1.5,\h*1.5+\r);
\draw[black](\w*1.5,\h*1.5-\r)--(\w*1.5,\h*0.5+\r);
\draw[black](\w*1.5,\h*2.5-\r)--(\w*1.5,\h*2.1)--(\w*3.5,\h*2.1)--(\w*3.5,\h*1.5+\r);
\draw[black](\w*3.5,\h*1.5-\r)--(\w*3.5,\h*1)--(\w*4.5,\h*1)--(\w*4.5,\h*0.5+\r);
\draw[black](\w*6.5,\h*1.5-\r)--(\w*6.5,\h*1)--(\w*7.5,\h*1)--(\w*7.5,\h*0.5+\r);
\draw[black](\w*7.5,\h*2.5-\r)--(\w*7.5,\h*1.5+\r);
\draw[black](\w*7.5,\h*1.5-\r)--(\w*7.5,\h*1.1)--(\w*8.5,\h*1.1)--(\w*8.5,\h*0.5+\r);
\end{scope}
\end{tikzpicture}
}
\end{example}

The following lemma states that the labels in $M$ of all balls in column $i+1$ of $\act_i(M)$ are equal to the label of the topmost ball moved by $\eL_i$. 

\begin{lemma} \label{lemma:active region}
    Suppose $\eL_i(M) \neq M$ and let $r$ be the row at which they differ. Then $L_M(s,i+1)=L_M(r,i+1)$ for all $(s,i+1)\in\act_i(M)$. 
\end{lemma}

\begin{proof}
    Let $\ell = L_M(r,i+1)$ and apply \cref{lemma:vertical} to obtain $L_M(s,i+1) \geq \ell$ for all $r \geq s \geq p$. Assume for sake of contradiction that there exists some $r \geq s \geq p$ such that $L_M(s,i+1) > \ell$. Choose $s$ maximal with respect to this property. Then $L_M(s,i+1) > s$ since by maximality,     
        \begin{align*}
            L_M(s,i+1) > L_M(s+1,i+1) \geq s+1 > s.
        \end{align*}
     Thus, we may apply \cref{lemma:horizontal} and since $L_M(s,i) \neq 0$, this yields $L_M(s,i) \geq L_M(s,i+1)$. But this contradicts the definition of $\act_i(M)$.
\end{proof}

 Let $L_M^{(r)}$ denote the $r$th row of $L_M$. The following lemma will be a base case for \cref{lemma:column swap}.

\begin{lemma} \label{lemma:no label change}
    Suppose $\eL_i(M) \neq M$ and let $r$ be the row at which they differ. Then $L_M$ and $L_{\eL_i(M)}$ match above row $r$ and $L_{\eL_i(M)}^{(r)} = s_i \cdot L_M^{(r)}$.
\end{lemma}

\begin{proof}
    The first claim follows directly from the fact that the FM labeling procedure is carried out from top to bottom and $M$ and $\eL_i(M)$ match above row $r$.
    
For the other claim, it suffices to show that if there is a ball $y$ in row $r+1$ that pairs down to $x$ in $M$, then that same ball also pairs down to $x$ in $\eL_i(M)$. If the ball $y$ is not in the same column as $x$, then $x$ being moved by $\eL_i$ has no impact on the pairing process and $y$ still pairs to $x$ in $\eL_i(M)$. Hence, assume $y$ is indeed in column $i+1$. Let $z$ be the ball directly to the left of $y$ at site $(r+1,i)$ in $M$. Since $y$ pairs down to $x$, the label of $z$ must be strictly less than the label of $y$.

Consider the strand $s$ going through $y$ and $x$. Scanning up column $i+1$, starting at row $r$, let $h+1$ be the first row such that $s$ does not appear in column $i+1$. By definition there are no vacancies in column $i+1$ between rows $r$ and $q$. Further, since $x$ is the highest unmatched ball in column $i+1$ of $M$, there are no vacancies in column $i$ between rows $r+1$ and $h$. Thus, both $y$ and $z$ have label at least $h$ by \cref{lemma:vertical}. If the strand $s$ began in row $h$, then $y$ would have label weakly less than $z$, a contradiction. So there is a ball in row $h+1$ and \textit{not} in column $i+1$ that pairs down to the ball at site $(h,i+1)$. But this implies that the ball at site $(h,i)$ was already paired to a ball in row $h-1$. By another application of \cref{lemma:vertical}, this means the label of $z$ must be weakly greater than the label of $y$, a contradiction.  
\end{proof}

The following lemma states when $\eL_i$ acts nontrivially, the labels in the active region $\act_i$ swap between columns $i$ and $i+1$. For an example, see \cref{ex:iactive}.
\begin{lemma} \label{lemma:column swap}
    Suppose $\eL_i(M) \neq M$ and suppose $\act_i(M)$ spans rows $p$ through $r$ for some $p\leq r$. Then,
        \begin{align}\label{eq:Lsm}
            L_{\eL_i(M)}^{(s)} = \begin{cases}
                s_i \cdot L_M^{(s)} & p \leq s \leq r, \\
                L_M^{(s)} & \text{otherwise}.
            \end{cases}
        \end{align}
\end{lemma}

\begin{proof}
    We first prove this under the assumption that $M$ is $\eL_i$-full. Induct on the number of $i$-active rows in $M$. The base case, when $M$ has just one $i$-active row, follows directly from \cref{lemma:no label change}. Now assume $M$ has more than one $i$-active row and let $M'$ be the multiline queue obtained by deleting the bottom row of $M$. As we have assumed $M$ is $\eL_i$-full, the bottom row of $M$ is $i$-active so $M'$ has one fewer $i$-active rows than $M$. By the induction hypothesis, the statement holds for $M'$ or equivalently, the statement holds for every row of $M$ except row $1$. To show \eqref{eq:Lsm} for $s=1$, we must prove the following points:
    
        \begin{enumerate}[label=(\arabic*), ref=(\arabic*)]
            \item\label{e:1} $L_{\eL_i(M)}(1,i) = L_M(1,i+1)$,
            \item\label{e:2} $L_{\eL_i(M)}(1,i+1) = L_M(1,i)$, 
            \item\label{e:3} $L_{\eL_i(M)}(1,i') = L_M(1,i')$ for $i' \neq i,i+1$.
        \end{enumerate}

    Since the first and second rows of $M$ are $i$-active, there are balls at sites $(1,i)$, $(1,i+1)$, $(2,i)$, and $(2,i+1)$ by \cref{lemma:active region}. Moreover, by the same result, the ball at site $(2,i+1)$ pairs down to the ball at site $(1,i+1)$. Thus, there are only two possible configurations in the cells in rows 1 and 2 in columns $i$ and $i+1$ of $M$ that are relevant, depending on whether the ball at site $(2,i)$ pairs to the ball directly below it or not. We show these configurations below: in configuration (a), the ball at site $(2,i)$ pairs directly below, and in configuration (b), it pairs to a ball at site $(1,j')$ while the ball below it is paired to from a ball at site $(2,j)$, for $j,j'\not\in\{i,i+1\}$. (As pairings may be wrapping, we allow for $j>i$ and $i> j'$.) We associate the red pairing with the label $L_M(1,i)$ and the blue pairing with the label $L_M(1,i+1)$.

   \begin{center}
        \resizebox{!}{1.5cm}
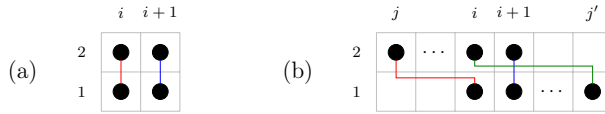
{
        \begin{tikzpicture}[scale=0.7]
        \def \w{1};
        \def \h{1};
        \def \r{0.20};
        \def \s{7}; 

        \begin{scope}[xshift=0cm]
        \node at (-2,1) {\large (a)};
        \draw[gray!50,thin,step=\w] (0,0) grid (2*\w,2*\h);
        \foreach \xx\yy\i\c in {0/0//black,0/1//black,1/0//black,1/1//black}
            {
            \draw[fill=\c](\w*.5+\w*\xx,\h*.5+\h*\yy) circle (\r cm);
            \node at (\w*.5+\w*\xx,\h*.5+\h*\yy) {\i};
            }
            
        \draw[red] (\w*0.5,\h*1.5-\r)--(\w*0.5,\h*0.5+\r);
        \draw[blue] (\w*1.5,\h*1.5-\r)--(\w*1.5,\h*0.5+\r);

        \node at (\w*-0.5,\h*0.5) {\scriptsize$1$};
        \node at (\w*-0.5,\h*1.5) {\scriptsize$2$};
        \node at (\w*0.5,\h*2.5) {\scriptsize$i$};
        \node at (\w*1.5,\h*2.5) {\scriptsize$i+1$};
        \end{scope}

        \begin{scope}[xshift=\s cm]
            \node at (-2,1) {\large (b)};
        \draw[gray!50,thin,step=\w] (0,0) grid (6*\w,2*\h);
        \foreach \xx\yy\i\c in {2/0//black,2/1//black,3/0//black,3/1//black,
        0/1//black,5/0//black}
            {
            \draw[fill=\c](\w*.5+\w*\xx,\h*.5+\h*\yy) circle (\r cm);
            \node at (\w*.5+\w*\xx,\h*.5+\h*\yy) {\i};
            }

        \draw[blue] (\w*3.5,\h*1.5-\r)--(\w*3.5,\h*0.5+\r);
        \draw[red] (\w*0.5,\h*1.5-\r)--(\w*0.5,\h*.85)--(\w*2.5,\h*.85)--(\w*2.5,\h*.5+\r);
        \draw[black!50!green] (\w*2.5,\h*1.5-\r)--(\w*2.5,\h*1.15)--(\w*5.5,\h*1.15)--(\w*5.5,\h*.5+\r);

        \node at (\w*1.5,\h*1.5) {$\cdots$};
        \node at (\w*4.5,\h*0.5) {$\cdots$};

        \node at (\w*-0.5,\h*0.5) {\scriptsize$1$};
        \node at (\w*-0.5,\h*1.5) {\scriptsize$2$};
        \node at (\w*2.5,\h*2.5) {\scriptsize$i$};
        \node at (\w*3.5,\h*2.5) {\scriptsize$i+1$};
        \node at (\w*0.5,\h*2.5) {\scriptsize$j$};
        \node at (\w*5.5,\h*2.5) {\scriptsize$j'$};
        \end{scope}

        \end{tikzpicture}
        }
        \end{center}

     The reader should note that the proof of \ref{e:1} will be identical for either configuration; we only distinguish which case we are in for the proofs of \ref{e:2} and \ref{e:3}.

   We claim that in $\eL_i(M)$, the pairings in the two relevant configurations will be in configurations (a$'$) and (b$'$) below, respectively.
   
     \begin{center}
        \resizebox{!}{1.5cm}
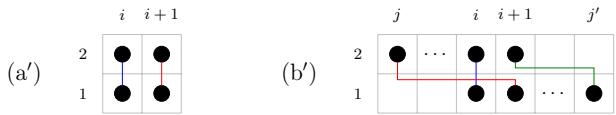
{
        \begin{tikzpicture}[scale=0.7]
        \def \w{1};
        \def \h{1};
        \def \r{0.20};
        \def \s{7}; 

        \begin{scope}[xshift=0cm]
                    \node at (-2,1) {\large (a$'$)};

        \draw[gray!50,thin,step=\w] (0,0) grid (2*\w,2*\h);
        \foreach \xx\yy\i\c in {0/0//black,0/1//black,1/0//black,1/1//black}
            {
            \draw[fill=\c](\w*.5+\w*\xx,\h*.5+\h*\yy) circle (\r cm);
            \node at (\w*.5+\w*\xx,\h*.5+\h*\yy) {\i};
            }
            
        \draw[blue] (\w*0.5,\h*1.5-\r)--(\w*0.5,\h*0.5+\r);
        \draw[red] (\w*1.5,\h*1.5-\r)--(\w*1.5,\h*0.5+\r);

        \node at (\w*-0.5,\h*0.5) {\scriptsize$1$};
        \node at (\w*-0.5,\h*1.5) {\scriptsize$2$};
        \node at (\w*0.5,\h*2.5) {\scriptsize$i$};
        \node at (\w*1.5,\h*2.5) {\scriptsize$i+1$};
        \end{scope}

        \begin{scope}[xshift=\s cm]
                        \node at (-2,1) {\large (b$'$)};

        \draw[gray!50,thin,step=\w] (0,0) grid (6*\w,2*\h);
        \foreach \xx\yy\i\c in {2/0//black,2/1//black,3/0//black,3/1//black,
        0/1//black,5/0//black}
            {
            \draw[fill=\c](\w*.5+\w*\xx,\h*.5+\h*\yy) circle (\r cm);
            \node at (\w*.5+\w*\xx,\h*.5+\h*\yy) {\i};
            }

        \draw[blue] (\w*2.5,\h*1.5-\r)--(\w*2.5,\h*0.5+\r);
        \draw[red] (\w*0.5,\h*1.5-\r)--(\w*0.5,\h*.85)--(\w*3.5,\h*.85)--(\w*3.5,\h*.5+\r);
        \draw[black!50!green] (\w*3.5,\h*1.5-\r)--(\w*3.5,\h*1.15)--(\w*5.5,\h*1.15)--(\w*5.5,\h*.5+\r);

        \node at (\w*1.5,\h*1.5) {$\cdots$};
        \node at (\w*4.5,\h*0.5) {$\cdots$};

        \node at (\w*-0.5,\h*0.5) {\scriptsize$1$};
        \node at (\w*-0.5,\h*1.5) {\scriptsize$2$};
        \node at (\w*2.5,\h*2.5) {\scriptsize$i$};
        \node at (\w*3.5,\h*2.5) {\scriptsize$i+1$};
        \node at (\w*0.5,\h*2.5) {\scriptsize$j$};
        \node at (\w*5.5,\h*2.5) {\scriptsize$j'$};
        \end{scope}

        \end{tikzpicture}
        }
        \end{center}
 The blue and red pairings will imply \ref{e:1} and \ref{e:2}, respectively. In configuration (a$'$), the blue and red pairings will also imply \ref{e:3}, and in (b$'$), the green pairing plus the first two claims will imply \ref{e:3}.

We will now prove our claim. Let $\ell = L_M(r,i+1)$, the label of the ball moved by $\eL_i$. We first prove \ref{e:1}. By \cref{lemma:active region} and the induction hypothesis, $L_{\eL_i(M)}(2,i) = \ell$. Assume for sake of contradiction that $L_{\eL_i(M)}(1,i) \neq \ell$. By \cref{lemma:vertical}, $L_{\eL_i(M)}(1,i) > \ell$, implying that there is a ball at some site $(2,j)$ that pairs down to the ball at site $(1,i)$ of $\eL_i(M)$. Again by the induction hypothesis, $L_M(2,j) = L_{\eL_i(M)}(2,j) > \ell$ so there is a ball $y$ at site $(2,j)$ of $M$ with label larger than $\ell$ that does not pair to any ball before column $i$. But since the first row of $M$ is $i$-active, $L_M(1,i) < \ell$, implying that $y$ would pair to the ball at site $(1,i)$ of $M$; therefore $L_M(1,i) > \ell$, a contradiction.
    
    Next, we prove \ref{e:2}. Let $x$ be the ball at site $(1,i)$ of $M$, let $y$ be the ball at site $(2,j)$ of $M$ that pairs down to $x$, and let $m = L_M(1,i)$. If $j = i$, then $M$ is in configuration (a) above. In this case, \ref{e:2} follows by an argument almost identical to that of \ref{e:1}. If $j = i+1$, then $m = \ell$, contradicting the definition of $\act_i(M)$. Therefore, we may assume $j \neq i, i+1$, so $M$ is in configuration (b) above. Let $x'$ be the ball at site $(1,i+1)$ of $\eL_i(M)$, and let $y'$ be the ball at site $(2,j)$ of $\eL_i(M)$. By the induction hypothesis, $L_{\eL_i(M)}(2,j) = L_M(2,j) = m$. Assume for sake of contradiction that $y'$ does not pair to $x'$. Then either $y'$ pairs to some new ball that came before $x'$ in the pairing process or some new ball $z'$ pairs to $x'$. The first scenario cannot occur since the only new ball before $x'$ is labeled $\ell > m$, so it will have already been paired. The second scenario also cannot occur, but requires a more detailed explanation: in order for such a $z'$ to exist, it must come before $y'$ in pairing order and must be in column $i$. In fact, such a $z'$ does exist: the ball at site $(2,i)$ does have label $\ell > m$. However, in the proof of \ref{e:1}, ball $z'$ was shown to pair down to the ball at site $(2,i)$. Thus, $y'$ does indeed pair down to $x'$, and so $L_{\eL_i(M)}(1,i+1) = m = L_M(1,i)$.

    Lastly, we prove \ref{e:3}. If $M$ is in configuration (a) above, this follows from \ref{e:1} and \ref{e:2}. Thus, assume $M$ is in configuration (b). Let $u$ be the ball at site $(2,i)$ of $M$, let $w$ be the ball at site $(1,j')$ of $M$ that $u$ pairs down to, and let $L_M(2,i) = h$. Since $M$ is in the right configuration, we can not have $j' = i$ and by \ref{e:1}, we can not have $j' = i+1$. Further, with $m$ as defined above, $h \leq m$ by \cref{lemma:vertical}. Let $u'$ be the ball at site $(2,i+1)$ of $\eL_i(M)$ and let $w'$ be the ball at site $(1,j')$ of $\eL_i(M)$. By the induction hypothesis, $L_{\eL_i(M)}(2,i+1) = L_M(2,i) = h$. Assume for sake of contradiction that $u'$ does not pair down to $w'$. Then, again, either $u'$ pairs to some new ball that came before $w'$ in the pairing process or some new ball $v'$ pairs to $w'$. The first scenario cannot occur because the only new ball that could come before $w'$ is the ball at site $(1,i+1)$. But by the proof of \ref{e:2}, the ball at site $(1,i+1)$ of $\eL_i(M)$ is paired before we determine which ball $u'$ pairs to. The second scenario cannot occur since there is not even a candidate for a new ball. Thus, $u'$ does indeed pair down to $w'$ and so $L_{\eL_i(M)}(1,j') = L_M(1,j')$. Moreover, since the balls pictured in row 2 of configuration (b) above paired to the same set of balls in $M$ and in $\eL_i(M)$, the pairings of the remaining balls in row 2 will remain the same, proving \ref{e:3}.

    We now drop the assumption that $M$ is $\eL_i$-full and prove the statement in full generality. Consider the multiline queue $M'$ obtained from $M$ by deleting the first $p-1$ rows of $M$. Then $M'$ is $\eL_i$-full so the statement is true for $M'$, or equivalently, for all but the bottom $p-1$ rows of $M$. Thus, it suffices to show that $L_{\eL_i(M)}(p-1,i') = L_M(p-1,i')$ for all $i'$. We analyze rows $p$ and $p-1$. Since row $p$ is $i$-active, $L_M(p,i) < L_M(p,i+1)$ but since row $p-1$ is not, $L_M(p-1,i) \neq L_M(p,i)$. Therefore, the cells in rows $p$ and $p-1$ in columns $i$ and $i+1$ must be in configuration (c) below. (Again, as pairings may be wrapping, we allow for $j < i+1$ and $j' < i$. We also allow for $j = i+1$.) We claim that in $\eL_i(M)$, the pairings will be in configuration (c') below, as this will prove 
    
            \begin{center}
            \resizebox{!}{1.5cm}
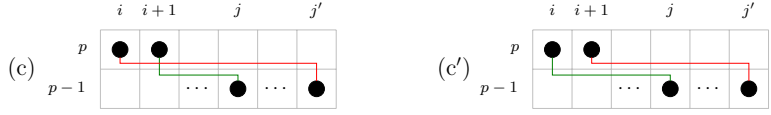
{
            \begin{tikzpicture}[scale=0.7]
            \def \w{1};
            \def \h{1};
            \def \r{0.20};
            \begin{scope}[xshift=0cm]
            \node at (-2,1) {\large (c)};
            \draw[gray!50,thin,step=\w] (0,0) grid (6*\w,2*\h);
            \foreach \xx\yy\i\c in {0/1//black,1/1//black,
            3/0//black,5/0//black}
                {
                \draw[fill=\c](\w*.5+\w*\xx,\h*.5+\h*\yy) circle (\r cm);
                \node at (\w*.5+\w*\xx,\h*.5+\h*\yy) {\i};
                }
  
            \draw[black!50!green] (\w*1.5,\h*1.5-\r)--(\w*1.5,\h*.85)--(\w*3.5,\h*.85)--(\w*3.5,\h*.5+\r);
            \draw[red] (\w*0.5,\h*1.5-\r)--(\w*0.5,\h*1.15)--(\w*5.5,\h*1.15)--(\w*5.5,\h*.5+\r);

            \node at (\w*2.5,\h*0.5) {$\cdots$};
            \node at (\w*4.5,\h*0.5) {$\cdots$};
    
            \node at (\w*-0.45,\h*1.5) {\scriptsize$p$};
            \node at (\w*-0.85,\h*0.5) {\scriptsize$p-1$};
            \node at (\w*0.5,\h*2.5) {\scriptsize$i$};
            \node at (\w*1.5,\h*2.5) {\scriptsize$i+1$};
            \node at (\w*3.5,\h*2.5) {\scriptsize$j$};
            \node at (\w*5.5,\h*2.5) {\scriptsize$j'$};
            \end{scope}
            \begin{scope}[xshift=11cm]
            \node at (-2,1) {\large (c$'$)};
            \draw[gray!50,thin,step=\w] (0,0) grid (6*\w,2*\h);
            \foreach \xx\yy\i\c in {0/1//black,1/1//black,
            3/0//black,5/0//black}
                {
                \draw[fill=\c](\w*.5+\w*\xx,\h*.5+\h*\yy) circle (\r cm);
                \node at (\w*.5+\w*\xx,\h*.5+\h*\yy) {\i};
                }
    
            \draw[black!50!green] (\w*0.5,\h*1.5-\r)--(\w*0.5,\h*.85)--(\w*3.5,\h*.85)--(\w*3.5,\h*.5+\r);
            \draw[red] (\w*1.5,\h*1.5-\r)--(\w*1.5,\h*1.15)--(\w*5.5,\h*1.15)--(\w*5.5,\h*.5+\r);

            \node at (\w*2.5,\h*0.5) {$\cdots$};
            \node at (\w*4.5,\h*0.5) {$\cdots$};
    
            \node at (\w*-0.45,\h*1.5) {\scriptsize$p$};
            \node at (\w*-0.85,\h*0.5) {\scriptsize$p-1$};
            \node at (\w*0.5,\h*2.5) {\scriptsize$i$};
            \node at (\w*1.5,\h*2.5) {\scriptsize$i+1$};
            \node at (\w*3.5,\h*2.5) {\scriptsize$j$};
            \node at (\w*5.5,\h*2.5) {\scriptsize$j'$};
            \end{scope}
            \end{tikzpicture}
            }
            \end{center}
        
    Let $x$ be the ball at site $(p,i+1)$ of $M$ and let $y$ be the ball at site $(p-1,j)$ of $M$ that $x$ pairs down to. Since the statement is true for row $p$, $L_{\eL_i(M)}(p,i) = L_M(p,i+1) = \ell$. Let $x'$ and $y'$ be the balls at sites $(p,i)$ and $(p-1,j)$ of $\eL_i(M)$ respectively. Assume for sake of contradiction that $x'$ does not pair down to $y'$. Then either $x'$ pairs to some new ball that came before $y'$ in the pairing process or some new ball $z'$ pairs to $y'$. To show the first scenario can not occur, first note that the only candidate for such a ball would be at site $(p-1,i)$ of $\eL_i(M)$. If no such ball exists, we are done, so assume $L_{\eL_i(M)}(p-1,i) > 0$. By definition of $\act_i(M)$, $L_{M}(p-1,i) > \ell$ and since the labels of balls greater than $\ell$ on row $p$ of $M$ and $\eL_i(M)$ are the same, the pairing process will be the same implying that $L_{\eL_i(M)}(p-1,i) > \ell$. This prohibits the first scenario since $x'$ has label $\ell$ and therefore will not pair to the ball at site $(p-1,i)$. The second scenario also can not occur since the only candidate for $z'$ would be the ball at site $(p,i+1)$ which has label strictly less than $\ell$. Thus, $L_{\eL_i(M)}(p-1,j) = \ell$. 

    Let $u$ and $v$ be the balls at sites $(p,i)$ and $(p-1,j')$ of $M$ respectively and let $m = L_{M}(p,i)$. By definition of $\act_i(M)$, $m < \ell$, and since the statement is true for row $p$, $L_{\eL_i(M)}(p,i+1) = m$. Let $u'$ and $v'$ be the balls at sites $(p,i+1)$ and $(p-1,j')$ of $\eL_i(M)$ respectively. Assume for sake of contradiction that $u'$ does not pair down to $v'$. Then either $u'$ pairs to some new ball that came before $v'$ in the pairing process or some new ball $w'$ pairs to $v'$. The first scenario can not occur since $u'$ comes across the same balls as $u$ does. The second scenario can not occur since there is no candidate for $w'$. Thus, $L_{\eL_i(M)}(p-1,j') = m$. Lastly, for $k \neq j,j'$, $L_{\eL(i)}(p-1,k) = L_M(p-1,k)$ follows from the fact that the balls at sites $(p,i)$ and $(p,i+1)$ in $\eL_i(M)$ paired to balls at the same sites as they did in $M$ and thus the pairings for every ball other than these two will be the same as they were in $M$, completing the proof.
\end{proof}

It is now immediate that $\eL_i$ acts on the type of a multiline queue by $s_i$ if and only if $M$ is $\eL_i$-full, with the same condition true for $\fR_i$-fullness by definition, thereby proving \cref{cor:Harper's lemma}.

\subsection{Properties of $\type$ under collapsing}\label{sec:rho}

In this section, we will write $\rho(M)\coloneqq \rho_N(M)$, as the recording tableau $\rho_Q(M)$ will not be needed. We focus on analyzing the relationship of $\type(M)$ with $\type(\rho(M))$. We obtain two key properties: \cref{thm:rho type} gives us that the set of coinversions of $\type(M)$ are contained in the set of coinversions of $\type(\rho(M))$, and \cref{lem:M full} gives us that a column crystal operator $g^\leftrightarrow_i$ changes the type in $M$ \emph{only if} it changes the type in $\rho(M)$.

Thus far, we have viewed the type of a multiline queue in terms of strands obtained through pairing of particles via the FM algorithm, with the pairing order within a given row given by the labels of the particles in that row. However, collapsing, and more generally, the row operators $\eD_i$, disregard labels of particles, and thus do not preserve the strands coming from the FM algorithm. Thus from the strand perspective, it becomes difficult to track the type of a multiline queue after applying $\eD_i$'s. To circumvent this obstacle, we introduce an alternative perspective on $\type(M)$ that allows it to be computed without recording particle labels. This perspective is based on the \textbf{cylindrical matching rule}, which retains a useful invariant structure in its interaction with the row crystal operators $\eD_i$. 

\begin{definition}[Cylindrical matching rule]\label{def:comb R}
    For $a,b\subseteq[n]$, consider the bracket word $w=\theta_1(\cw((a,b)))$, where $(a,b)$ is viewed as a two-row generalized multiline queue with row $b$ above row $a$. We regard $w$ as written on a cylinder, so that the rightmost entry is left-adjacent to the leftmost entry. The cylindrical pairing rule recursively matches pairs of open and closed parentheses in $w$ (on the cylinder) with no unmatched parentheses between them. Define $R(a,b)\subseteq a$ to be the set of elements of $a$ corresponding to the cylindrically matched \emph{closed} parentheses in $w$. We extend the notation to a tuple $(b_1,\ldots,b_L)$ by recursively defining
    \[R(b_1,\ldots,b_L)\coloneq R(b_1,R(b_2,\cdots,R(b_{L-1},b_L)\cdots)),\]
    where in the right hand side, $R$ is evaluated from the inside out.
\end{definition}

\begin{remark} The cylindrical pairing rule coincides with the algorithm of Nakayashiki--Yamada~\cite[Rule~3.10]{NY95}, which is used to compute the output of the \emph{combinatorial $R$ matrix} between single-column Kirillov--Reshetikhin crystals. In this interpretation, when $(a,b)$ are two adjacent rows of a multiline queue with $|a|\geq |b|$ and $b$ above $a$, $R(a,b)$ is precisely the set of balls in row $a$ that are paired to by balls in row $b$, independent of any pairing order. 
\end{remark}

It was first observed in \cite{KMO15} that the FM algorithm for multiline queues can be realized as a corner transfer matrix built from the combinatorial $R$ via the map $R(\cdot,\cdot)$. See \cite[Definition~4.1, Theorem~4.4]{MS25} for an extension to generalized multiline queues.\footnote{In the multiline queue setting, our notation $R(a,b)$ corresponds to the quantity $R(a\otimes b)_1$ in \cite{MS25}. When $|b|>|a|$, the definition needs to be modified to include the unmatched balls in row $b$.}

For a multiline queue $M$, denote the set of positions with label $k$ by 
\[I_k(M) \coloneq \{i:\type(M)_i = k\},\qquad I_{\geq k}(M)\coloneq \bigcup_{j\geq k} I_j(M)=\{i:\type(M)_i \geq k\}\,.\]
 Since only $\type(M)$ is needed to compute $I_k(M)$, we may also write $I_k(\type(M))$ for this set. 

The reader may verify, by induction on rows, that for a multiline queue $M=(B_1,\ldots,B_L)$, 
    \begin{equation}\label{eq:R}
    I_L(M) = R(B_1,\ldots, B_L)\,.
    \end{equation}
For $k < L$, let $\pi^{(k)}(M)$ denote the truncated multiline queue obtained by restricting $M$ to the first $k$ rows. Note that $\cup_{k\geq 1}I_k(M)=B_1$, the set of balls in the bottom row. For $k>1$, repeating the process on the multiline queue $\pi^{(k)}(M)$ and using
\[
I_{\geq k}(M)=I_{k}(\pi^{(k)}(M))\,,
\]
 we obtain
    \begin{align}I_{\geq k}(M) &= R(B_1,\ldots, B_{k})\,. 
    \label{eq:R2}
\end{align}
Since $\type(M)$ can be recovered from the sets $\{I_k(M)\}_{k\geq 1}$ (equivalently, the sets $\{I_{\geq k}(M)\}_{k\geq 1}$), it follows that $\type(M)$ can be computed strictly in terms of the map $R(\cdot,\cdot)$. See \cref{ex:R}.

\begin{example}
    \label{ex:R}
    We compute $I_k(M)$ for $1\leq k\leq 5$ for the multiline queue $M=(B_1,B_2,B_3)=(\{1,3,4,5\},\{2,3,4\},\{3,5\})$ in \cref{ex:FM algorithm} with $\type(M)=(1,0,3,3,2)$. We have:
\[
        R(B_2,B_3)=\{2,3\},\qquad R(B_1,B_2)=\{3,4,5\},\qquad R(B_1,R(B_2,B_3))=\{3,4\}.
\]
    Thus
\begin{align*}
I_{\geq 3}(M)
  &= R(B_1,R(B_2,B_3))= \{3,4\}\,,&
  \qquad I_3(M)=\{3,4\}\,, \\[0.5ex]
I_{\geq2}(M)
  &= R(B_1,B_2) = \{3,4,5\}\,,&
  \qquad I_2(M)=\{5\}\,, \\[0.5ex]
I_{\geq1}(M)
  &= B_1 = \{1,3,4,5\}\,,&
  \qquad I_1(M)=\{1\}\,.
\end{align*}
\end{example}

In seeking a similar expression for $\type(\rho(M))$, cylindrical matching is no longer relevant, and we return instead to classical matching. Although we cannot compute the full $\type(\rho(M))$ in this way, we will obtain a simple expression that parallels \eqref{eq:R} for $I_L(\rho(M))$, where $L$ is the largest label in $\type(M)$. 

\begin{definition}
    For $a,b \subseteq [n]$, let $\theta(a,b)\subseteq a$ denote the set of elements in $a$ that are \emph{classically  matched} in $\theta_1(\cw((a,b)))$ when $(a,b)$ is viewed as a two-row generalized multiline queue. We extend the definition for a tuple $(b_1,\ldots,b_L)$ by recursively defining
    \[\theta(b_1,\ldots,b_L)\coloneq\theta(b_1,\theta(b_2,\ldots,\theta(b_{L-1},b_L)\ldots)),\]
    where $\theta$ is evaluated from inside out. 
    Notice that this construction is well-defined for any generalized multiline queue in $\mathcal{M}_{(2)}$. To emphasize this generality, we introduce the notation 
    \[b_1\otimes b_2\otimes\cdots\otimes b_L \coloneq (b_1,\ldots,b_L)\qquad\text{and}\qquad \theta(b_1\otimes\cdots\otimes b_L)\coloneq \theta(b_1,\ldots,b_L).\] 
    We also write $\eDs_i(b_1\otimes \cdots \otimes b_L)$ for the result of applying the classical raising operator $\eDs_i$ to the adjacent pair $b_i\otimes b_{i+1}$, leaving all other components unchanged. See \cref{ex:theta}.
\end{definition}

\begin{example}\label{ex:theta}
For some examples, 
\[\theta(\{1,3,4\}\otimes\{2,5\})=\{3\}\,, \qquad \theta(\{3,5\}\otimes\{1,2,4\}\otimes\{3,4\})=\{5\}\,,
\] and for $b_1\otimes b_2\otimes b_3=\{3,5\}\otimes\{1,2,4\}\otimes\{3,4\}$, we have
\[\eDs_1(b_1\otimes b_2\otimes b_3)=\{1,3,5\}\otimes\{2,4\}\otimes\{3,4\}\,,\qquad \eDs_2(b_1\otimes b_2\otimes b_3)=\{3,5\}\otimes\{1,2,3,4\}\otimes\{4\}\,.\]
\end{example}

We first prove a technical lemma for $\theta$ that will be fundamental in our recursive arguments. 

\begin{lemma}
    \label{lem:three row}
    Let $a,b,c\in[n]$. Then
    \begin{equation}\label{eq:three term}
    \theta(a\otimes b\otimes c)=\theta(\eDs_1(a\otimes b\otimes c)).
    \end{equation}
    Furthermore, for any $b_1,\ldots,b_L\subseteq [n]$, we have 
    \begin{equation}\label{eq:ei invariant}
    \theta(b_1\otimes\ldots\otimes b_L)=\theta(\eDs_i(b_1\otimes \ldots\otimes b_L))\qquad\text{for any $1\leq i<L$}\,.
    \end{equation}
\end{lemma}

See \cref{ex:three row} for an illustration of the lemma.

\begin{proof}
    To prove the identity \eqref{eq:three term}, decompose $b$ into its classically matched and unmatched elements in $\theta_1((a,b))$, denoted $b^m$ and $b^u$, respectively, and write $b=b^m \sqcup b^u$. Then $\eDs_1(a\otimes b)=(a\sqcup b^u)\otimes b^m$, 
    so we can write $\theta(\eDs_1(a\otimes b\otimes c))$ as $\theta((a\sqcup b^u)\otimes b^m \otimes c)$.

    We first observe that if $\theta((d\sqcup d')\otimes e)\cap d'=\emptyset$, then $\theta((d\sqcup d')\otimes e)=\theta(d\otimes e)$. Therefore, since the action of $\eDs_1$ can be inverted, all elements in $b^u$ are unmatched in row 1 of $\theta_1((a\sqcup b^u)\otimes b^m)$, and so $\theta((a\sqcup b^u)\otimes b')\cap b^u=\emptyset$ for any $b'\subseteq b^m$, including $b'=\theta(b^m\otimes c)$. Therefore,
   \begin{equation}\label{eq:unm}
   \theta((a\sqcup b^u)\otimes b^m\otimes c)=\theta(a\otimes b^m\otimes c).
   \end{equation}

Thus the statement reduces to showing
 \begin{equation}\label{eq:abc}\theta(a\otimes b\otimes c)=\theta(a\otimes b^m\otimes c),
\end{equation}
namely that the unmatched balls in the intermediate row do not impact the set of matched positions in the bottom row in strings originating from the top row.

First we show that $\theta(a\otimes b^m \otimes c)\subseteq \theta(a\otimes b\otimes c)$. Suppose $x\in c$ matches with $y\in b^m$, which matches with $z\in a$ for an element $z\in\theta(a\otimes b^m\otimes c)$. Then there is no  $y'\in b^m\setminus \theta(b^m\otimes c)$ such that $x\leq y'<y$ (otherwise the string of matchings would be $x\to y'\to z$). 
If $y\in \theta(b\otimes c)$, then $x\to y\to z$ is a string of matchings in $\theta(a\otimes b \otimes c)$ as well. If $y\not\in \theta(b\otimes c)$, then in $\theta(a\otimes b\otimes c)$, $x$ matches with some $y''\in \theta(b\otimes c)$ where $x\leq y''<y$ and hence $y''\in b^u$. As $y''\in b^u$, there is no $z''\in a\setminus \theta(a\otimes b)$ such that $y''\leq z''<y$, and thus $y''$ must match with $z$ in $\theta(a\otimes\theta(b\otimes c))$, giving the string of matchings $x\to y''\to z$ in $\theta(a\otimes b\otimes c)$. Thus in either case, $z\in\theta(a\otimes b\otimes c)$, proving the containment of the two sets. 

Now we show the converse. Suppose $x\in c$ matches with $y\in b$, which matches with $z\in a$ for an element $z\in\theta(a\otimes\theta(b\otimes c))$.  The string of matchings $x\to y\to z$ implies there does not exist $x'\in c$ such that $x<x'\leq y$ (otherwise the matchings would be $x'\to y\to z$), and there does not exist $(x'',y')\in c\otimes b$ such $y<x''\leq y'\leq z$ (otherwise the matchings would be $x''\to y'\to z$). Consider the element $y$. If $y\in b^m$, we still have the matchings $x\to y\to z$ in $a\otimes b^m \otimes c$. If $y\in b^u$, there must be some $y'\in b^m$ that matches with $z$ in $a\otimes b$ (since $z\in\theta(a\otimes b)$), and hence $y<y'\leq z$). But then $x\to y'\to z$ is a string of matchings in $a\otimes b^m\otimes c$. Thus in either case, $z\in \theta(a\otimes b^m \otimes c)$. Since this holds for each element $z\in\theta(a\otimes \theta(b\otimes c))$, we obtain the reverse containment, giving equality \eqref{eq:abc}, which in turn yields the desired three-row equality.

Finally, we prove \eqref{eq:ei invariant} using the recursive structure of the $\theta$ map. Notice that the case $L=2$ is obtained by setting $b=c$ in \eqref{eq:three term}. For $L\geq 3$, we have
\begin{align*}\theta(\eDs_1(b_1\otimes\cdots \otimes b_L))&= \theta(\eDs_1(b_1\otimes b_2)\otimes b_2\otimes\cdots\otimes b_L)\\
&=\theta(\eDs_1(b_1\otimes b_2\otimes \theta(b_3\otimes\cdots\otimes b_L)))\\
&=\theta(b_1\otimes b_2\otimes \theta(b_3\otimes\cdots\otimes b_L))\\
&=\theta(b_1\otimes\cdots\otimes b_L).
\end{align*}
Here, the second equality is due to the fact that evaluating the $\theta$ function on $(b_3\otimes \cdots\otimes b_L)$ is independent of the first two rows, and thus can be completed before the action of $\eDs_1$ on the pair $b_1\otimes b_2$. The third equality invokes \eqref{eq:three term} for $a=b_1$, $b=b_2$, and $c=\theta(b_3\otimes\cdots\otimes b_L)$.  Now we use the fact that $\eDs_i(b_1\otimes\cdots \otimes b_L)=b_1\otimes\cdots\otimes b_{i-1}\otimes \eDs_1(b_i\otimes \cdots\otimes b_L)$ to obtain \eqref{eq:ei invariant}.
\end{proof}

\begin{example}
    \label{ex:three row}
    For $a\otimes b\otimes c=\{1,3,6\}\otimes\{1,2,3,5,6\}\otimes\{2,4,6\}$, let $a'\otimes b'\otimes c=\eDs_1(a\otimes b\otimes c)$. We show the strings of matched balls (highlighted below) that give $\theta(b\otimes c)=\{2,5,6\}$ and $\theta(a\otimes \theta(b\otimes c))=\{3,6\}$ on the left, and the strings of matched balls that give $\theta(b'\otimes c)=\{3,6\}$ and $\theta(a'\otimes \theta(b\otimes c))=\{3,6\}$ on the right. Observe that even though $\theta(b\otimes c)\neq \theta(b'\otimes c)$, we still have that $\theta(a\otimes b\otimes c)=\theta(a'\otimes b'\otimes c)=\{3,6\}$.
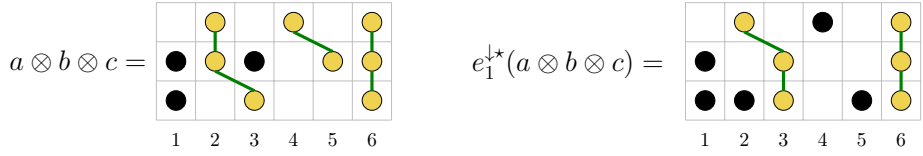
     \begin{center}
\resizebox{!}{2cm}{
\begin{tikzpicture}[scale=0.7]
\def \w{1};
\def \h{1};
\def \r{0.25};
    
\begin{scope}[xshift=0cm]
\node at (-2,1.5) {\Large $a\otimes b\otimes c=$};
\draw[gray!50,thin,step=\w] (0,0) grid (6*\w,3*\h);
\foreach \i in {1,...,6}
{
\node at (\w*\i-\w*.5,-0.5) {\small \i};
}
\foreach \xx\yy\i\c in {1/2/3/\highlight,3/2/3/\highlight,5/2/3/\highlight,0/1/3/black,1/1/3/\highlight,2/1/3/black,4/1/2/\highlight,5/1/3/\highlight,0/0/1/black,2/0/3/\highlight,5/0/3/\highlight}
    {
    \draw[fill=\c] (\w*.5+\w*\xx,\h*.5+\h*\yy) circle (\r cm);
    }

\draw[black!50!green,ultra thick] (\w*1.5,\h*2.5-\r)--(\w*1.5,\h*1.5+\r) (\w*1.5,\h*1.5-\r)--(\w*2.5,\h*0.5+\r);
\draw[black!50!green,ultra thick] (\w*3.5,\h*2.5-\r)--(\w*4.5,\h*1.5+\r);
\draw[black!50!green,ultra thick] (\w*5.5,\h*2.5-\r)--(\w*5.5,\h*1.5+\r) (\w*5.5,\h*1.5-\r)--(\w*5.5,\h*0.5+\r);
\end{scope}

\begin{scope}[xshift=13.5cm]
\node at (-3,1.5) {\Large $\eDs_1(a\otimes b\otimes c)=$};
\draw[gray!50,thin,step=\w] (0,0) grid (6*\w,3*\h);
\foreach \i in {1,...,6}
{
\node at (\w*\i-\w*.5,-0.5) {\small \i};
}
\foreach \xx\yy\i\c in {1/2/3/\highlight,3/2/3/black,5/2/3/\highlight,0/1/3/black,1/0/3/black,2/1/3/\highlight,4/0/2/black,5/1/3/\highlight,0/0/1/black,2/0/3/\highlight,5/0/3/\highlight}
    {
    \draw[fill=\c] (\w*.5+\w*\xx,\h*.5+\h*\yy) circle (\r cm);
    }
\draw[black!50!green,ultra thick] (\w*1.5,\h*2.5-\r)--(\w*2.5,\h*1.5+\r) (\w*2.5,\h*1.5-\r)--(\w*2.5,\h*0.5+\r);
\draw[black!50!green,ultra thick] (\w*5.5,\h*2.5-\r)--(\w*5.5,\h*1.5+\r) (\w*5.5,\h*1.5-\r)--(\w*5.5,\h*0.5+\r);
\end{scope}

\end{tikzpicture}
}
\end{center}
\end{example}

\begin{prop}\label{prop:label L through collapsing}
    For a multiline queue $M = (B_1,\ldots,B_L)$ of height $L$,
        \begin{align*}
            I_L(\rho(M)) = \theta(B_1 \otimes \cdots \otimes  B_L).
        \end{align*}
\end{prop}

\begin{proof}
Our proof is by induction on $L$, where the base case $L=2$ indeed corresponds to $I_L(\rho(B_1\otimes B_2))=I_L(\eDs_1(B_1\otimes B_2))=\theta(B_1\otimes B_2)$, from \eqref{eq:ei invariant}. 

We will use the shorthand $R(N)$ and $\theta(N)$ to mean $R(D_1,\ldots,D_L)$ and $\theta(D_1,\ldots,D_L)$, respectively, for any generalized multiline queue $N=(D_1,\ldots,D_L)$.

Our inductive arguments thus far entailed adding a row to the top of a multiline queue; here, we will add a row to the bottom. Consider the multiline queue $M'=(B_2,\ldots,B_L)$ consisting of the top $L-1$ rows of $M$. As the operators $\{\eDs_i\}$ satisfy the braid relations (see \cite[Theorem 4.6]{MV24}), we may write $\rho$ as follows (recall that $\eDs_{[b,a]}=\eDs_a\cdots\eDs_{b-1}\eDs_b$ and $\eDs_{[a,b]}=\eDs_b\eDs_{b-1}\cdots\eDs_a$ when $a<b$):
\[
\rho(M)=\eDs_{[1,L-1]}\eDs_{[L-1,2]}\cdots\eDs_{[3,2]}\eDs_{[2,2]}(M)= 
\eDs_{L-1}\cdots\eDs_1\Big[\eDs_{[L-1,2]}\cdots\eDs_{[3,2]}\eDs_{[2,2]}(M)\Big].
\]
 We leave it as an exercise to the reader to verify the equality above by repeatedly applying the following identity to \eqref{eq:rho op} for $r=2,\ldots,L-1$:
\[\eDs_{[r,1]}\eDs_{[1,r-1]}=(\eDs_1\cdots \eDs_r)(\eDs_{r-1}\cdots\eDs_1)=(\eDs_r\cdots\eDs_1)(\eDs_2\cdots\eDs_r)=\eDs_{[1,r]}\eDs_{[r,2]}.\]

We have that
$\rho(M')$ is given by rows 2 through $L$ of $\eDs_{[L-1,2]}\cdots\eDs_{[3,2]}\eDs_{[2,2]}(M)$. Writing $\rho(M')=(C_2,\ldots,C_L)$, we then obtain
\[\rho(M)=\eDs_{L-1}\cdots \eDs_1((B_1,C_2,\ldots,C_L))\,.\] 
We will construct a sequence of partially-collapsed multiline queues $N^{(k)}$ corresponding to sequential applications of the $\eDs_i$ from bottom to top:
\[N^{(0)}\coloneq (B_1,C_2,\ldots,C_L),\qquad N^{(k)}\coloneq \eDs_k(N^{(k-1)})\qquad\text{for $1\leq k\leq L$},\]
so that $\rho(M)=N^{(L-1)}$. The careful reader may observe that $N^{(k)}$ is a generalized multiline queue: since only rows 1 through $k+1$ have been collapsed, row $k+2$ may have more particles than row $k+1$. Nonetheless, $\theta(N^{(k)})$ is still well-defined. We construct a chain of equalities by iteratively invoking \eqref{eq:ei invariant} of \cref{lem:three row} for each operator $\eDs_i$ for $i=1,\ldots,L-1$ to obtain
\begin{equation}\label{eq:3}
\theta(N^{(0)})=\theta(N^{(1)})=\cdots=\theta(N^{(L-1)})=\theta(\rho(M)).
\end{equation}
By induction, assume $I_{L-1}(\rho(M'))=\theta(M')$. We observe that if $a,b\subseteq[n]$ are two rows such that $(a,b)$ has no wrapping pairings, then $R(a,b)=\theta(a\otimes b)$. Thus in a nonwrapping multiline queue, classical pairing coincides with cylindrical pairing, so we may write 
\[I_L(\rho(M))=R(\rho(M))=\theta(\rho(M))\,,\]   
which equals $\theta(N^{(0)})$ by \eqref{eq:3}. By induction, we have that $I_{L-1}(\rho(M'))=\theta(\rho(M'))=\theta(M')$, and thus $I_L(\rho(M))=\theta(N^{(0)})=\theta(B_1,\theta(M'))=\theta(M)$, as claimed. 
\end{proof}

\begin{remark}
    \cref{prop:label L through collapsing} is true for the generalized multiline queues studied in \cite{MV24}, as well.
\end{remark}

Next we prove a property of collapsing a row above a nonwrapping multiline queue, as illustrated in \cref{ex:one row}. For $1\leq \ell\leq L$, let $\rho^{(\ell)}(M)=\rho(\pi^{(\ell)}(M))$ denote the collapsing of the first $\ell$ rows of $M$. 
    \begin{lemma}\label{lemma:onerowcollapse}
        Let $\lambda$ be a partition with $\lambda_1 = L$, and $M\in\MLQ_\lambda$ be a multiline queue. Then, writing $\type(\rho (M))=\alpha$ and $\type(\rho^{(L-1)}(M))=\beta$, we have 
        $\alpha_i-\beta_i\in\{0,1\}$ for each $i$.
    \end{lemma}
    
    \begin{proof}    
        We will think of $\rho(M)$ as the collapsing of the $L$th row of $M$ onto the partial collapse $\rho^{(L-1)}(M)$ of the bottom $L-1$ rows. Define 
        \[
        M' = (C_1,\ldots,C_{L-1},\emptyset,B_L)\,,
        \]
        where $\rho^{(L-1)}(M)=(C_1,\ldots,C_{L-1})$ and $B_L$ is the $L$th row of $M$. Equivalently, 
        \[M'=\fRs_L\eDs_{[L-2,1]}\eDs_{[L-3,1]}...\eDs_{[1,1]}(M)\,,
        \]
        where $\fRs_L$ lifts all balls in row $L$ to row $L+1$. Then, by \eqref{eq:rho op}, 
        \[
\rho^{(L-1)}(M)=\pi^{(L-1)}(M')
\qquad\text{and}\qquad
\rho(M)=\eDs_{[L,1]}(M')\,.
\]

Consider the intermediate steps corresponding to applying the operators $\eDs_{k}$ one at a time. Define 
        \[
        M^{(L)}\coloneq M',\qquad\text{and}\quad  M^{(k)}\coloneq \eDs_{k+1}(M^{(k+1)})\qquad \text{for $0\leq k\leq L-1$}\,,
        \]
        so that $M^{(0)}=\rho(M)$. We wish to compare the types $\beta=\type(\pi^{(L-1)}(M^{(L)}))$ and $\alpha=\type(M^{(0)})$ by sequentially comparing the types of the intermediate configurations $M^{(k)}$ and $M^{(k-1)}$ for each $k$. However, the intermediate configurations $M^{(k)}$ are not necessarily valid multiline queues, so $\type(\cdot)$ isn't well defined on them. To get around this, for each $M^{(k)}$, we define a ``maximal'' non-wrapping multiline queue $N^{(k)}$ contained within it, as follows. 

        For each $1\leq k\leq L$, define $N^{(k)}$ to be the multiline queue obtained by keeping $M^{(k)}$ unchanged at all rows $i\neq k+1$, and in row $k+1$ keeping only those balls that are matched in the bracketing word $\theta_{k}(\cw(M^{(k)}))$ (equivalently, these are the balls that are not moved by $\eDs_{k}$ applied to $M^{(k)}$). Set $N^{(0)}\coloneq M^{(0)}$, and note that $N^{(L)}=\rho^{(L-1)}(M)$. By construction each $\eDs_i$ acts trivially on $N^{(k)}$, hence $N^{(k)}$ is nonwrapping. 
        
       Let $\alpha^{(k)} = \type(N^{(k)})$, so that $\alpha^{(L)} = \beta$ and $\alpha^{(0)} = \alpha$. We compare $\alpha^{(k)}$ and $\alpha^{(k-1)}$ via the label sets $I_{\geq j}(N^{(k)})$ and $I_{\geq j}(N^{(k-1)})$ for $1\leq j\leq L$. For a nonwrapping multiline queue $N$, by \eqref{eq:R2}  we have
\[
I_{\geq j}(N)=\theta(\pi^{(j)}(N))\,.
\]
By construction of $N^{(k)}$ and \eqref{eq:unm} applied at row $k+1$, $\theta(\pi^{(j)}(N^{(k)}))=\theta(\pi^{(j)}(M^{(k)}))$ for all $k,j$. Thus, 
\[I_{\geq j}(N^{(k)})=\theta(\pi^{(j)}(M^{(k)})).\]
If $j>k$, then by \cref{prop:label L through collapsing} we obtain
\[
I_{\geq j}(N^{(k-1)})=\theta(\pi^{(j)}(\eDs_k(M^{(k)})))=\theta(\pi^{(j)}(M^{(k)}))=I_{\geq j}(N^{(k)}).
\]
If $j\leq k-1$, then $N^{(k)}$ and $N^{(k-1)}$ agree on rows $k-1$ and below, so again
$I_{\geq j}(N^{(k-1)})=I_{\geq j}(N^{(k)})$.
Thus, the only possible change occurs for $j=k$, which we now examine. In this case, the $k$th row of $M^{(k-1)}=\eDs_k M^{(k)}$ contains the $k$th row of $M^{(k)}$, hence
\[
I_{\geq k}(N^{(k)})=\theta(\pi^{(k)}(M^{(k)}))\subseteq \theta(\pi^{(k)}(M^{(k-1)})) = I_{\geq k}(N^{(k-1)}).
\]

This inclusion means precisely that when we pass from $\alpha^{(k)}$ to $\alpha^{(k-1)}$, no entry can decrease, and the only entries that can change are those equal to $k-1$, which may increase to $k$. Note that these entries correspond to the newly added elements $I_{\geq k}(N^{(k-1)})\setminus I_{\geq k}(N^{(k)})$, and they are indeed a subset of the entries in $I_{k-1}(N^{(k)})$:
\[
I_{\geq k}(N^{(k-1)})\setminus I_{\geq k}(N^{(k)})\subseteq I_{\geq k-1}(N^{(k-1)})\setminus I_{\geq k}(N^{(k)})=I_{\geq k-1}(N^{(k)})\setminus I_{\geq k}(N^{(k)})=I_{k-1}(N^{(k)}).
\]
Since each $k$ occurs exactly once in the collapse, each index $1\leq i\leq n$ can be increased at most once overall. Therefore
\[
\alpha_i-\beta_i\in\{0,1\}\qquad\text{for all }i,
\]
as claimed.
\end{proof}

    \begin{example}\label{ex:one row}
    In the fourth image of \cref{fig:collapsing steps}, we see $\rho^{(4)}(M)$ in black with type $\beta=(1,1,4,4,3,2)$. In the fifth image, we see $\rho(M)$ with type $\alpha=(1,1,4,5,3,3)$, and indeed $\alpha_i-\beta_i\in\{0,1\}$ for each $i$.
    \end{example}

    \begin{example}
        \label{ex:ghost balls}
       For $L=4$, we show an example of $M$ and $\eDs_{[2,1]}\eDs_{[1,1]}(M)$ below:
        \begin{center}
    \resizebox{!}{2cm}{
    \begin{tikzpicture}[scale=0.7]
    \def \w{1};
    \def \h{1};
    \def \r{0.25};

    \begin{scope}[xshift=0cm]
    \node at (-1.5,2) {\large $M=$};
    \draw[gray!50,thin,step=\w] (0,0) grid (6*\w,4*\h);
    \foreach \xx\yy\i\c in {2/3/white,4/3/white,5/3/white,0/2/white,2/2/white,1/1/white,2/1/white,3/1/white,5/2/white,0/0/white, 1/0/white,2/0/white,4/0/white,5/1/white}
    {
    \draw[fill=black] (\w*.5+\w*\xx,\h*.5+\h*\yy) circle (\r cm);
    }
    \end{scope}
    
    \begin{scope}[xshift=15cm]
    \node at (-2.5,2) {\large $\eDs_{[2,1]}\eDs_{[1,1]}(M)=$};
    \draw[gray!50,thin,step=\w] (0,0) grid (6*\w,4*\h);
    \foreach \xx\yy\i\c in {2/3/white,4/3/white,5/3/white,0/2/white,2/2/white,1/1/white,2/1/white,3/1/white,5/1/white,0/0/white,1/0/white,2/0/white,4/0/white,5/0/white}
    {
    \draw[fill=black] (\w*.5+\w*\xx,\h*.5+\h*\yy) circle (\r cm);
    }
    \end{scope}
    
    \end{tikzpicture}
    }
    \end{center}
    We show the sequence of configurations $M^{(k)}$, with multiline queues $N^{(k)}$ shown as the black balls, and the balls in $M^{(k)}\setminus N^{(k)}$ highlighted. Accordingly, the strands reflect the pairings in $N^{(k)}$, ignoring the highlighted balls.
    
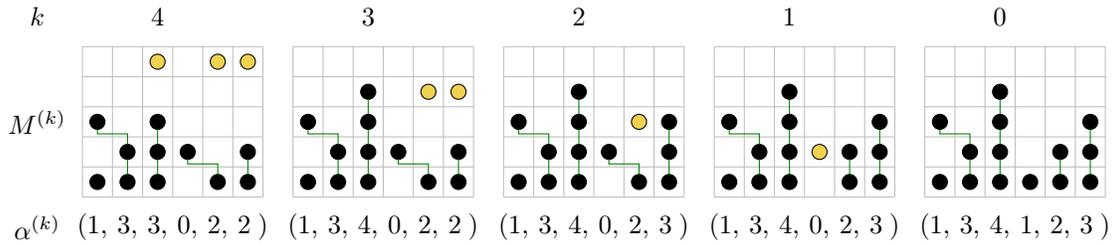
\begin{center}
        \begin{tikzpicture}[scale=0.4]        
    \def \w{1};
        \def \h{1};
        \def \r{0.25};
        \node at (-1.5,6) {$k$};
        \node at (-1.5,2.5) {$M^{(k)}$};
        \node at (-1.5,-1) {$\alpha^{(k)}$};
        \foreach \i\x in {4/0,3/1,2/2,1/3,0/4}
        {
        \node at (2.5+7*\x,6) {\i};
        }
        \begin{scope}[xshift=0cm]
        \foreach \i in {0,...,5}
        {
        \draw[gray!50] (0,\i*\h)--(\w*6,\i*\h);
        }
        \foreach \i in {0,...,6}
        {
        \draw[gray!50] (\w*\i,0)--(\w*\i,5*\h);
        }
        
        \foreach \xx\yy\c in {2/4/\highlight,4/4/\highlight,5/4/\highlight,0/2/black,2/2/black,1/1/black,2/1/black,3/1/black,5/1/black,0/0/black,1/0/black,2/0/black,4/0/black,5/0/black}
        {
        \draw[fill=\c] (\w*.5+\w*\xx,\h*.5+\h*\yy) circle (\r cm);
        }

        \draw[black!50!green] (\w*0.5,\h*2.5-\r)--(\w*0.5, \h*2.1)--(\w*1.5,\h*2.1)--(\w*1.5,\h*1.5+\r);
        \draw[black!50!green] (\w*2.5,\h*2.5-\r)--(\w*2.5,\h*1.5+\r);

        \draw[black!50!green] (\w*1.5,\h*1.5-\r)--(\w*1.5,\h*0.5+\r);
        \draw[black!50!green] (\w*2.5,\h*1.5-\r)--(\w*2.5,\h*0.5+\r);
        \draw[black!50!green](\w*3.5,\h*1.5-\r)--(\w*3.5,\h*1.1)--(\w*4.5,\h*1.1)--(\w*4.5,\h*0.5+\r);
        \draw[black!50!green] (\w*5.5,\h*1.5-\r)--(\w*5.5,\h*0.5+\r);
        \node at (0,-1) {(};
        \node at (6,-1) {)};
        \foreach \x\i in {0.5/1,1.5/3,2.5/3,3.5/0,4.5/2}
        {
        \node at (\x,-1) {\i,};
        }
        \node at (5.5,-1) {2\textcolor{white}{,}};
        \end{scope}
    
        
        \begin{scope}[xshift=7cm]
        \foreach \i in {0,...,5}
        {
        \draw[gray!50] (0,\i*\h)--(\w*6,\i*\h);
        }
        \foreach \i in {0,...,6}
        {
        \draw[gray!50] (\w*\i,0)--(\w*\i,5*\h);
        }
        
        \foreach \xx\yy\c in {2/3/black,4/3/\highlight,5/3/\highlight,0/2/black,2/2/black,1/1/black,2/1/black,3/1/black,5/1/black,0/0/black,1/0/black,2/0/black,4/0/black,5/0/black}
        {
        \draw[fill=\c] (\w*.5+\w*\xx,\h*.5+\h*\yy) circle (\r cm);
        }

        \draw[black!50!green] (\w*2.5,\h*3.5-\r)--(\w*2.5,\h*2.5+\r);
        \draw[black!50!green] (\w*0.5,\h*2.5-\r)--(\w*0.5, \h*2.1)--(\w*1.5,\h*2.1)--(\w*1.5,\h*1.5+\r);
        
        \draw[black!50!green] (\w*2.5,\h*2.5-\r)--(\w*2.5,\h*1.5+\r);

        \draw[black!50!green] (\w*1.5,\h*1.5-\r)--(\w*1.5,\h*0.5+\r);
        \draw[black!50!green] (\w*2.5,\h*1.5-\r)--(\w*2.5,\h*0.5+\r);
        \draw[black!50!green](\w*3.5,\h*1.5-\r)--(\w*3.5,\h*1.1)--(\w*4.5,\h*1.1)--(\w*4.5,\h*0.5+\r);
        \draw[black!50!green] (\w*5.5,\h*1.5-\r)--(\w*5.5,\h*0.5+\r);
        \node at (0,-1) {(};
        \node at (6,-1) {)};
        \foreach \x\i in {0.5/1,1.5/3,2.5/4,3.5/0,4.5/2}
        {
        \node at (\x,-1) {\i,};
        }
        \node at (5.5,-1) {2\textcolor{white}{,}};
        \end{scope}
    
        
        \begin{scope}[xshift=14cm]
        \foreach \i in {0,...,5}
        {
        \draw[gray!50] (0,\i*\h)--(\w*6,\i*\h);
        }
        \foreach \i in {0,...,6}
        {
        \draw[gray!50] (\w*\i,0)--(\w*\i,5*\h);
        }
        
        \foreach \xx\yy\c in {2/3/black,4/2/\highlight,5/2/black,0/2/black,2/2/black,1/1/black,2/1/black,3/1/black,5/1/black,0/0/black,1/0/black,2/0/black,4/0/black,5/0/black}
        {
        \draw[fill=\c] (\w*.5+\w*\xx,\h*.5+\h*\yy) circle (\r cm);
        }

        \draw[black!50!green] (\w*2.5,\h*3.5-\r)--(\w*2.5,\h*2.5+\r);
        \draw[black!50!green] (\w*0.5,\h*2.5-\r)--(\w*0.5, \h*2.1)--(\w*1.5,\h*2.1)--(\w*1.5,\h*1.5+\r);
        \draw[black!50!green] (\w*2.5,\h*2.5-\r)--(\w*2.5,\h*1.5+\r);
        \draw[black!50!green] (\w*5.5,\h*2.5-\r)--(\w*5.5,\h*1.5+\r);

        \draw[black!50!green] (\w*1.5,\h*1.5-\r)--(\w*1.5,\h*0.5+\r);
        \draw[black!50!green] (\w*2.5,\h*1.5-\r)--(\w*2.5,\h*0.5+\r);
        \draw[black!50!green](\w*3.5,\h*1.5-\r)--(\w*3.5,\h*1.1)--(\w*4.5,\h*1.1)--(\w*4.5,\h*0.5+\r);
        \draw[black!50!green] (\w*5.5,\h*1.5-\r)--(\w*5.5,\h*0.5+\r);
        \node at (0,-1) {(};
        \node at (6,-1) {)};
        \foreach \x\i in {0.5/1,1.5/3,2.5/4,3.5/0,4.5/2}
        {
        \node at (\x,-1) {\i,};
        }
        \node at (5.5,-1) {3\textcolor{white}{,}};
        \end{scope}
    
        
        \begin{scope}[xshift=21cm]
        \foreach \i in {0,...,5}
        {
        \draw[gray!50] (0,\i*\h)--(\w*6,\i*\h);
        }
        \foreach \i in {0,...,6}
        {
        \draw[gray!50] (\w*\i,0)--(\w*\i,5*\h);
        }
        
        \foreach \xx\yy\c in {2/3/black,4/1/black,5/2/black,0/2/black,2/2/black,1/1/black,2/1/black,3/1/\highlight,5/1/black,0/0/black,1/0/black,2/0/black,4/0/black,5/0/black}
        {
        \draw[fill=\c] (\w*.5+\w*\xx,\h*.5+\h*\yy) circle (\r cm);
        }

        \draw[black!50!green] (\w*2.5,\h*3.5-\r)--(\w*2.5,\h*2.5+\r);
        \draw[black!50!green] (\w*0.5,\h*2.5-\r)--(\w*0.5, \h*2.1)--(\w*1.5,\h*2.1)--(\w*1.5,\h*1.5+\r);
        \draw[black!50!green] (\w*2.5,\h*2.5-\r)--(\w*2.5,\h*1.5+\r);
        \draw[black!50!green] (\w*5.5,\h*2.5-\r)--(\w*5.5,\h*1.5+\r);

        \draw[black!50!green] (\w*1.5,\h*1.5-\r)--(\w*1.5,\h*0.5+\r);
        \draw[black!50!green] (\w*2.5,\h*1.5-\r)--(\w*2.5,\h*0.5+\r);
        \draw[black!50!green](\w*4.5,\h*1.5-\r)--(\w*4.5,\h*0.5+\r);
        \draw[black!50!green] (\w*5.5,\h*1.5-\r)--(\w*5.5,\h*0.5+\r);
        \node at (0,-1) {(};
        \node at (6,-1) {)};
        \foreach \x\i in {0.5/1,1.5/3,2.5/4,3.5/0,4.5/2}
        {
        \node at (\x,-1) {\i,};
        }
        \node at (5.5,-1) {3\textcolor{white}{,}};
        \end{scope}
    
        
        \begin{scope}[xshift=28cm]
        \foreach \i in {0,...,5}
        {
        \draw[gray!50] (0,\i*\h)--(\w*6,\i*\h);
        }
        \foreach \i in {0,...,6}
        {
        \draw[gray!50] (\w*\i,0)--(\w*\i,5*\h);
        }
        
        \foreach \xx\yy\c in {2/3/black,4/1/black,5/2/black,0/2/black,2/2/black,1/1/black,2/1/black,3/0/black,5/1/black,0/0/black,1/0/black,2/0/black,4/0/black,5/0/black}
        {
        \draw[fill=\c] (\w*.5+\w*\xx,\h*.5+\h*\yy) circle (\r cm);
        }

        \draw[black!50!green] (\w*2.5,\h*3.5-\r)--(\w*2.5,\h*2.5+\r);
        \draw[black!50!green] (\w*0.5,\h*2.5-\r)--(\w*0.5, \h*2.1)--(\w*1.5,\h*2.1)--(\w*1.5,\h*1.5+\r);
        \draw[black!50!green] (\w*2.5,\h*2.5-\r)--(\w*2.5,\h*1.5+\r);
        \draw[black!50!green] (\w*5.5,\h*2.5-\r)--(\w*5.5,\h*1.5+\r);

        \draw[black!50!green] (\w*1.5,\h*1.5-\r)--(\w*1.5,\h*0.5+\r);
        \draw[black!50!green] (\w*2.5,\h*1.5-\r)--(\w*2.5,\h*0.5+\r);
        \draw[black!50!green] (\w*4.5,\h*1.5-\r)--(\w*4.5,\h*0.5+\r);
        \draw[black!50!green] (\w*5.5,\h*1.5-\r)--(\w*5.5,\h*0.5+\r);
        \node at (0,-1) {(};
        \node at (6,-1) {)};
        \foreach \x\i in {0.5/1,1.5/3,2.5/4,3.5/1,4.5/2}
        {
        \node at (\x,-1) {\i,};
        }
        \node at (5.5,-1) {3\textcolor{white}{,}};
        \end{scope}
    \end{tikzpicture}
    \end{center}
    We see that for $k=4,3,2,1$, $\alpha^{(k)}$ may only differ from $\alpha^{(k-1)}$ on particles of type $k-1$ in $\alpha^{(k)}$, which may increase to type $k$ in $\alpha^{(k-1)}$.  
\end{example}

\begin{lemma}\label{lemma:ascents onerowcollapse}
    Let $\alpha = \type(\rho^{(k-1)}(M))$ and $\beta = \type(\rho^{(k)}(M))$. Then $\alpha_i < \alpha_j$ implies $\beta_i < \beta_j$ for $i<j$.
\end{lemma}

\begin{proof}
Since $k$ is arbitrary, it suffices to prove the claim for $k=L$. By \cref{lemma:onerowcollapse}, we have that $(\beta_i - \alpha_i),(\beta_j - \alpha_j) \in\{0,1\}$. Since $\alpha_i < \alpha_j$, we have $\beta_i < \beta_j$ unless $\alpha_i + 1 = \alpha_j=\beta_i = \beta_j=\ell$ for some $\ell\leq L-1$. Suppose this is the case.  Let us use the notation from the proof of \cref{lemma:onerowcollapse}, defining the intermediate configurations $M^{(L)},\ldots,M^{(0)}$ and the nonwrapping multiline queues $N^{(L)},\ldots,N^{(0)}$ they contain, having types $\alpha^{(L)},\alpha^{(L-1)},\ldots,\alpha^{(0)}$ with $\alpha^{(L)}=\alpha$ and $\alpha^{(0)}=\beta$.

By the proof of \cref{lemma:onerowcollapse},  $\alpha^{(k)}_i=\alpha^{(k-1)}_i$ for all $k\neq\ell$. Thus the change from $\alpha^{(k)}_i=\alpha_i=\ell-1$ to $\alpha^{(k-1)}=\ell=\beta_i$ must occur at the step $k=\ell$ when passing from $M^{(\ell)}$ to $M^{(\ell-1)}=\eDs_\ell M^{(\ell)}$. Thus we have
\[
i\in I_{\geq \ell-1}(N^{(\ell)})\setminus I_{\geq \ell}(N^{(\ell)}),\qquad j\in I_{\geq \ell}(N^{(\ell)})\setminus I_{\geq \ell+1}(N^{(\ell)}),\qquad i,j\in I_{\geq \ell}(N^{(\ell-1)}).
\]
Let $x$ and $y$ be the balls in $N^{(\ell)}$ at the tops of the strands anchored at $i$ and $j$, respectively: $x$ is in row $\ell-1$ and $y$ in row $\ell$; since $i\not\in I_{\geq \ell}(N^{(\ell)})$, we must have that $x$ is to the left of $y$ by \cref{lemma:intersectingstrands}. Let $x'$ be the ball in $N^{(\ell-1)}$ in row $\ell-1$ of the strand anchored at $i$. Note that $x'$ must be weakly left of $x$. If $i\in\theta(\pi^{(\ell)}(\eDs_\ell(M^{(\ell)})))\setminus \theta(\pi^{(\ell)}(M^{(\ell)}))$, there must be a ball $z$ in row $\ell+1$ of $M^{(\ell)}$ that is dropped by $\eDs_\ell$ to pair with $x'$: thus $z$ is weakly left of $x'$, and thus also of $y$. But as $y$ is unmatched in row $\ell$ of $\theta_\ell(M^{(\ell)})$, it is not possible for such $z$ to exist.
\end{proof}

The following lemma will assist us with the proof of \cref{thm:rho type}. See \cref{ex:srho type} for an illustration.
\begin{lemma}\label{lem:srho type}
    Let $j\in R(B_1,B_2,...,B_r,J)$ and $j\not\in\theta(B_1,B_2,...,B_r,J)$. Then for all $i<j$ with $i\in \theta(B_1,B_2,....,B_r)$, we have $i\in R(B_1,B_2,...,B_r,J)$. 
\end{lemma}
\begin{proof}
    We proceed by induction on $r$. For $r=1$, let $b\in J$ such that $b$ matches with $j$ in $R(B_1,J)$. Since $j\notin \theta(B_1,J)$, we have $b>j$. Now since $b$ matched with $j$ while skipping $i$, $i$ is already matched with something in $J$, so we have $i\in R(B_1,J)$.

    Now let $r>1$. Let $a\in \theta(B_2,B_3,...,B_r)$ be such that $a$ matches with $i$ in $\theta(B_1,B_2,...,B_r)$. Let $b\in R(B_2,B_3,...,B_r,J)$ such that $b$ matches with $j$ in $R(B_1,B_2,...,B_r,J)$. We have two cases:
    \begin{enumerate}[label=\textbf{Case \Roman*}.,ref=Step~\Roman*,
    align=left,        
        leftmargin=44pt,    
          labelindent=0pt,   
          labelsep=0.2em,   
        labelwidth=4.2em]
    \item Suppose $b<i$ or $j<b$. In either case, $b$ matches with $j$ while skipping $i$, so $i$ is matched with something in $R(B_2,B_3,...,B_r,J)$, so $i\in R(B_1,B_2,...,B_r,J)$.
    \item Suppose $i<b\leq j$. We will show the following:
    \begin{equation} \label{eq:contain}
        c\in [a,i] \cap \theta(B_2,B_3,...,B_r)\implies c\in R(B_2,B_3,...,B_r,J) 
    \end{equation}
    Now since $a$ matched with $i$ in $\theta(B_1,B_2,...,B_r)$, we conclude $a$ matches with $i$ in $R(B_1,B_2,...,B_r,J)$ and we have $i\in R(B_1,B_2,...,B_r,J)$. 
    
    Here, we have two possibilities, $b\in \theta(B_2,B_3,...,B_r,J)$ or $b\notin \theta(B_2,B_3,...,B_r,J)$. If $b\notin \theta(B_2,B_3,...,B_r,J)$, then using the induction hypothesis with $b$ in place of $j$ and $i<b$, we get \eqref{eq:contain}. If $b \in \theta(B_2,B_3,...,B_r,J)$, then the fact that $j\notin \theta(B_1,B_2,...,B_r,J)$ implies that there is a $b'\in R(B_1,B_2,...,B_r,J)\setminus\theta(B_2,B_3,...,B_r,J)$ such that $b<b'\leq j$. Now using the induction hypothesis with $b'$ in place of $j$ and $i<b<b'$, we get \eqref{eq:contain}.
    \end{enumerate}
\end{proof}

\begin{example}
    \label{ex:srho type}
    We show an example of \cref{lem:srho type} for the multiline queue $M=(B_1,B_2,B_3,J)$ shown below, with strings of matched balls giving $R(M)=\{3,4,5,7\}$, $\theta(M)=\{3,4,7\}$, and $\theta(B_1,B_2,B_3)=\{3,4,5,7\}$ shown below (matched balls are highlighted). The index $j=5$ satisfies $j\in R(M)\setminus\theta(M)$. Then, the indices $i\in\theta(B_1,B_2,B_3)$ with $i<j$ indeed satisfy the property $i\in R(M)$.
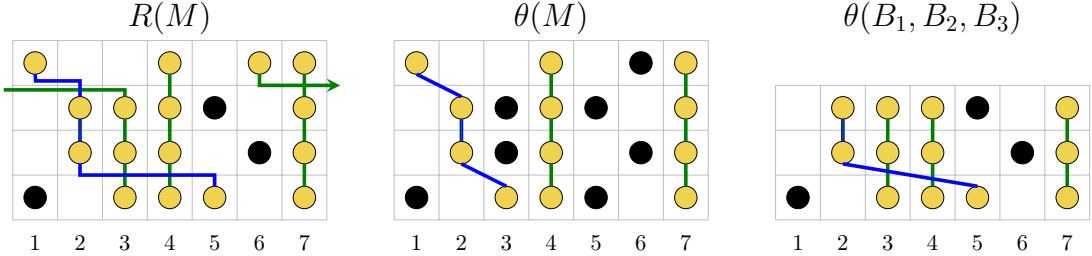
     \begin{center}
\resizebox{\linewidth}{!}{
\begin{tikzpicture}[scale=0.7]
\def \w{1};
\def \h{1};
\def \r{0.25};
    
\begin{scope}[xshift=0cm]
\node at (3.5*\w,4.5*\h) {\Large $R(M)$};
\draw[gray!50,thin,step=\w] (0,0) grid (7*\w,4*\h);
\foreach \i in {1,...,7}
{
\node at (\w*\i-\w*.5,-0.5) {\small \i};
}
\foreach \xx\yy\i\c in {0/3/4/\highlight,3/3/5/\highlight,5/3/5/\highlight,6/3/5/\highlight,1/2/5/\highlight,2/2/4/\highlight,3/2/5/\highlight,4/2/3/black,6/2/5/\highlight, 1/1/5/\highlight,2/1/4/\highlight,3/1/5/\highlight,5/1/3/black,6/1/5/\highlight, 0/0/3/black,2/0/5/\highlight, 3/0/5/\highlight, 4/0/4/\highlight,6/0/5/\highlight}
    {
    \draw[fill=\c](\w*.5+\w*\xx,\h*.5+\h*\yy) circle (\r cm);
    }

\draw[black!50!green,ultra thick](\w*6.5,\h*3.5-\r)--(\w*6.5,\h*2.5+\r) (\w*6.5,\h*2.5-\r)--(\w*6.5,\h*1.5+\r) (\w*6.5,\h*1.5-\r)--(\w*6.5,\h*.5+\r);
\draw[black!50!green,ultra thick](\w*3.5,\h*3.5-\r)--(\w*3.5,\h*2.5+\r) (\w*3.5,\h*2.5-\r)--(\w*3.5,\h*1.5+\r) (\w*3.5,\h*1.5-\r)--(\w*3.5,\h*.5+\r);

\draw[black!50!green](\w*3.5,\h*3.5-\r)--(\w*3.5,\h*2.5+\r);
\draw[black!50!green,-stealth, ultra thick](\w*5.5,\h*3.5-\r)--(\w*5.5,\h*3)--(\w*7.3,\h*3);
\draw[black!50!green, ultra thick](-0.2,\h*2.9)--(\w*2.5,\h*2.9)--(\w*2.5,\h*2.5+\r) (\w*2.5,\h*2.5-\r)--(\w*2.5,\h*1.5+\r) (\w*2.5,\h*1.5-\r)--(\w*2.5,\h*.5+\r);
\draw[blue,ultra thick](0.5*\w, \h*3.5-\r)--(0.5\w, \h*3.1)--(1.5*\w, \h*3.1)--(1.5*\w,\h*2.5+\r) (1.5*\w,\h*2.5-\r)--(1.5*\w,\h*1.5+\r) (1.5*\w,\h*1.5-\r)--(1.5*\w,\h*1)--(4.5*\w,\h*1)--(4.5*\w,\h*.5+\r);

\draw[black!50!green](\w*1.5,\h*2.5-\r)--(\w*1.5,\h*1.5+\r);
\draw[black!50!green](\w*3.5,\h*2.5-\r)--(\w*3.5,\h*1.5+\r);

\end{scope}

\begin{scope}[xshift=8.5cm]
\node at (3.5*\w,4.5*\h) {\Large $\theta(M)$};
\draw[gray!50,thin,step=\w] (0,0) grid (7*\w,4*\h);
\foreach \i in {1,...,7}
{
\node at (\w*\i-\w*.5,-0.5) {\small \i};
}
\foreach \xx\yy\i\c in {0/3/4/\highlight,3/3/5/\highlight,5/3/5/black,6/3/5/\highlight,1/2/5/\highlight,2/2/4/black,3/2/5/\highlight,4/2/3/black,6/2/5/\highlight, 1/1/5/\highlight,2/1/4/black,3/1/5/\highlight,5/1/3/black,6/1/5/\highlight, 0/0/3/black,2/0/5/\highlight, 3/0/5/\highlight, 4/0/4/black,6/0/5/\highlight}
    {
    \draw[fill=\c](\w*.5+\w*\xx,\h*.5+\h*\yy) circle (\r cm);
    }

\draw[black!50!green,ultra thick](\w*6.5,\h*3.5-\r)--(\w*6.5,\h*2.5+\r) (\w*6.5,\h*2.5-\r)--(\w*6.5,\h*1.5+\r) (\w*6.5,\h*1.5-\r)--(\w*6.5,\h*.5+\r);
\draw[black!50!green,ultra thick](\w*3.5,\h*3.5-\r)--(\w*3.5,\h*2.5+\r) (\w*3.5,\h*2.5-\r)--(\w*3.5,\h*1.5+\r) (\w*3.5,\h*1.5-\r)--(\w*3.5,\h*.5+\r);

\draw[black!50!green](\w*3.5,\h*3.5-\r)--(\w*3.5,\h*2.5+\r);
\draw[blue,ultra thick](0.5*\w, \h*3.5-\r)--(1.5*\w,\h*2.5+\r) (1.5*\w,\h*2.5-\r)--(1.5*\w,\h*1.5+\r) (1.5*\w,\h*1.5-\r)--(2.5*\w,\h*.5+\r);

\draw[black!50!green](\w*1.5,\h*2.5-\r)--(\w*1.5,\h*1.5+\r);
\draw[black!50!green](\w*3.5,\h*2.5-\r)--(\w*3.5,\h*1.5+\r);
\end{scope}

\begin{scope}[xshift=17cm]
\node at (3.5*\w,4.5*\h) {\Large $\theta(B_1,B_2,B_3)$};
\draw[gray!50,thin,step=\w] (0,0) grid (7*\w,3*\h);
\foreach \i in {1,...,7}
{
\node at (\w*\i-\w*.5,-0.5) {\small \i};
}
\foreach \xx\yy\i\c in {1/2/5/\highlight,2/2/4/\highlight,3/2/5/\highlight,4/2/3/black,6/2/5/\highlight, 1/1/5/\highlight,2/1/4/\highlight,3/1/5/\highlight,5/1/3/black,6/1/5/\highlight, 0/0/3/black,2/0/5/\highlight, 3/0/5/\highlight, 4/0/4/\highlight,6/0/5/\highlight}
    {
    \draw[fill=\c](\w*.5+\w*\xx,\h*.5+\h*\yy) circle (\r cm);
    }

\draw[black!50!green,ultra thick] (\w*6.5,\h*2.5-\r)--(\w*6.5,\h*1.5+\r) (\w*6.5,\h*1.5-\r)--(\w*6.5,\h*.5+\r);
\draw[black!50!green,ultra thick](\w*3.5,\h*2.5-\r)--(\w*3.5,\h*1.5+\r) (\w*3.5,\h*1.5-\r)--(\w*3.5,\h*.5+\r);
\draw[black!50!green,ultra thick](\w*2.5,\h*2.5-\r)--(\w*2.5,\h*1.5+\r) (\w*2.5,\h*1.5-\r)--(\w*2.5,\h*.5+\r);

(\w*2.5,\h*2.5-\r)--(\w*2.5,\h*1.5+\r) (\w*2.5,\h*1.5-\r)--(\w*2.5,\h*.5+\r);
\draw[blue,ultra thick] (1.5*\w,\h*2.5-\r)--(1.5*\w,\h*1.5+\r) (1.5*\w,\h*1.5-\r)--(4.5*\w,\h*.5+\r);

\draw[black!50!green](\w*1.5,\h*2.5-\r)--(\w*1.5,\h*1.5+\r);
\draw[black!50!green](\w*3.5,\h*2.5-\r)--(\w*3.5,\h*1.5+\r);
\end{scope}
\end{tikzpicture}
}
\end{center}
\end{example}

We proceed to the first main result of this section, which is illustrated in \cref{ex:coinv}.

\begin{theorem}\label{thm:rho type}
    Let $\alpha=\type(M)$ and $\beta=\type(\rho(M))$. For $i<j$,
    \[\alpha_i<\alpha_{j}\qquad\implies \qquad \beta_i<\beta_{j}.\]
\end{theorem}

\begin{proof}
    If $\alpha_j <L$, we induct on the number of rows in $M$. Let $\alpha' = \type(\pi^{(L-1)}M)$ and $\beta' = \type(\rho^{(L-1)}(M))$. Since $\alpha'$ is obtained from $\alpha$ by replacing every occurrence of $L$ with an $L-1$, $\alpha'_i < \alpha'_j$. By induction, $\beta'_i < \beta'_j$ and so $\beta_i < \beta_j$ by \cref{lemma:ascents onerowcollapse}.

    Thus assume $\alpha_j = L$. Let $\beta^{(k)} = \type(\rho^{(k)}(M))$ so that $\beta^{(L)} = \beta$. We want to show that there exists some $1\leq r \leq L$ such that $\beta_i^{(r)} < \beta_j^{(r)}$ since then, we may repeatedly apply \cref{lemma:ascents onerowcollapse} to conclude that $\beta_i < \beta_j$. If $\beta_j = L$, or equivalently $j \in I_L(\rho(M))$, then $\beta_i < \beta_j$ since $I_L(\rho(M)) \subseteq I_L(M)$. Thus, we may assume there exists an $1\leq r< L$ such that $j \in I_r(\rho^{(r)}(M))$ but $j \notin I_{r+1}(\rho^{(r+1)}(M))$. Assume for sake of contradiction $i \in I_r(\rho^{(r)}(M))$. We want to show that this implies $\alpha_i=L$. Let $M = (B_1,B_2,...,B_L)$. From \eqref{eq:R}, we have
    \[
        \alpha_i= L \iff i\in R(B_1,B_2,...,B_L)
    \]

    By \cref{prop:label L through collapsing},
        \begin{align*}
            i,j \in \theta(B_1,B_2,\ldots, B_r),\qquad j\not\in\theta(B_1,\ldots,B_r,B_{r+1}).
        \end{align*}
    Let $J= R(B_{r+1},B_{r+2},...,B_L)$. Since $J\subset B_{r+1}$, we get $\theta(B_1,B_2,...,B_r,J)\subset \theta(B_1,B_2,...,B_r,B_{r+1})$, hence $j\notin \theta(B_1,B_2,...,B_r,J)$. We also have
    \[
    \alpha_j = L\implies j\in R(B_1,B_2,...,B_L) = R(B_1,B_2,...,B_r,J)
    \] 
    Since $i<j$, we use \cref{lem:srho type} to conclude $i\in R(B_1,B_2,...,B_r,J) = R(B_1,B_2,...,B_L)$, hence, $\alpha_i=L$.
    \end{proof}

    Note that \cref{lem:srho type} can be obtained as the contrapositive of \cref{thm:rho type} with $\alpha_j=L$ and $\beta_i\geq L-1$.

    \begin{example}\label{ex:coinv}
    Consider the multiline queue $M$ from \cref{ex:collapsing} with $\type(M)=(3,0,5,5,4,0)$ and $\type(\rho(M))=(1,1,4,5,3,3)$. The coinversions in $\type(M)$ are $\{(1,3),(1,4),(1,5),(2,3),(2,4),(2,5)\}$ and the coinversions in $\type(\rho(M))$ are $\{(1,3),(1,4),(1,5),(1,6),(2,3),(2,4),(2,5),(2,6),(3,4)\}$, verifying \cref{thm:rho type}. Observe that there is no preservation of the inversion structure: $(4,5)$ is an inversion in both $\type(M)$ and $\type(\rho(M))$, whereas $(1,6)$ is an inversion in $\type(M)$ and a coinversion in $\type(\rho(M))$.
    \end{example}

One of the powerful aspects of $\fR_i$- and $\eL_i$-fullness is that the property is preserved under $\rho$.

\begin{lemma}\label{lem:M full}
For a multiline queue $M$, if $M$ is $g^\leftrightarrow_i$-full, then so is $\rho(M)$. 
\end{lemma}

\begin{proof}
We first prove the claim for $g^\leftrightarrow_i=\eL_i$. Assume $M$ is $\eL_i$-full, which means that $\type(M)_i < \type(M)_{i+1}$ and that there is exactly one unmatched ball in column $i+1$ in $\theta_i(\rw(M))$. From \cref{thm:rho type}, we have $\type(\rho(M))_i<\type(\rho(M))_{i+1}$. From \cref{prop:num of balls preserved}, there is exactly one unmatched ball in column $i+1$ in $\theta_i(\rw(\rho(M)))$ as well. Thus, $\rho(M)$ is $\eL_i$-full.

For $g^\leftrightarrow_i=\fR_i$, assume $M$ is $\fR_i$-full, or in other words, $\fR_i(M)$ is $\eL_i$-full. Then $\rho(\fR_i(M))$ is $\eL_i$-full and thus $\fR_i(\rho(M))=\rho(\fR_i(M))$ is $\eL_i$-full as well. By definition, this means that $\rho(M)$ is $\fR_i$-full.
\end{proof}

Finally, as a corollary of \cref{lem:M full} we obtain the second main result of the section.

\begin{cor}\label{cor:M change type only if}
    For a multiline queue $M$ and a column operator $g^\leftrightarrow_i\in\{\eL_i,\fR_i\}$, 
    \[\type(g^\leftrightarrow_i(M))=s_i\cdot\type(M)\qquad \text{only if}\qquad \type(\rho(g^\leftrightarrow_i(M)))=s_i\cdot\type(\rho(M))\,.\]
\end{cor}

Thus our analysis of the interaction between $\rho$ and $\type$ in \cref{sec:rho} reduces to the following three properties. 
    \begin{lemma}
        \label{lem:type si}
        Let $M\in\MLQ_\lambda$ with $L=\max \lambda$. 
        \begin{itemize}
        \item[i.] Suppose $i$ is such that $\type(M)_i<\type(M)_{i+1}$. Then, for $s\in\{\epsilon,s_i\}$, we have 
        \begin{equation}
        \type(\rho(\eL_i\,M))=s\cdot\type(\rho(M))\qquad\iff\qquad  \type(\eL_i\,M)=s\cdot\type(M).
        \end{equation}
        \item[ii.]
       \begin{equation}
\type(\rho(\fR_i\,M))=\type(\rho(M))\qquad\implies\qquad
            \type(\fR_i\,M)=
                \type(M).
        \end{equation}
        \item[iii.]
     \begin{equation}
            \type(\rho(\fR_i\,M))=s_i\cdot\type(\rho(M))\qquad\implies\qquad \type(M)_{i}\geq \type(M)_{i+1}.
        \end{equation}
       \end{itemize}
    \end{lemma}

    \begin{proof}
    (i.) follows from \cref{thm:rho type} along with \cref{prop:num of balls preserved} by \cref{def:ei full}.
    
    (ii.) follows from the contrapositive of \cref{cor:M change type only if}.

    (iii.) is the contrapositive of \cref{thm:rho type}.
    \end{proof}

\section{Nonsymmetric and quasisymmetric refinements of $\text{MLQ}_\lambda$}\label{sec:induction}
In this section, we prove our main results, \cref{theorem:weakmain,theorem:main}.

    The key to proving \cref{theorem:weakmain} is \cref{theorem:weaktype}, which we restate here for convenience. The proof of this result will be the focus of this section.
    \weaktype*

    Our proof will be by induction on number of rows $L$ of the multiline queue. The base case $L=1$ is trivial. Suppose that the claim holds for any multiline queue on fewer than $L$ rows. Our proof will consist of the following steps. Let $\lambda$ be a partition with $\lambda_1=L$, and let $M,M'\in\MLQ_\lambda$ be two multiline queues in the same connected component. 
    \begin{enumerate}[label=\textbf{Step \Roman*}.,ref=Step~\Roman*,
    align=left,        
        leftmargin=44pt,    
          labelindent=0pt,   
          labelsep=0.2em,   
        labelwidth=4.2em]
  \item \label{step:I} In \cref{lem:step1} we show that if $\rho(M)$ and $\rho(M')$ have the same type, then the collapses of their restrictions to the bottom $L-1$ rows also have the same type:
            \[
                \type(\rho(M)) = \type(\rho(M')) \quad\implies\quad \type(\rho^{(L-1)}(M)) = \type(\rho^{(L-1)}(M')).
            \]
        \item \label{step:II} From \ref{step:I} and the induction hypothesis applied to the restrictions of $M$ and $M'$ to their bottom $L-1$ rows, we conclude that the pair of truncated multiline queues must have the same type:
            \begin{equation*}
               \type(\rho(M)) = \type(\rho(M'))\implies \type(\pi^{(L-1)}(M)) = \type(\pi^{(L-1)}(M')).
            \end{equation*}
            This implies $I_k(M)=I_k(M')$ for all $k<L-1$, and 
            \[I_L(M)\cup I_{L-1}(M)=I_L(M')\cup I_{L-1}(M')\,.\]
        \item \label{step:III} Lastly, \cref{lem:step 3} shows that $\type(\rho(M)) = \type(\rho(M'))$ implies that $I_L(M)=I_L(M')$. Together with \ref{step:II}, this proves \cref{theorem:weaktype}.
    \end{enumerate}
   
    The following lemma corresponds to \ref{step:I} in the proof sketch.

    \begin{lemma}\label{lem:step1}
    Let $\lambda$ be a partition with $L\coloneq \lambda_1$, and let $M,M'\in\MLQ_\lambda$ be two multiline queues in the same connected component such that $\type(\rho(M))=\type(\rho(M'))$. Then
    \[\type(\rho^{(L-1)}(M)) = \type(\rho^{(L-1)}(M'))\,.\]
    \end{lemma}

    \newcommand{\abf}{\mathbf{a}}
    \newcommand{\bbf}{\mathbf{b}}
    \newcommand{\cbf}{\mathbf{c}}

    \begin{proof}
         Let $k$ be the number of columns in $M$, and set $\abf = a_1a_2 \ldots a_k = \type(\rho^{(L-1)}(M))$ and $\bbf = b_1b_2 \ldots b_k=\type(\rho^{(L-1)}(M'))$. Since $M$ and $M'$ are in the same connected component, by \cref{cor:sequence},  there exists a sequence of column operators $\sbf$ such that $M'=\sbf M$. By \cref{cor:Harper's lemma}, a column operator may change the type of a multiline queue at most by a simple transposition. Hence $\abf$ is a permutation of $\bbf$, and so $\sort(\abf)=\sort(\bbf)$. We show that if $\type(\rho(M)) = \type(\rho(M'))$ then $\abf=\bbf$.
        
        Let $\type(\rho(M)) = \type(\rho(M')) = c_1c_2 \ldots c_k$. From \cref{lemma:onerowcollapse} we get that $0\leq c_i-a_i\leq 1$ and $0\leq c_i-b_i\leq 1$, which also implies that $|a_i-b_i|\leq 1$ for each $i$. Now assume $\abf\neq\bbf$ and let $\eta$ be the smallest integer such that there exists some index $i$ such that $a_i \neq b_i$ with $\{a_i,b_i\} = \{\eta,\eta+1\}$. Fix $j$ to be the smallest such index, and without loss of generality assume that $a_j=\eta$ and $b_j=\eta+1$. 
        
        The minimality of $\eta$ implies that $j$ is the first index for which $\bbf$ has fewer $\eta$'s than $\abf$ within their first $j$ entries. 
        Now, since $\bbf$ is a permutation of $\abf$, it follows there must exist some index $\ell>j$ such that $b_\ell = \eta$ and $a_\ell = \eta+1$. 
        
      We examine two cases: $\eta=0$ and $\eta>0$:
        \begin{enumerate}[align=left,        
                  leftmargin=34pt,    
                  labelindent=0pt,   
                  labelsep=0.25em,   
                  labelwidth=3em]
            \item[$\eta=0$:] Let $x$ be the ball in the first row of column $\ell$. Let $t$ be the ball that is collapsed from the highest level and ended in the first row of column $j$. We already have a contradiction: indeed, since $j<\ell$, $t$ could have paired with $x$ during the FM algorithm, rather than collapsing.
            
            \item[$\eta>0$:] Let $s_1$ and $s_2$ be the strands anchored at columns $j$ and $\ell$ respectively in $\rho^{(L-1)}(M)$, and let $v_1$ and $v_2$ be their topmost balls, lying in rows $\eta$ and $\eta+1$, respectively. Let $u_1$ and $u_2$ be the bottommost balls of strands $s_1$ and $s_2$, lying in sites $j$ and $\ell$, respectively. Since $j<\ell$, and the length of $s_1$ is less than $s_2$, \cref{lemma:intersectingstrands} (since $\rho^{(L-1)}(M)$ is non-wrapping) implies that $s_1$ and $s_2$ don't intersect. Thus since $u_1$ is to the left of $u_2$, the ball $v_1$ is to the left of the ball in row $\eta$ in $s_2$. Finally, the fact that the FM algorithm didn't pair $v_2$ with $v_1$ implies that $v_2$ is \emph{strictly} to the right of $v_1$. 
        
            Now consider the collapsing of the balls at row $L$ to obtain $\rho(M)$ and $\rho(M')$ from $\rho^{(L-1)}(M)$ and $\rho^{(L-1)}(M')$. As $c_j-1\leq a_j,b_j\leq c_j$, necessarily $c_j=\eta+1$. Thus there must exist some ball  $t$ in a row \emph{above} $\eta+1$ in $M$, that collapses to row $\eta+1$ in $\rho(M)$ and is paired with $v_1$ by the FM algorithm. Since $\rho(M)$ is non-wrapping, $t$ is weakly to the left of $v_1$. 
            
            Now, since $a_\ell = b_\ell+1$, and $c_\ell\leq b_\ell+1=\eta+1$, we have $c_\ell=\eta+1$, so $v_2$ is the topmost ball in its strand in $\rho(M)$ (as well as in $\rho^{(L-1)}(M)$). But since $v_2$ is in row $\eta+1$ in $\rho^{(L-1)}(M)$, this produces a contradiction: indeed, $v_2$ lies strictly to the right of $v_1$ and thus also of $t$, and so $t$ could not have collapsed to row $\eta+1$, as it could have matched with $v_2$ instead. 
        \end{enumerate}
    \end{proof}
    
    Next, we focus on \ref{step:III} in the proof sketch.

    For a permutation $\sigma\in\mathfrak{S}_n$ and a set $I\subseteq [n]$, denote the induced action of $\sigma$ on $I$ by
\[ \sigma(I)=\{\sigma(j)\ \colon\ j\in I\}.
\]
For example, $s_1(\{2,3\})=\{1,3\}$ and $s_2(\{2,3\}=\{2,3\}$.

    \begin{prop}
        \label{lem:step 3}
        Let $\lambda$ be a partition with $L\coloneq \lambda_1$, and let $M,M'\in\MLQ_\lambda$ be two multiline queues in the same connected component such that $\type(\rho(M))=\type(\rho(M'))$. Then
        \[I_L(M)=I_L(M').\]
    \end{prop}
    We prove \cref{lem:step 3} by showing the above for the special case when $\rho(M') = M_\alpha$ where $\alpha= \type(M)$. We do this by producing a sequence from $\rho(M)$ to the lowest weight element in $\NMLQ_{\sort(\alpha)}$ using only $\fR_i$s, and from there back to $M_\alpha$ using only $\eLs_i$s in a particular order. This allows us to track the changes in type explicitly for the preimages under $\rho$ using \cref{lem:type si}.

    \begin{proof}[Proof of \cref{lem:step 3}]
    Let $N=\rho(M)$, $N'=\rho(M')$, and let $\alpha=\type(N)=\type(N')$. We wish to show that 
    \begin{equation}\label{eq:IL}
    I_L(M)=I_L(M').
    \end{equation}
Set $n=\sum_i\alpha_i$ to be the total number of balls in $M$, and consider the crystal $\NMLQ_{\sort(\alpha)}$ on $n$ columns. We will construct a sequence $\mathbf{g}$ of column operators that provides a path from $N$ to the straight multiline queue $M_\alpha$ within $\NMLQ_{\sort(\alpha)}$ by passing through the lowest weight element of the crystal  (the unique multiline queue such that every $\fR_i$ for $1\leq i<n$ acts trivially). By commutativity of $\mathbf{g}$ with $\rho$, we then have $M'=\mathbf{g}\,M$. We will choose $\mathbf{g}$ in order to explicitly track its action on the positions of the largest particles $L$ in the types of the multiline queues in $\MLQ_\lambda$ obtained by sequentially applying the operators in $\mathbf{g}$.

Let $\sbf=\fR_{i_m}\cdots \fR_{i_1}$ be a sequence of column lowering operators with $1\leq i_k<n$ such that $\sbf\,N$ is the straight multiline queue $M_\delta$, where $\delta_1\leq \cdots\leq \delta_n$ and $\sort(\delta)=\sort(\alpha)$. Such a sequence exists since $\MLQ_{\sort(\alpha)}$ is connected and finite, and $M_\delta$ is its lowest-weight element. 

Whenever an application of $\fR_i$ changes the type of the multiline queue by removing a descent at site $i$ (equivalently, when the associated simple transposition $s_i$ acts on the type by removing a descent), we decorate this action with a right arrow and write $s^\rightarrow_i$. Similarly, if an application of $\eL_i$ changes the type by removing an ascent, we decorate the corresponding $s_i$ with a left arrow, writing $s^\leftarrow_i$.

By \cref{cor:Harper's lemma}, there exists a subsequence $\fR_{i_{u_k}},\ldots,\fR_{i_{u_1}}$ of $\sbf$, for some set of indices $u_1<\ldots<u_k$, such that each successive operator $\fR_{i_{u_j}}$ changes the type of the multiline queue in $\NMLQ_{\sort(\alpha)}$ by removing a descent (with all other operators acting trivially on the type of the multiline queue). Let the corresponding decorated sequence of simple transpositions be $s^\rightarrow_{a_k}\cdots s^\rightarrow_{a_1}$, with $a_j=i_{u_j}$. Define
\[
\alpha^{(j)} := 
  s^\rightarrow_{a_j}\cdots s^\rightarrow_{a_1}\cdot\alpha,
  \qquad 0\le j\le k,
\]
so that the types evolve sequentially as
\[
  \alpha=\alpha^{(0)}
  \xrightarrow{s^\rightarrow_{a_1}}
  \alpha^{(1)}
  \xrightarrow{s^\rightarrow_{a_2}}
  \cdots
  \xrightarrow{s^\rightarrow_{a_k}}
  \alpha^{(k)}=\delta=\type(M_\delta),
\]
and hence
$\delta=s^\rightarrow_{a_k}\cdots s^\rightarrow_{a_1}\cdot\alpha$. Equivalently, the sequence of transpositions $s^\rightarrow_{a_k}\cdots s^\rightarrow_{a_1}$ sorts $\alpha$ into the weakly increasing composition $\delta$. 

We now examine how the sequence of operators $\sbf$ acts on the corresponding multiline queue in $\MLQ_\lambda$ 
by tracking the sequence of types of multiline queues obtained by successively applying the operators in $\sbf$ to $M$.

When the sequence $\sbf$ is applied to $M$, there is some subsequence of indices $v_1<\cdots<v_\ell$ corresponding to those operators $\fR_{i_{v_j}}$ that change the type of the multiline queue in $\MLQ_{\lambda}$.  By \cref{lem:type si} (ii), this subsequence is contained in the subsequence that acts nontrivially on the type in $\NMLQ_{\sort(\alpha)}$. We track how the associated sequence of transpositions $s^\rightarrow_{b_\ell}\cdots s^\rightarrow_{b_1}$ with $b_j=i_{v_j}$ affects the positions of the largest label $L$ in the multiline queue in $\MLQ_\lambda$ by tracking $I_L$ in the sequence of multiline queues $\{\fR_{i_{u_j}}\fR_{i_{u_j-1}}\cdots \fR_{i_1}\,M\}_{1\leq j\leq k}$. Note that if $s\not\in\{u_1,\ldots,u_k\}$, then $\type(\fR_{i_{s}}\fR_{i_{s-1}}\cdots \fR_{i_1}\,M)=\type(\fR_{i_{s-1}}\fR_{i_{s-2}}\cdots\fR_{i_1}\,M)$; thus keeping track of this sequence is sufficient for our purposes. 

Set $\mathbf{s}^{(0)}=\id$ and $I_L^{(0)}=I_L(M)$, and for $1\leq j\leq k$, let \[\mathbf{s}^{(j)}=s^\rightarrow_{u_j}\cdots s^\rightarrow_{u_1}.\]
We define $\sbf_1= \fR_{i_{u_1}}\cdots \fR_{i_1}$ and for $2\leq j\leq k$ let
\[ \sbf_j = \fR_{i_{u_j}}\fR_{i_{u_j-1}}\cdots\fR_{i_{u_{j-1}+1}}.\]
Let $\sbf^{(j)}$ be the subsequence of operators $\{\fR_i\}$ in $\sbf$ ending with the operator that causes the change in type given by $s^\rightarrow_{a_j}$ in $\NMLQ_{\sort(\alpha)}$: 
\[\sbf^{(j)}\coloneq \sbf_j\sbf_{j-1}\cdots\sbf_1, \qquad I_L^{(j)}\coloneq I_L(\sbf^{(j)}\,M).\]

We have the following property: if $a_j=i$, meaning that
\[\mathbf{s}^{(j)}=s^\rightarrow_i\mathbf{s}^{(j-1)},\]
then the following hold:
\begin{enumerate}[label=(\Alph*), ref=\Alph*]
\item\label{i:A} By  \cref{lem:type si} (iii), if $i+1\in I_L^{(j-1)}$ then necessarily $i\in I_L^{(j-1)}$.
\item\label{i:B} We have either $I_L^{(j)}=I_L^{(j-1)}$ or $I_L^{(j)}=s_i\cdot I_L^{(j-1)}$. In the latter case, necessarily $i+1\in I_L^{(j)}$ and $i\not\in I_L^{(j)}$. 
\item\label{i:C} By Property \ref{i:A}, conversely, if $i+1\in I_L^{(j)}$ and $i\not\in I_L^{(j)}$, we must have $I_L^{(j)}=s^\rightarrow_i\cdot I_L^{(j-1)}$.
\end{enumerate}

We now define the sequence of operators 
\[\mathbf{t}=\eLs_{u_1}\cdots \eLs_{u_k}
\]
by reversing the sequence of transpositions $\mathbf{s}$. We claim that $\mathbf{t}\,M_\delta=M_\alpha$. Indeed, for any composition $\beta$, if $\beta_i<\beta_{i+1}$, then 
\begin{equation}\label{eq:si}
\eLs_i\,M_\beta=M_{s^\leftarrow_i\beta}.
\end{equation}
Since each $s^\rightarrow_i$ in the sequence $\mathbf{s}$ creates an ascent, applying the corresponding $\eLs_i$'s in the reverse order successively removes
these ascents, as \eqref{eq:si} holds at each application. Consequently, 
\[\mathbf{t}\sbf\,N=\mathbf{t}\,M_\delta=M_\alpha,\]
as desired.

Now we will show that the action of the operators in $\mathbf{t}$ on the positions of the $L$'s in the multiline queues in $\MLQ_\lambda$ ``reverses'' the action of the operators in $\sbf$ on those positions.  Define 
\[
\mathbf{t}^{\rev(k)}\coloneq \id,\qquad \mathbf{t}^{\rev(j)}\coloneq \eLs_{a_{j+1}}\cdots \eLs_{a_k}\quad\text{for $0\leq j<k$.}
\]
and 
\[I_L^{\rev(j)}\coloneq I_L(\mathbf{t}^{\rev(j)}\sbf\,M).
\]

We proceed by induction on $j$ with the goal of proving that 
\begin{equation}\label{eq:M alpha} I_L(\mathbf{t}\sbf\,M)=I_L(\mathbf{t}^{\rev(0)}\sbf\,M)=I_L^{\rev(0)}=I_L^{(0)}=I_L(M)\,,
\end{equation}
 yielding the desired equality \eqref{eq:IL}. The base case $j=k$ is the trivial identity $I_L^{\rev(k)}= I_L(\sbf\,M)=I_L^{(k)}$. We will show that for $j\leq k$,
\[I_L^{\rev(j)}=I_L^{(j)}\qquad \implies \qquad I_L^{\rev(j-1)}=I_L^{(j-1)},
\]
from which it follows inductively that $I_L^{\rev(0)}=I_L^{(0)}$, as desired.

We now prove the induction hypothesis. Suppose $I_L^{\rev(j)}=I_L^{(j)}$, and let $i=a_j$, so that $\mathbf{s}^{(j)}=s^\rightarrow_i \mathbf{s}^{(j-1)}$ and $\mathbf{t}^{\rev(j-1)}=\eLs_i\mathbf{t}^{\rev(j)}$. 
Then, either $I_L^{(j)}=I_L^{(j-1)}$, or $I_L^{(j)}=s_i\cdot I_L^{(j-1)}$. We examine both cases: 

\smallskip
\noindent\textbf{Case 1:} $I_L^{(j)}=I_L^{(j-1)}$.
First suppose that
\[
\{i,i+1\}\subseteq I_L^{(j)}
\qquad\text{or}\qquad
\{i,i+1\}\cap I_L^{(j)}=\emptyset.
\]
By induction, the same holds for $I_L^{\rev(j)}$, so we also have
$I_L^{\rev(j-1)}=I_L^{(j)}$ since $\eLs_i$ only acts on sites $i,i+1$.

Now suppose that
\[
i\in I_L^{(j)} \qquad\text{and}\qquad i+1\not\in I_L^{(j)}.
\]
The same holds for $I_L^{\rev(j)}$. Thus site $i$ in
$\type(\mathbf{t}^{\rev(j)}\sbf\,M)$ does not have an ascent, so
$s^\leftarrow_i$ cannot act on the type; hence
$I_L^{\rev(j-1)}=I_L^{\rev(j)}$.

Property~\ref{i:C} implies that
$i+1\in I_L^{(j-1)}=I_L^{(j)}$ and
$i\not\in I_L^{(j-1)}=I_L^{(j)}$ cannot occur, so all
possibilities are covered. Thus
\[
I_L^{\rev(j-1)}=I_L^{\rev(j)}=I_L^{(j)}=I_L^{(j-1)}.
\]

\smallskip
\noindent\textbf{Case 2:} $I_L^{(j)}=s_i\cdot I_L^{(j-1)}$.
By Property~\ref{i:B}, we have
\[
i+1\in I_L^{(j)} \qquad\text{and}\qquad i\not\in I_L^{(j)},
\]
and the same holds for $I_L^{\rev(j)}$.
When $\eLs_i$ is applied to $\mathbf{t}^{(j)}M_\delta$, the type changes
by $s^\leftarrow_i$ at the last application of $\eL_i$, so
\cref{lem:type si} (i) applies. Hence
\[
\type(\mathbf{t}^{(j-1)}\sbf\,M)
= \type(\eLs_i\mathbf{t}^{(j)}\sbf\,M)
= s^\leftarrow_i\cdot \type(\mathbf{t}^{(j)}\sbf\,M),
\]
and therefore
\[
I_L^{\rev(j-1)}=s_i\cdot I_L^{\rev(j)}=I_L^{(j-1)}.
\]

In both cases we have that $I_L^{\rev(j)}=I_L^{(j)}$ implies $I_L^{\rev(j-1)}=I_L^{(j-1)}$. Applying the property iteratively proves \eqref{eq:M alpha}. Since the same is true for $M'$, the statement is proved. 
\end{proof}

We complete the proof of \cref{theorem:weaktype}.
    \begin{proof}[Proof of Theorem \ref{theorem:weaktype}]
    Let $M$ and $M'$ satisfy
    \begin{equation}
        \label{eq:ass}
        \type(\rho(M))=\type(\rho(M')).
    \end{equation}
        We proceed by induction on $L$. The case where $L=1$ is trivial, so let $L>1$ and assume the statement holds for a multiline queue of height $L-1$. 
        Then from \cref{lem:step1} and the induction hypothesis, we have that 
        \begin{equation}\label{eqn:induct}
            \type(\pi^{(L-1)}(M)) = \type(\pi^{(L-1)}(M')).
        \end{equation}
        
        Now, for any multiline queue $N$, $I_L(N)\cup I_{L-1}(N)=I_{L-1}(\pi^{(L-1)}(N))$, and for all $k<L-1$ we have $I_k(N)=I_k(\pi^{(L-1)}(N))$ (this follows from \eqref{eq:R} and \eqref{eq:R2}; for a detailed explanation, see \cite[Section 4.1]{MS25}). Thus \eqref{eqn:induct} implies $I_k(M)=I_k(M')$ for all $k<L-1$ and $I_{L-1}(M)\cup I_L(M)=I_{L-1}(M')\cup I_L(M')$. 
       By \cref{lem:step 3} we have that $I_L(M)=I_L(M')$. Combined, this is equivalent to the desired $\type(M)=\type(M')$.
    \end{proof}

It would be interesting to study whether the crystal graph can be enhanced with additional edges to recover connectivity in the nonsymmetric and quasisymmetric components.
    \begin{question}
    Can one add additional edges to the crystal graph so that the nonsymmetric or quasisymmetric components become connected, and, if so, is there a local characterization of the resulting connected components?
    \end{question}
    \subsection{Positive expansions of $f_\alpha(X;q,0)$ and $G_\gamma(X;q,0)$}\label{sec:proof of main}

In this section, we state and prove our main results to obtain nonsymmetric and quasisymmetric analogues of the Kostka--Foulkes coefficients for the ASEP polynomials and quasisymmetric Macdonald polynomials at $t=0$.

\begin{theorem}\label{theorem:weakmain}
For a composition $\alpha$, the ASEP polynomial is given by
  \begin{equation}\label{eq:nsym}
f_{\alpha}(X;q,0) = \sum_{\beta:\sort(\beta)\leq \sort(\alpha)} K_{\alpha,\beta}(q) \cA_{\beta},
\end{equation}
where
\begin{equation}\label{eq:nsymK}
K_{\alpha,\beta}(q) = \sum_{\substack{Q\in\SSYT(\sort(\beta)',\sort(\alpha)')\\\type(\rho^{-1}(M_{\beta},Q))=\alpha}} q^{\charge(Q)}.
\end{equation}
\end{theorem}

\begin{proof}
By \cref{theorem:weaktype}, for each $Q\in\SSYT(\sort(\beta)',\lambda')$ the preimage of $\NMLQ[\beta]$ under $\rho$ in the fiber of $Q$ is contained entirely in $\MLQ[\alpha]$ for some $\alpha$ that sorts to $\lambda$, and the weight of such a contribution is $q^{\charge(Q)}\cA_\beta(X)$. Thus to compute the coefficient of $\cA_\beta(X)$ in $f_\alpha(X;q,t)$, it suffices to identify which recording tableaux $Q$ satisfy $\rho^{-1}(M,Q)\in\MLQ[\alpha]$ for some $M\in\NMLQ[\beta]$, and this needs to be checked for only one canonical representative $M\in\NMLQ[\beta]$; 
we will use $M_\beta$ as the canonical representative. 
Define
        \begin{align*}
            S_{\alpha, \beta}&:=\{Q\in \SSYT(\sort(\beta)', \sort(\alpha)') \ :\,\strtype(\rho^{-1}(M_\beta,Q)) = \gamma\}\\
            S_\alpha&:=\{\beta\ :\ S_{\alpha,\beta}\not=\emptyset\}
        \end{align*}
        Now we can restrict the collapsing map further to $\MLQ[\alpha]$, and again using \cref{theorem:weaktype}, we obtain the following bijection:
    \begin{equation*}
        \MLQ[\alpha]\simeq \bigcup_{\beta\in S_\alpha}\NMLQ[\beta]\times S_{\alpha,\beta}
    \end{equation*}
    Now noting \[K_{\alpha,\beta}(q) = \sum_{Q\in S_{\alpha,\beta}}q^{\charge(Q)}\] we get \eqref{eq:nsymK}, as desired.
\end{proof}

Next we give the quasisymmetric analogue.

\begin{theorem}\label{theorem:main}
For a strong composition $\gamma$, the quasisymmetric Macdonald polynomial is given by
  \begin{equation}\label{eq:quasi}
G_{\gamma}(X;q,0) = \sum_{\tau:\sort(\tau)\leq \sort(\gamma)} K_{\gamma,\tau}(q) \QS_{\tau},
\end{equation}
where
\begin{equation}\label{eq:K}
K_{\gamma,\tau}(q) = \sum_{\substack{Q\in\SSYT(\sort(\tau)',\sort(\gamma)')\\\strtype(\rho^{-1}(M_{\tau},Q))=\gamma}} q^{\charge(Q)}\,.
\end{equation} 
\end{theorem}

To prove \cref{theorem:main}, we first need \cref{prop:sameQ}. While it is possible to give a similar argument as in the proof of \cref{theorem:weaktype}, we will instead invoke the result directly. Specifically, we will show that if two nonwrapping multiline queues have the same strong type, then their preimages under $\rho$ (for a fixed recording tableau $Q$) have the same strong type as well. In fact, by \cref{theorem:weaktype}, the argument reduces to showing that the multiline queues in the set $\{\rho^{-1}(M_\alpha,Q):\alpha^+=\gamma\}$ all have the same strong type. 

  \begin{proof}[Proof of \cref{prop:sameQ}]
  Let $Q=\rho_Q(M)\in\SSYT(\sort(\gamma)',\mu')$. Let $N=\rho(M)$ and $N'=\rho(M')$ such that $\type(N) = \alpha$ and $\type(N') = \beta$ with $\gamma\coloneq \alpha^+ = \beta^+$.

    Define $r_i$ to be the \textbf{crystal reflection operator} on the columns of $\MLQ_{\sort(\gamma)}$. When $K$ is a multiline queue with either column $i$ or $i+1$ empty, then $r_i$ simply swaps the two columns, coinciding in this case with the action of $\eLs_i$ or $f^{\rightarrow\star}_i$, respectively\footnote{We omit the general definition of $r_i$, as it is only needed for this special case.}. Since the FM algorithm disregards empty columns, we have $\strtype(K)=\strtype(r_i(K))$. Moreover,  whenever column $i$ in $K$ is empty, it is also empty in $\rho^{-1}(K,Q)$; thus $\rho^{-1}(r_i(K),Q)=r_i(\rho^{-1}(K,Q))$, and so 
    \begin{equation}\label{eq:strtype}
    \strtype(\rho^{-1}(K,Q))=\strtype(\rho^{-1}(r_i(K),Q)).
    \end{equation}

    For a reduced word $\tau=s_{i_1}\cdots s_{i_k}$, define $r_\tau\coloneqq r_{i_1}\cdots r_{i_k}$ to be the composition of operators (applied successively from right to left). Then, let $\sigma = s_{i_1} \cdots s_{i_k}$ be a permutation of shortest length such that $\sigma \cdot \alpha=\beta$. Since $\alpha$ and $\beta$ are both shuffles of $\gamma$ with a word of zeros, each simple transposition in $\sigma$ swaps a non-zero entry with a zero. Consider the straight multiline queues $M_\alpha$ and $M_\beta$. For $M_\alpha$, column $i$ is empty exactly when $\alpha_i=0$. Thus every operator $r_i$ corresponding to a transposition in $\sigma$ satisfies \eqref{eq:strtype}, and thus preserves the strong type of the preimage under $\rho^{-1}$. Therefore, using $r_\sigma(M_\alpha)=M_\beta$ along with \cref{theorem:weaktype} giving the second and fourth equalities, we conclude that
    \begin{multline*}
    \strtype(M)=\strtype(\rho^{-1}(N,Q))=\strtype(\rho^{-1}(M_\alpha,Q))=\strtype(\rho^{-1}(M_\beta,Q))\\
    =\strtype(\rho^{-1}(N',Q))=\strtype(M'),
    \end{multline*}
    as desired.
  \end{proof}

Finally, we obtain \cref{theorem:main} following the same argument as the proof of \cref{theorem:weakmain} after replacing $\type$ by $\strtype$ and using \cref{prop:sameQ}.

\subsection{$E_\alpha^\sigma(X;q,0)$ is Demazure atom-positive}\label{sec:permuted}

As a corollary of our result that the ASEP polynomials are atom-positive at $t=0$, we deduce that any permuted basement Macdonald polynomial is also atom-positive at $t=0$. To see this, we use the straightening rule for $E_\alpha^\sigma$ derived in \cite{DM25+}, which expresses $E_\alpha^\sigma$ in the $\{f_\beta\}$ basis. We restate the result below, omitting technical details.

\begin{theorem}[{\cite{DM25+}}]
Let $\alpha$ be a length-$n$ composition that sorts to a partition $\lambda$.
    \[E_\alpha^\sigma(X;q,t)=\sum_{\beta}c_{\alpha\beta}^\sigma(q,t) f_\beta(X;q,t)\]
where the sum is over length-$n$ compositions $\beta$ that sort to $\lambda$, and $c_{\alpha\beta}^\sigma(q,t)\in\QQ(q,t)$ has the form
\[c_{\alpha\beta}^\sigma(q,t)=\sum_{(a,b,D)\in\mathcal{I}_{\alpha\beta}^\sigma}q^at^b\prod_{(i,j)\in D}\frac{1-t}{1-q^jt^j},
\]
where $\mathcal{I}_{\alpha\beta}^\sigma$ is a finite multiset of triples $(a,b,D)$ with $a,b\geq 0$ and $D$ is a finite multiset of pairs in $\ZZ^2_{\geq 1}$.
\end{theorem}

In particular, at $t=0$, the given expression for the coefficients is sufficient to determine that $c_{\alpha\beta}^\sigma(q,0)\in\NN[q]$, regardless of the choice of the indexing set $\mathcal{I}_{\alpha\beta}^\sigma$. Thus the theorem above implies $E_\alpha^\sigma$ has a positive expansion in the ASEP polynomials:
\[
E_\alpha^\sigma(X;q,0)=\sum_{\beta}c_{\alpha\beta}^\sigma(q) f_\beta(X;q,0),\qquad c_{\alpha\beta}^\sigma(q)\in\NN[q],
\]
where the sum is over length-$n$ compositions $\beta$ that sort to $\lambda$. Since $f_\beta(X;q,0)$ is atom positive, this immediately implies the following corollary.

\begin{cor}
Let $\alpha$ be a length-$n$ composition that sorts to $\lambda$. Then
    \[E_\alpha^\sigma(X;q,0)=\sum_{\beta} d_{\alpha\beta}^\sigma(q) \cA_\beta(X),\qquad d_{\alpha\beta}^\sigma(q)\in\NN[q],\]
    where the sum is over length-$n$ compositions $\beta$ that sort to $\lambda$.
\end{cor}

However, finding a (nice) combinatorial description for the coefficients $d_{\alpha\beta}^\sigma(q)$ remains an open question, since the explicit expression for the coefficients $c_{\alpha\beta}^\sigma(q)$ is complicated.

  \appendix
  
  \section{Comparison to tableau formulas}\label{sec:comparison}

In the literature, the polynomials studied in this paper are modeled by tableau formulas, which are closely related to multiline queues. We briefly describe the connections to these formulas. 

\subsection{Semistandard augmented fillings}

In \cite{Mas09}, Mason 
defined the Demazure atoms as generating functions for \textbf{semistandard augmented fillings (SSAF)}, proving that $\cA(X)=E_\alpha(X;0,0)$. The closely related \textbf{semistandard composition tableaux (SSCT)} were introduced in \cite{HLMvW11} to model the quasisymmetric Schur polynomials. 

Each composition $\alpha$ determines a \textbf{composition diagram} $\dg(\alpha)$, which consists of $\alpha_i$ cells in the $i$th column, with columns arranged left to right. The cell in row $i$, column $j$ of $\dg(\alpha)$ is given coordinates $(i,j)$, and the entry in the cell $(i,j)$ in a filling $T$ of $\dg(\alpha)$ is denoted $T(i,j)$. If $(i,j)\not\in\dg(\alpha)$, then $T(i,j)=0$ by convention. The content of the filling $T$ is, as with semistandard Young tableaux, 
\[x^T=\prod_{u\in\dg(\gamma)}x_{T(u)}\,.\] 
Note that this convention is a $90^{\circ}$ rotation of that in \cite{HLMvW11}.

 \begin{definition}\label{def:SSAF}
     For a composition $\alpha$, a \textbf{semistandard augmented filling} $T$ of shape $\alpha$ is a filling of $\dg(\alpha)$ with positive integers with no repeated entries within rows, such that:
         \begin{enumerate}[label= (\roman*),ref=Item~(\roman*)]
             \item\label{item:i} Entries weakly decrease from bottom to top within columns
            \item\label{item:ii} If $\alpha_j\neq 0$, then $T(1,j)=j$.
             \item\label{item:iii} Given columns $j<k$ and a row $r>1$, if $T(r,k) \neq 0$ and $T(r,j) \leq T(r,k)$, then we must have $T(r-1,j) < T(r,k)$.
         \end{enumerate}
         Let $\SSAF(\alpha)$ be the set of all such fillings. 
 \end{definition}
  \noindent\ref{item:iii} corresponds to a certain triple condition: for each triple of the form $(r,j), (r,k), (r-1,j)$ with entries $T(r,j)=b, T(r,k)=a, T(r-1,j)=c$ as in the configuration
    \begin{align*}
        \tableau{
        b &&&& a \\
        c
        }\,,
    \end{align*}
the triple must satisfy the property that if $b \leq a$ then $c < a$. In fact, as $b\leq c$ by \ref{item:i}, this corresponds to the classical Haglund--Haiman--Loehr coinversion-free condition in \cite{HHL08}.
  \begin{remark}
     In the original definition, a ``zeroth'' row called the \emph{basement} filled with the entries $1,\ldots,n$ from left to right was added to the bottom of $\dg(\alpha)$. A filling satisfying \ref{item:iii} with the basement condition will equivalently satisfy  \ref{item:ii}. Varying the filling of the basement gives additional flexibility, such as being able to model key polynomials, but in the Demazure atom case the full generality isn't needed, so we omit the basement from our description.
 \end{remark}
 
For a strong composition $\gamma$, a semistandard composition tableau of shape $\gamma$ is defined very similarly, with \ref{item:ii} replaced by requiring that the entries strictly increase from left to right in the bottom row. Removing all empty columns from an SSAF gives the bijection 
$\SSCT(\gamma)\leftrightharpoons\bigcup_{\alpha:\alpha^+=\gamma}\SSAF(\alpha)$.


Then from \cite{Mas09} and from \cite{HLMvW11}, we obtain
\begin{equation}
    \cA_\alpha(X)=\sum_{T\in\SSAF(\alpha)}x^T,\qquad\qquad
         \QS_\gamma(X) = \sum_{T \in \SSCT(\gamma)} x^T.
     \end{equation}

     \newcommand{\Tab}{Tab}

The connection with nonwrapping multiline queues is simple, though its proof requires showing that Mason's RSK for composition shapes in \cite{Mas09} behaves in the same way as the collapsing map (sending the same entries to the same rows, and keeping the same shape). See \cite[Proposition~4.1.10]{N25} for details and \cref{ex:SSAF} for an example.
\begin{definition}
    Denote by $\Tab$ the union of the sets of fillings of $\dg(\alpha)$ over all compositions $\alpha$. Define the row content map
    \[\eta:\Tab\to\MLQ\]
    to be the map that sends a filling $T\in\Tab$ with row contents $(R_1,\ldots,R_L)$ to the multiline queue $M=(R_1,\ldots,R_L)\in\MLQ_\lambda$, where $\lambda'_i=|R_i|$ for $1\leq i\leq L$.
\end{definition}
For examples of $\eta$, see \cref{ex:SSAF} or \cref{ex:quinv tab}.

\begin{lemma}
   Let $\alpha$ be a composition and $\gamma$ a strong composition. The map $\eta$ gives the content-preserving bijections:
    \[\eta:\SSAF(\alpha)\leftrightharpoons \NMLQ[\alpha],\qquad\text{and}\qquad
   \eta:\SSCT(\gamma) \leftrightharpoons\SNMLQ[\gamma].
    \] 
\end{lemma}

\begin{remark}
    For the inverse map of $\rho$, for a multiline queue $M=(B_1,\ldots,B_L)$, there is a unique composition $\alpha$ such that an SSAF of shape $\alpha$ with this row content exists, and this $\alpha$ coincides with $\type(M)$. The corresponding SSAF may be obtained via Mason's RSK procedure in \cite{Mas09} applied to $\rev(\rw(M))$. More precisely, the link is that the insertion rule of Mason's RSK parallels the requirement in the FM procedure that a particle from row $i$ must pair with the closest unpaired particle in row $i-1$ weakly to its right.
\end{remark}

\begin{remark}
    Our choice in this article to focus on the multiline queue model for these polynomials is partly driven by the fact that from the crystal perspective, they are a more natural object than tableaux. As we have seen in \cref{sec:crystals}, they carry a direct crystal structure as they are not constrained by shape, whereas defining crystal operators on SSAFs or SSCTs involves additional subtleties to account for the shapes of the fillings, which may be changed by some of these operators as we have established in \cref{cor:Harper's lemma}.
\end{remark}

\begin{example}\label{ex:SSAF}
    We show the set of fillings $\SSAF(\alpha)$ and their images under $\eta$ in $\NMLQ(\alpha)$ for $\alpha=(1,0,3,2)$:
    
    \resizebox{\linewidth}{!}
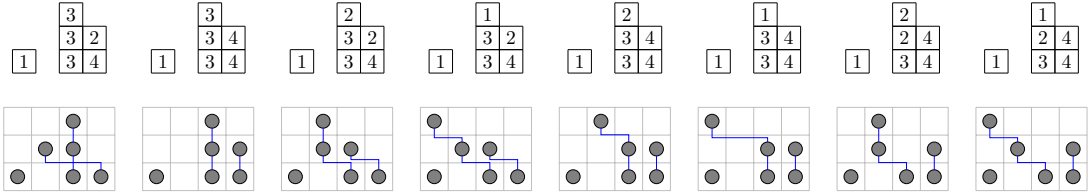
{\begin{tikzpicture}
        \begin{scope}[shift={(0,0)}]
        \node at (0,-2) {\scalebox{0.5}{\begin{tikzpicture}\queue{0\\3&1\\2&2&2}\end{tikzpicture}}};
        \node at (0,0) {\tableau{&&3\\&&3&2\\1&&3&4}};
        \end{scope}
        \begin{scope}[shift={(2.5,0)}]
        \node at (0,-2) {\scalebox{0.5}{\begin{tikzpicture}\queue{0\\3&3\\2&2&2}\end{tikzpicture}}};
        \node at (0,-0) {\tableau{&&3\\&&3&4\\1&&3&4}};
        \end{scope}
         \begin{scope}[shift={(5,0)}]
        \node at (0,-2) {\scalebox{0.5}{\begin{tikzpicture}\queue{0\\3&2\\2&1&1}\end{tikzpicture}}};
        \node at (0,0) {\tableau{&&2\\&&3&2\\1&&3&4}};
        \end{scope}
         \begin{scope}[shift={(7.5,0)}]
        \node at (0,-2) {\scalebox{0.5}{\begin{tikzpicture}\queue{0\\3&2\\2&1&0}\end{tikzpicture}}};
        \node at (0,0) {\tableau{&&1\\&&3&2\\1&&3&4}};
        \end{scope}
         \begin{scope}[shift={(10,0)}]
        \node at (0,-2) {\scalebox{0.5}{\begin{tikzpicture}\queue{0\\3&3\\2&2&1}\end{tikzpicture}}};
        \node at (0,0) {\tableau{&&2\\&&3&4\\1&&3&4}};
        \end{scope}
         \begin{scope}[shift={(12.5,0)}]
        \node at (0,-2) {\scalebox{0.5}{\begin{tikzpicture}\queue{0\\3&3\\2&2&0}\end{tikzpicture}}};
        \node at (0,0) {\tableau{&&1\\&&3&4\\1&&3&4}};
        \end{scope}
         \begin{scope}[shift={(15,0)}]
        \node at (0,-2) {\scalebox{0.5}{\begin{tikzpicture}\queue{0\\3&3\\2&1&1}\end{tikzpicture}}};
        \node at (0,0) {\tableau{&&2\\&&2&4\\1&&3&4}};
        \end{scope}
         \begin{scope}[shift={(17.5,0)}]
        \node at (0,-2) {\scalebox{0.5}{\begin{tikzpicture}\queue{0\\3&3\\2&1&0}\end{tikzpicture}}};
        \node at (0,0) {\tableau{&&1\\&&2&4\\1&&3&4}};
        \end{scope}
    \end{tikzpicture}
    }

\end{example}

\subsection{Non-attacking quinv tableaux}
For the $q$-generalization, there are several related tableau formulas. The one most closely connected to multiline queues, obtained as the $t=0$ specialization from a tableau formula proved in \cite{Man24}, uses the statistic of \textbf{queue inversions} (quinv), which are a different form of inversions that are directly tied to the pairing procedure on multiline queues. We call these objects semistandard non-attacking quinv tableaux (SSQT). 
\begin{definition}
 For a partition $\lambda$, a \textbf{semistandard non-attacking quinv tableau} is a filling of $\dg(\lambda)$ with positive integers such that rows contain no repeated entries, and the filling satisfies the ``maximal-quinv'' condition: for any triple of cells $x=(r,i)$, $y=(r-1,i)$, and $z=(r-1,j)$ with $i<j$ with entries $T(x)=a$, $T(y)=b$, and $T(z)=c$ as in the configuration
    \begin{align*}
        \tableau{
        a \\b&&& c \\
        }\,,
    \end{align*}
the entries $a,b,c$ satisfy
\[a\leq b<c\qquad \text{or}\qquad b<c<a\qquad\text{or}\qquad c<a\leq b\,.\] 
\end{definition}

The key property is that for fixed row content, there is a unique maximal-quinv filling of $\dg(\lambda)$, and hence a unique SSQT with that row content.  The $\maj$ statistic on SSQT is the classical Haglund--Haiman--Loehr major index: 
\[\maj(T)=\sum_{T(u)> T(\South(u))}(\leg(u)+1),\]
where if $u=(r,i)$, then $\South(u)=(r-1,i)$ and $\leg(u)=\lambda_i-r$. The \emph{type} of an SSQT is computed from the content of the bottom row: construct $\type(T)=(w_1,\ldots,w_n)$ by setting
\[
w_j=\begin{cases}
\lambda_i&T(1,i)=j\\
    0&j\not\in \{T(1,i):1\leq i\leq \ell(\lambda)\}\,.
\end{cases}
\]
In particular, the map $\eta$ is a weight-preserving bijection on $\SSQT$ that preserves type; see \cite[Section~3.2]{MV24} and \cite[Section~5]{Man24} for details and \cref{ex:quinv tab} for an example.

\begin{prop}[{\cite[Theorem~3.23]{MV24}}]
   Let $\alpha$ be a composition with $\lambda=\sort(\alpha)$. The map $\eta$ gives the weight-preserving bijection
    \[\eta:\{T\in\SSQT(\lambda):\type(T)=\alpha\}\leftrightharpoons \MLQ[\alpha]\,,
    \] 
with $\maj(M)=\maj(\eta(M))$ for all $M\in\MLQ[\alpha]$.
\end{prop}
\begin{remark}
    A similar statement holds for \emph{non-attacking coinversion-free fillings}, obtained by specializing the Haglund--Haiman--Loehr tableaux for Macdonald polynomials at $t=0$ \cite{HHL08}. There, the coinversion-free condition plays the same role as the maximal-quinv condition, analogously resulting in a unique filling for any given row content. In that case as well, the map to $\MLQ(\lambda)$ that preserves row content also preserves $\maj$.
\end{remark}

\begin{example}\label{ex:quinv tab}
For the multiline queue $M\in\MLQ[(4,4,3,2)]$ below, we show its image under $\eta^{-1}$ in $\SSQT((4,4,3,2))$, which is the unique maximal-quinv filling of $\dg(\lambda)$ with row content equal to that of $M$. Note that for each $k$, the columns of height $k$ at row $r$ in $\eta^{-1}(M)$ correspond to the set of particles with label $k$ in row $r$ of $M$, but the FM strands need not coincide with these columns. This holds for row $r=1$ as well, so in particular, $\type(M)=\type(\eta^{-1}(M))=(2,0,0,4,4,3)$. Moreover, the numbers of descents at each row and for each column height of $\eta^{-1}(M)$ map to the numbers of wrappings at each row of $M$ for the corresponding particle label, so $\maj(M)=\maj(\eta^{-1}(M))=3$.
 \begin{center}
\resizebox{!}{3cm}{
\begin{tikzpicture}[scale=0.7]
\def \w{1};
\def \h{1};
\def \r{0.25};

\draw[gray!50,thin,step=\w] (0,1) grid (6*\w,5*\h);
\foreach \xx\yy\i\c in {2/4/4/white,3/4/4/white,
0/3/3/white,3/3/4/white,5/3/4/white,
1/2/4/white,2/2/3/white,3/2/4/white, 4/2/2/white,
0/1/2/white,3/1/4/white,4/1/4/white,5/1/3/white}
    {
    \draw[fill=\c](\w*.5+\w*\xx,\h*.5+\h*\yy) circle (\r cm);
    \node at (\w*.5+\w*\xx,\h*.5+\h*\yy) {\i};
    }
    
\draw[black!50!green](\w*2.5,\h*4.5-\r)--(\w*2.5,\h*3.9)--(\w*3.5,\h*3.9)--(\w*3.5,\h*3.5+\r);
\draw[black!50!green](\w*3.5,\h*4.5-\r)--(\w*3.5,\h*4.1)--(\w*5.5,\h*4.1)--(\w*5.5,\h*3.5+\r);

\draw[black!50!green](\w*3.5,\h*3.5-\r)--(\w*3.5,\h*2.5+\r);
\draw[black!50!green,-stealth](\w*5.5,\h*3.5-\r)--(\w*5.5,\h*3)--(\w*6.3,\h*3);
\draw[black!50!green](-0.2,\h*2.9)--(\w*1.5,\h*2.9)--(\w*1.5,\h*2.5+\r);
\draw[black!50!green](\w*1.5,\h*2.5-\r)--(\w*1.5,\h*1.9)--(\w*3.5,\h*1.9)--(\w*3.5,\h*1.5+\r);
\draw[blue](0.5*\w, \h*3.5-\r)--(0.5\w, \h*3.1)--(2.5*\w, \h*3.1)--(2.5*\w,\h*2.5+\r);

\draw[black!50!green](\w*3.5,\h*2.5-\r)--(\w*3.5,\h*2.1)--(\w*4.5,\h*2.1)--(\w*4.5,\h*1.5+\r);
\draw[blue](\w*2.5,\h*2.5-\r)--(\w*2.5,\h*2)--(\w*5.5,\h*2)--(\w*5.5,\h*1.5+\r);
\draw[red,-stealth](4.5*\w, \h*2.5-\r)--(4.5\w, \h*2.2)--(\w*6.3,\h*2.2);
\draw[red](-0.2,\h*2.2)--(0.5*\w, \h*2.2)--(0.5*\w, \h*1.5+\r);

\draw[thick,->] (8,3)--(10,3) node[midway,above] {\large $\eta^{-1}$};
\node at (13,3) {\tableau{3&4\\4&6&1\\4&2&3&4\\4&5&6&1}};
\end{tikzpicture}
}
\end{center}
\end{example}

Finally, we draw some comparisons to work of Assaf and Gonz\'alez in \cite{AG19}, in which they proved that the polynomials $E_\alpha(X;q,0)$ expand positively in the key polynomials. Their model uses \emph{semistandard key tabloids} (corresponding to fillings of composition diagrams similar to SSAFs with a $\maj$ statistic, but with  suitable modifications to the conditions in \cref{def:SSAF}) to define $E_\alpha(X;q,0)$, with crystal operators acting on them directly. At $q=0$, the relevant objects are the \emph{semistandard key tableaux}, which generate key polynomials. Instead of the collapsing map of \cite{MV24}, they use a \emph{rectification operator}, which commutes with the crystal operators, and sends an arbitrary semistandard key tabloid to a semistandard key tableau. The subcrystal corresponding to a given key polynomial is connected, and they show that its preimage lies inside a single subcrystal for some $E_\alpha$ in the original crystal, with the same grading.

We observe, in particular, that rectification, collapsing, and RSK can all be seen as bijections from a graded set (key tabloids, multiline queues, tableaux) to the corresponding 0-graded subset, retaining the data of the original grading. In this sense, the overarching strategy of our proof parallels that of \cite{AG19}.

It would be interesting to find a robust multiline queue model for key polynomials, and more generally, for nonsymmetric Macdonald polynomials $E_\alpha(X;q,0)$, to recover the results of \cite{AG19} using our methods.

\printbibliography

@article{Ale19,
  title         = {Non-symmetric {M}acdonald polynomials and {D}emazure--Lusztig operators},
  author        = {Alexandersson, Per},
  year          = {2019},
  volume        = {76},
  issn          = {1286-4889},
  pages         = {Art. B76d, 27},
  fjournal  = {Séminaire Lotharingien de Combinatoire},
  journal  = {Sém. Lothar. Combin.},
  url           = {https://www.mat.univie.ac.at/~slc/wpapers/s76alexand.html},
}

@article{AGS20,
author = {Aas, Erik and Grinberg, Darij and Scrimshaw, Travis},
year = {2020},
month = {03},
pages = {1743-1786},
title = {Multiline Queues with Spectral Parameters},
volume = {374},
journal = {Comm. in Math. Phys.},
fjournal = {Communications in Mathematical Physics},
doi = {10.1007/s00220-020-03694-4}
}

@article{MS25,
  author       = {Mandelshtam, Olya and Scrimshaw, Travis},
  title        = {Twisted multiline queues for the steady states of TASEP and TAZRP},
  journal      = {Electron. J. Comb.},
volume= {30}, 
number={115}, 
pages={1-31},
  primaryClass = {math.CO},
  year         = {2025},
}

@article{AAMP11,
  TITLE = {{Recursive structures in the multispecies TASEP}},
  AUTHOR = {Arita, Chikashi and Ayyer, Arvind and Mallick, Kirone and Prolhac, Sylvain},
  URL = {https://hal.science/hal-00796165},
  JOURNAL = {Journal of Physics A: Mathematical and Theoretical},
  PUBLISHER = {{IOP Publishing}},
  VOLUME = {44},
  NUMBER = {33},
  PAGES = {335004},
  YEAR = {2011},
  DOI = {10.1088/1751-8113/44/33/335004},
  HAL_ID = {hal-00796165},
  HAL_VERSION = {v1},
}

@article{Mar20,
author = "Martin, J. B.",
doi = "10.1214/20-EJP421",
fjournal = "Electronic Journal of Probability",
journal = "Electron. J. Probab.",
pages = "41 pp.",
pno = "43",
publisher = "The Institute of Mathematical Statistics and the Bernoulli Society",
title = "Stationary distributions of the multi-type {A}{S}{E}{P}",
optURL = "https://doi.org/10.1214/20-EJP421",
volume = "25",
year = "2020"
}

@article{MV24,
title={Macdonald polynomials at $t = 0$ through (generalized) multiline queues},
author={Olya Mandelshtam and Jer\'onimo {Valencia Porras}},
journal={Advances in Applied Math},
note={to appear},
year={2024},
}

@misc{DM25+,
  author       = {Olya Mandelshtam and Zeus {Dantas e Moura}},
  title        = {Shape changing identities for permuted basement Macdonald polynomials},
  note={In preparation.},
  year         = {2025},
  primaryClass = {math.CO}
}

@thesis{Fer11,
  type        = {phdthesis},
  title       = {Row-strict quasisymmetric {S}chur functions, characterizations of {D}emazure atoms, and permuted basement nonsymmetric {M}acdonald polynomials},
  fpagetotal   = {90},
  finstitution = {University of California, Davis},
  flocation    = {California, United States},
  author      = {Ferreira, Jeffrey Paul},
  year        = {2011},
 furl         = {https://www.proquest.com/docview/940887941}
}

@article{CMW22,
  title         = {From multiline queues to {M}acdonald polynomials via the exclusion process},
  volume        = {144},
  issn          = {0002-9327},
  doi           = {10.1353/ajm.2022.0007},
  pages         = {395--436},
  number        = {2},
  fjournal  = {American Journal of Mathematics},
  journal  = {Am. J. Math.},
  author        = {Corteel, Sylvie and Mandelshtam, Olya and Williams, Lauren},
  year          = {2022},
}

@article{HHL08,
  title         = {A combinatorial formula for nonsymmetric {M}acdonald polynomials},
  volume        = {130},
  issn          = {1080-6377},
  url           = {https://muse.jhu.edu/pub/1/article/235639},
  pages         = {359--383},
  number        = {2},
  fjournal  = {American Journal of Mathematics},
  journal  = {Amer. J. Math.},
  author        = {Haglund, Jim and Haiman, Mark and Loehr, Nick},
  year          = {2008},
  publisher     = {Johns Hopkins University Press},
}

@article{CdGW15,
  author       = {Cantini, Luigi and de Gier, Jan and Wheeler, Michael},
  title        = {Matrix product formula for {M}acdonald polynomials},
  journal      = {J. Phys. A},
  fjournal     = {Journal of Physics. A. Mathematical and Theoretical},
  volume       = {48},
  number       = {38},
  pages        = {384001, 30},
  year         = {2015}
}

@article{CHMMW22,
  title        = {Compact formulas for {M}acdonald polynomials and quasisymmetric {M}acdonald polynomials},
  author       = {Corteel, Sylvie and Haglund, Jim and Mandelshtam, Olya and Mason, Sarah and Williams, Lauren},
  volume       = {28},
  issn         = {1420-9020},
  doi          = {10.1007/s00029-021-00721-7},
  pages        = {32},
  number       = {2},
  journaltitle = {Selecta Mathematica},
  shortjournal = {Sel. Math. New Ser.},
  year         = {2022}
}

@article{Mac88,
  author    = {Macdonald, Ian},
  title     = {A new class of symmetric functions},
  journal   = {S{\'e}m. Loth. Com.},
  fjournal  = {S{\'e}minaire Lotharingien de Combinatoire},
  volume    = {20},
pages = {B20a, 41 p.},
  year      = {1988},
}

@article{KMO15,
    AUTHOR = {Kuniba, Atsuo and Maruyama, Shouya and Okado, Masato},
     TITLE = {Multispecies {TASEP} and combinatorial {$R$}},
   JOURNAL = {J. Phys. A},
  FJOURNAL = {Journal of Physics. A. Mathematical and Theoretical},
    VOLUME = {48},
      YEAR = {2015},
    NUMBER = {34},
     PAGES = {34FT02, 19},
      ISSN = {1751-8113,1751-8121},
   MRCLASS = {81R12 (17B80 60K35)},
  fMRNUMBER = {3376035},
MRREVIEWER = {Julien\ Bichon},
       DOI = {10.1088/1751-8113/48/34/34FT02},
       URL = {https://doi.org/10.1088/1751-8113/48/34/34FT02},
}

@article{NY95,
    AUTHOR = {Nakayashiki, Atsushi and Yamada, Yasuhiko},
     TITLE = {Kostka polynomials and energy functions in solvable lattice
              models},
   JOURNAL = {Selecta Math. (N.S.)},
  FJOURNAL = {Selecta Mathematica. New Series},
    VOLUME = {3},
      YEAR = {1997},
    NUMBER = {4},
     PAGES = {547--599},
      ISSN = {1022-1824,1420-9020},
   MRCLASS = {05E05 (17B10 17B37 81R50 82B23)},
  fMRNUMBER = {1613527},
MRREVIEWER = {Kailash\ C.\ Misra},
       DOI = {10.1007/s000290050020},
       URL = {https://doi.org/10.1007/s000290050020},
}

@article{Man24,
  title   = {Probabilistic operators for non-attacking tableaux and a compact formula for the symmetric Macdonald polynomials},
  volume  = {13},
  doi     = {10.1017/fms.2025.10090},
  fjournal = {Forum of Mathematics, Sigma},
journal={Forum of Math., Sigma},
  author  = {Mandelshtam, Olya},
  year    = {2025},
  pages   = {e146}
}

@article{Che95,
	title = {Nonsymmetric {M}acdonald polynomials},
	volume = {1995},
	issn = {1073-7928},
	url = {https://doi.org/10.1155/S1073792895000341},
	doi = {10.1155/S1073792895000341},
	pages = {483--515},
	number = {10},
	journaltitle = {International Mathematics Research Notices},
	shortjournal = {International Mathematics Research Notices},
	author = {Cherednik, Ivan},
	urldate = {2024-10-10},
	date = {1995-01-01},
}

@article{Opd95,
	title = {Harmonic analysis for certain representations of graded {H}ecke algebras},
	volume = {175},
	issn = {0001-5962},
	url = {http://projecteuclid.org/euclid.acta/1485890867},
	doi = {10.1007/BF02392487},
	pages = {75--121},
	number = {1},
	journaltitle = {Acta Mathematica},
	shortjournal = {Acta Math.},
	author = {Opdam, Eric M.},
	urldate = {2025-01-20},
	date = {1995},
}

@article{Mac96,
     author = {Macdonald, Ian G.},
     title = {Affine {Hecke} algebras and orthogonal polynomials},
     booktitle = {S\'eminaire Bourbaki: volume 1994/95, expos\'es 790-804},
     series = {Ast\'erisque},
     note = {Talk no.~797},
     pages = {189--207},
     publisher = {Soci\'et\'e math\'ematique de France},
     number = {237},
     year = {1996},
     url = {http://www.numdam.org/item/SB_1994-1995__37__189_0/}
}

@article{HLMvW11,
	title = {Quasisymmetric {S}chur functions},
	volume = {118},
	issn = {0097-3165},
	doi = {10.1016/j.jcta.2009.11.002},
	pages = {463--490},
	number = {2},
	journaltitle = {Journal of Combinatorial Theory, Series A},
	shortjournal = {J. Comb. Theory Ser. A.},
	author = {Haglund, Jim and Luoto, Kurt and Mason, Sarah and van Willigenburg, Stephanie},
	date = {2011},
}

@article{Ass17,
  title={Nonsymmetric {M}acdonald polynomials and a refinement of {K}ostka–-{F}oulkes polynomials},
  author={Sami H. Assaf},
  fjournal={Transactions of the American Mathematical Society},
journal={Trans. Am. Math. Soc.},
   volume={370},
   ISSN={1088-6850},
   url={http://dx.doi.org/10.1090/tran/7374},
   DOI={10.1090/tran/7374},
   number={12},
  year={2017},
}

@article{AG19,
  title={Demazure crystals for specialized nonsymmetric {M}acdonald polynomials},
  author={Sami H. Assaf and Nicolle Gonz{\'a}lez},
  fjournal={Journal of Combinatorial Theory Series A},
journal={J. Comb. Theory A},
  year={2019},
  volume={182},
  pages={105463},
  url={https://api.semanticscholar.org/CorpusID:119160857}
}

@article{LS78,
   author={Alain Lascoux and Marcel-Paul Schützenberger},
   title={Sur une conjecture de H. O. Foulkes.},
   volume={286},
   ISSN={A323–A324},
   url={},
   DOI={},
   number={7},
   journal={C. R. Acad. Sci. Paris Sér. A-B},
   publisher={},
   year={1978},
   month=jul, pages={8777–8796} }

@article{FM07,
   title={Stationary distributions of multi-type totally asymmetric exclusion processes},
   volume={35},
   ISSN={0091-1798},
   url={http://dx.doi.org/10.1214/009117906000000944},
   DOI={10.1214/009117906000000944},
   number={3},
   journal={The Annals of Probability},
   publisher={Institute of Mathematical Statistics},
   author={Ferrari, Pablo A. and Martin, James B.},
   year={2007},
   month=may }

@incollection{LS90,
  author       = {Lascoux, Alain and Sch{\"u}tzenberger, Marcel-Paul},
  title        = {Keys and Standard Bases},
  booktitle    = {Invariant Theory and Tableaux},
  seditor       = {Stanton, Dennis},
  series       = {IMA Volumes in Mathematics and its Applications},
  volume       = {19},
  pages        = {125--144},
  year         = {1990},
  spublisher    = {Springer},
  saddress      = {New York, NY}
}

@book{Lothaire_2002, place={Cambridge}, series={Encyclopedia of Mathematics and its Applications}, title={Algebraic Combinatorics on Words}, publisher={Cambridge University Press}, author={Lothaire, M.}, year={2002}, collection={Encyclopedia of Mathematics and its Applications}}

@article{Mas09,
  author       = {Mason, Sarah},
  title        = {An explicit construction of type {A} Demazure atoms},
  journal      = {J. Algebraic Combin.},
  fjournal     = {Journal of Algebraic Combinatorics},
  volume       = {29},
  number       = {3},
  pages        = {295--313},
  year         = {2009}
}

@article{NY97,
  author       = {Nakayashiki, Atsushi and Yamada, Yasuhiko},
  title        = {Kostka polynomials and energy functions in solvable lattice models},
  journal      = {Selecta Math. (N.S.)},
  fjournal     = {Selecta Mathematica. New Series},
  volume       = {3},
  number       = {4},
  pages        = {547--599},
  year         = {1997}
}

@article{vanleeuwen2006double,
      title={Double crystals of binary and integral matrices}, 
      author={Marc A. A. van Leeuwen},
    journal={Electron. J. Combin.}, 
volume={13},
number={1},
      year={2006},
      note={Research Paper 86},
feprint={math/0605420},
      farchivePrefix={arXiv},
      fprimaryClass={math.CO}
}

@article{Kashiwara:1989md,
    author = "Kashiwara, Masaki",
    title = "{Crystalizing the $Q$ Analog of Universal Enveloping Algebras}",
    reportNumber = "RIMS-676",
    doi = "10.1007/BF02097367",
    journal = "Commun. Math. Phys.",
    volume = "133",
    pages = "249--260",
    year = "1990"
}

@article{N25,
  title={Reflected and nonsymmetric crystal graphs},
  author={Niergarth, Harper},
  year={2025},
  publisher={University of Waterloo}
}

\end{document}